\documentclass[11pt,twoside]{article}

\usepackage{float}

\newcommand{\HeadTitle}{Higher-order tangent and secant numbers, special-value integrals, and a common zero attractor}

\newcommand{\HeadTitleTwo}{\begin{center}
\Large{\textit{Higher-order tangent and secant numbers, special-value integrals, and a common zero attractor}}
\end{center}}

\usepackage[margin=0.8in]{geometry}
\usepackage{amsmath, amssymb, amsfonts, mathtools}
\usepackage{amsthm}
\usepackage{mathrsfs}
\usepackage{stmaryrd}
\usepackage{esint}
\usepackage{eqnalign}

\usepackage{graphicx}
\usepackage{enumitem}
\usepackage[most]{tcolorbox}
\usepackage{float}
\usepackage{listings}
\usepackage{array}
\usepackage{booktabs}
\usepackage{xcolor}
\usepackage{emptypage}
\usepackage{tensor}
\usepackage{tabularx}

\usepackage{fancyhdr}
\usepackage{hyperref}
\usepackage[nameinlink,noabbrev]{cleveref}

\usepackage{titlesec}
\usepackage[T1]{fontenc}
\usepackage[utf8]{inputenc}
\usepackage{lmodern}
\usepackage{titling}

\usepackage[backend=biber,style=numeric]{biblatex}
\providecommand{\HeadTitle}{} % <-- TITRE VARIABLE
\providecommand{\HeadTitleTwo}{} % <-- TITRE VARIABLE
\providecommand{\HeadAuthor}{Luc Ramsès TALLA WAFFO} % <-- FIXE

\titleformat{\section}[block]
  {\normalfont\large\bfseries\itshape\centering}
  {§\thesection.}
  {1em}
  {}

\titleformat{\subsection}
  {\normalfont\normalsize\bfseries}
  {\thesubsection}
  {1em}
  {}

\newtheorem{theorem}{Theorem}[section]
\newtheorem{lemma}[theorem]{Lemma}
\newtheorem{proposition}[theorem]{Proposition}
\newtheorem{corollary}[theorem]{Corollary}

\newtheorem{remark}[theorem]{Remark}
\newtheorem{definition}[theorem]{Definition}

\crefname{example}{example}{examples}
\Crefname{example}{Example}{Examples}

\crefname{corollary}{corollary}{corollaries}
\Crefname{corollary}{Corollary}{Corollaries}

\crefname{definition}{definition}{definitions}
\Crefname{definition}{Definition}{Definitions}

\crefname{remark}{remark}{remarks}
\Crefname{remark}{Remark}{Remarks}

\crefname{conjecture}{conjecture}{conjectures}
\Crefname{conjecture}{Conjecture}{Conjectures}

\crefname{lemma}{lemma}{lemmas}
\Crefname{lemma}{Lemma}{Lemmas}

\crefname{proposition}{proposition}{propositions}
\Crefname{proposition}{Proposition}{Propositions}

\crefname{theorem}{theorem}{theorems}
\Crefname{theorem}{Theorem}{Theorems}

\numberwithin{equation}{section}
\usepackage{fontspec}
\begin{document}

\thispagestyle{fancy}

\vspace{0.2cm}

\begin{center}
\Large{\HeadTitleTwo}
\end{center}

\hspace{3cm}

\begin{center}
Luc Ramsès TALLA WAFFO \\
Technische Universität Darmstadt\\
Karolinenplatz 5, 64289 Darmstadt, Germany\\
ramses.talla@stud.tu-darmstadt.de\\
\vspace{0.5cm}
\today
\end{center}

\begin{abstract}
We study two families of coefficients associated with powers of the tangent
function and with products of the secant and tangent functions, and develop
a unified treatment of the corresponding higher-order tangent and secant
numbers.  The first part establishes a common identity for even derivatives
of $\tanh x$ and odd derivatives of $\operatorname{sech}x$.  Combined with
integral formulae for these derivatives, this yields finite power-integral
representations for the normalized values $\zeta(2n+1)$ and $\beta(2n)$ and
naturally introduces two polynomial families $P_n$ and $Q_n$.

We determine several arithmetic and structural properties of these
polynomials.  In particular, we identify the exact content of $P_n$, prove
that $Q_n$ is primitive, establish strict log-concavity and unimodality of
their coefficient sequences, and describe their real zeros.  For even
indices each family has a unique negative real zero, and the two sequences
of real zeros have the same limit.  We also derive differential recurrences
and generating functions for $P_n$ and $Q_n$.

The generating functions lead to a global complex-singularity analysis.
We construct a compact Jordan curve $\Gamma$ and prove that the complex zero
sets of both $P_n$ and $Q_n$ converge to $\Gamma$ in Hausdorff distance.  The
normalized zero-counting measures converge weakly to the same equilibrium
measure on $\Gamma$.  We determine the logarithmic capacity of the enclosed
compact set, give an explicit arclength density for the limiting measure,
and compute its first moments.  Numerical plots for $P_{20}$, $P_{50}$, $P_{100}$ and
$P_{150}$ illustrate the zero-attractor theorem.
\end{abstract}

\vspace{0.2cm}

\paragraph{Notation.}
We collect here the notation used throughout the paper, with particular
attention to the set-theoretic and potential-theoretic notation in
\cref{sec:complex-zero-attractor}.

\emph{\underline{Basic notation.}}
We write
$\mathbb N=\{1,2,3,\ldots\}$,
$\mathbb Z$ for the integers, $\mathbb Q$ for the rational numbers,
$\mathbb R$ for the real numbers, and $\mathbb C$ for the complex plane;
$\mathbb Q^{\times}=\mathbb Q\setminus\{0\}$.
The Riemann sphere is
\[
\widehat{\mathbb C}:=\mathbb C\cup\{\infty\}.
\]
For $z\in\mathbb C$, $\Re z$, $\Im z$, $\overline z$, and $|z|$ denote,
respectively, its real part, imaginary part, complex conjugate, and modulus.
The symbols $\zeta$ and $\beta$ denote the Riemann zeta function and the
Dirichlet beta function.  We use the usual Landau symbols $O(\cdot)$ and
$o(\cdot)$; a subscript, as in $O_L(\cdot)$, indicates that the estimate is
uniform for the parameter in the set $L$.

\emph{\underline{Coefficient and arithmetic notation.}}
For a formal or convergent power series $f(x)$, the coefficient of $x^m$ is
written $[x^m]f(x)$; similarly, $[y^j]P(y)$ denotes the coefficient of
$y^j$ in a polynomial $P$.  The Bernoulli numbers are denoted by $B_m$, and
the Euler numbers $E_{2m}$ are normalized by
\[
\sec x=\sum_{m\geq0}(-1)^mE_{2m}\frac{x^{2m}}{(2m)!}.
\]
For $r\in\mathbb Q^{\times}$, $\nu_2(r)$ is the $2$-adic valuation,
normalized by $\nu_2(2)=1$.  For a nonnegative integer $m$, $s_2(m)$ denotes
the sum of the binary digits of $m$.  For a nonzero polynomial with integer
coefficients, $\operatorname{cont}(P)$ denotes the positive greatest common
divisor of its coefficients.  The notation $\lfloor t\rfloor$ has its usual
meaning.

\emph{\underline{Sets and topology.}}
For a subset $A\subseteq\mathbb C$, we write $A^\circ$ for its interior,
$\overline A$ for its closure, and $\partial A:=\overline A\setminus A^\circ$
for its boundary.  If $A,B$ are sets, $A\setminus B$ denotes set difference.
For $r>0$,
\[
\mathbb D_r:=\{z\in\mathbb C:|z|<r\},
\qquad
B(z_0,r):=\{z\in\mathbb C:|z-z_0|<r\}.
\]
If $U\subseteq\mathbb C$ is open, the notation $L\Subset U$ means that
$L$ is compactly contained in $U$, i.e. $\overline L$ is compact and
$\overline L\subset U$.  If a set $A$ contains the origin,
$\operatorname{Comp}_0 A$ denotes the connected component of $A$ containing
$0$.  For $z\in\mathbb C$ and a nonempty set $A\subseteq\mathbb C$,
\[
\operatorname{dist}(z,A):=\inf_{a\in A}|z-a|.
\]
For nonempty compact sets $A,B\subseteq\mathbb C$, their Hausdorff distance
is
\[
d_H(A,B)
:=\max\left\{
\sup_{a\in A}\operatorname{dist}(a,B),
\sup_{b\in B}\operatorname{dist}(b,A)
\right\}.
\]
For a sequence of sets $(A_n)$, the Kuratowski upper and lower limits are
understood in the standard sequential sense:
\[
\limsup_{n\to\infty}A_n
:=\left\{z:\exists\,n_j\to\infty,\ \exists\,z_j\in A_{n_j}
\text{ with }z_j\to z\right\},
\]
while
\[
\liminf_{n\to\infty}A_n
:=\left\{z:\text{every neighborhood of }z\text{ meets }A_n
\text{ for all sufficiently large }n\right\}.
\]
For a polynomial $p$, $Z(p)$ denotes its zero set without multiplicity.
When sums are written in the form $\sum_{p(\zeta)=0}$, zeros are counted
with multiplicity unless stated otherwise.

\emph{\underline{Measure and potential notation.}}
For $\zeta\in\mathbb C$, $\delta_\zeta$ denotes the unit point mass at
$\zeta$, and $\operatorname{supp}\mu$ denotes the closed support of a Borel
measure $\mu$.  The notation $\mu_n\xrightarrow{\mathrm w}\mu$ (or
$\mu_n\Rightarrow\mu$) means weak convergence of finite Borel measures.
The planar Laplacian is denoted by $\Delta$, and $|dy|$ denotes arclength on
a rectifiable curve.  On a regular boundary arc, $\partial_{\mathbf n}$
denotes differentiation in the normal direction pointing from the compact
side toward the exterior domain.  The logarithmic capacity of a compact set
$K$ is written $\operatorname{cap}(K)$.

\emph{\underline{The principal sets in \cref{sec:complex-zero-attractor}.}}
The complex-zero analysis uses the open central strip $\mathcal S$, the
principal critical domain $\Omega$, the exterior domain $E$, its compact
complement $K$, and the common zero-attractor curve $\Gamma$.  They are
defined precisely in \eqref{eq:complex-Omega-def} and
\eqref{eq:complex-E-def}--\eqref{eq:complex-K-Gamma-def}; in particular,
\[
\mathcal S=\{x\in\mathbb C:|\Re x|<\rho\},
\qquad
\Omega=\operatorname{Comp}_0\{x\in\mathcal S:|f(x)|<\rho\},
\]
\[
E=Y(\mathbb D_R),
\qquad
K=\widehat{\mathbb C}\setminus E,
\qquad
\Gamma=\partial K=\partial E.
\]
Thus $E$ is the unbounded component containing $\infty$, $K\subset\mathbb C$
is compact, $K^\circ$ denotes its interior, and $\Gamma$ is their common
boundary.  The constants $\rho=\pi/2$ and $R=\rho^2$ are fixed in
\eqref{eq:complex-rho-R-c}.  The dented domains
$\Delta(s_0;\eta,\theta)$ used in singularity analysis are defined in
\eqref{eq:complex-Delta-domain}.  Whenever fractional powers, square roots,
or $\arcsin$ occur in the exterior analysis, the relevant branches are the
ones specified locally in the text, ultimately obtained by analytic
continuation from the real branch $y>1$.

\vspace{0.5cm}

\section*{Introduction}
The purpose of this paper is to connect three themes that arise naturally
from higher-order tangent and secant numbers: derivative identities for
hyperbolic functions, integral representations of special values of zeta
and beta functions, and the arithmetic and zero geometry of the polynomial
families generated by those representations.

\vspace{0.2cm}

The starting point is a unified coefficient array $C_\alpha(m,k)$,
$\alpha\in\{0,1\}$, which simultaneously contains the higher-order tangent
and secant numbers.  The common recurrence for these coefficients is paired
with the differential recurrence satisfied by
$\operatorname{sech}^{\alpha}x\,\tanh^k x$.  This leads to the unified
identity \eqref{eq:unified-main-identity} and, in particular, to the two
finite derivative formulae \eqref{eq:main-tanh-identity} and
\eqref{eq:main-sech-identity}.  Integrating these identities against
$x^{-1}$ yields the power-integral representations in
\cref{prop:zeta-beta-power-integrals} for the normalized odd zeta values
$\zeta(2n+1)/\pi^{2n}$ and even beta values $\beta(2n)/\pi^{2n-1}$.

\vspace{0.2cm}

These finite sums naturally define the polynomial families $P_n$ and $Q_n$.
Their arithmetic behavior is notably different: the content of $P_n$ is an
explicit power of two, whereas $Q_n$ is primitive.  We also determine edge
coefficients, prove strict log-concavity and unimodality, derive differential
recurrences, and analyze the real zeros.  In particular, $P_{2n}$ and
$Q_{2n}$ each have exactly one real zero, both negative, while the odd-index
polynomials have no real zeros; the two even-index real zeros approach the
same limiting value.

\vspace{0.2cm}

The final part of the paper concerns the full complex zero sets.  The
generating functions in \cref{prop:PQ-generating-functions} lead to the same
critical-point geometry for both polynomial families.  We isolate a
principal critical domain, construct an exterior conformal map, and obtain a
compact Jordan curve $\Gamma$.  Uniform coefficient asymptotics on the two
sides of $\Gamma$, together with a separate control at infinity, imply the
main zero-attractor theorem: $Z(P_n)$ and $Z(Q_n)$ both converge to $\Gamma$
in Hausdorff distance.  The normalized zero-counting measures converge to a
common probability measure $\mu_\Gamma$, which is identified with the
equilibrium measure of the enclosed compact set.  Its logarithmic capacity,
arclength density, first moments, and the asymptotic ratio of the two monic
families are then made explicit.  The numerical experiments for
$n=20,50,100, 150$ provide a direct visualization of this limiting geometry.

\vspace{0.2cm}

The paper is organized accordingly.  Section~1 develops the unified
higher-derivative identities.  Sections~2 and~3 derive the special-value
integrals and introduce $P_n$ and $Q_n$.  Sections~4 and~5 study their
arithmetic, coefficient, recurrence, and real-zero properties.  Section~6
contains the global complex-zero analysis, the common zero-attractor theorem,
and the limiting zero distribution.

\section{Novel identities for the even-order derivatives of $\tanh x$ and the odd-order derivatives of $\operatorname{sech}x$}
\label[section]{sec:novel-identities}

The classical tangent and secant numbers arise as the exponential
Taylor coefficients of $\tan z$ and $\sec z$, respectively, and go back to
the work of Andr\'e; see \cite{andre1879}.  For classical treatments and
computational aspects, see also \cite{knuth-buckholtz1967}.  For explicit
formulas and recurrence relations for the classical tangent numbers, as well
as higher-order derivative formulae for the tangent function and related
trigonometric and hyperbolic functions, see also Qi \cite{Qi2015}.  Their
higher-order (or generalized) analogues occur in the work of Comtet and
Carlitz--Scoville and have subsequently been studied in a number of
contexts; see, for example,
\cite{Comtet1974,CarlitzScoville1972,carlitz1975,Cvijovic2009,cvijovic2011}.
A systematic treatment of the higher-order tangent and secant numbers,
including recurrence relations, explicit formulae, Bell-polynomial
representations, and a discussion of the different normalizations used in
the literature, is given by Cvijovi\'c \cite{cvijovic2011}.
\subsection{Normalization and unified coefficients}

Since two normalizations are commonly encountered, we make our convention
explicit.  In the normalized convention used by Comtet for the
higher-order tangent numbers and by Cvijovi\'c \cite{cvijovic2011}, one may
write
\begin{equation}
    \dfrac{\tan^k z}{k!}
    =
    \sum_{m=k}^{\infty}
    \mathcal{T}(m,k)\dfrac{z^m}{m!},
    \qquad
    \dfrac{\sec z\,\tan^k z}{k!}
    =
    \sum_{m=k}^{\infty}
    \mathcal{S}(m,k)\dfrac{z^m}{m!}.
    \label{eq:normalized-tangent-secant-numbers}
\end{equation}
Carlitz and Scoville instead use the corresponding \emph{unnormalized}
numbers, customarily denoted by $T_m^{(k)}$ and $S_m^{(k)}$, for which
\begin{equation}
    \tan^k z
    =
    \sum_{m=k}^{\infty}
    T_m^{(k)}\dfrac{z^m}{m!},
    \qquad
    \sec z\,\tan^k z
    =
    \sum_{m=k}^{\infty}
    S_m^{(k)}\dfrac{z^m}{m!}.
    \label{eq:unnormalized-tangent-secant-numbers}
\end{equation}
Thus
$T_m^{(k)}=k!\,\mathcal{T}(m,k)$ and
$S_m^{(k)}=k!\,\mathcal{S}(m,k)$; see
\cite{CarlitzScoville1972,cvijovic2011}.  We adopt this unnormalized
convention throughout, which is also the convention used in the derivative
formulae of Cvijovi\'c \cite{Cvijovic2009}.

The two families admit a single treatment.  We therefore use one auxiliary
parameter $\alpha\in\{0,1\}$ throughout.

\begin{definition}[Unified tangent--secant coefficients]
For $\alpha\in\{0,1\}$ and $k\geq 0$, define $C_\alpha(m,k)$ by
\begin{equation}
    \sec^\alpha z\,\tan^k z
    =
    \sum_{m=k}^{\infty}
    C_\alpha(m,k)\dfrac{z^m}{m!}.
    \label{eq:unified-coefficients}
\end{equation}
We use the convention $C_\alpha(m,k)=0$ if $k<0$ or $m<k$.  Thus
$T(m,k):=C_0(m,k)$ and $S(m,k):=C_1(m,k)$ are precisely the
unnormalized higher-order tangent and secant numbers, respectively; in the
notation above,
$T(m,k)=T_m^{(k)}$ and $S(m,k)=S_m^{(k)}$.  Equivalently,
$T(m,k)=k!\,\mathcal{T}(m,k)$ and
$S(m,k)=k!\,\mathcal{S}(m,k)$.
For $k=0$, our convention gives $T(0,0)=1$ and $T(m,0)=0$ for
$m\geq1$, while $S(m,0)$ are the classical secant numbers.  Moreover,
$T(m,1)$ are the classical tangent numbers.  By parity,
$T(m,k)=S(m,k)=0$ unless $m\equiv k\pmod 2$.  Finally, put
\begin{equation}
    \rho_\alpha(r):=1-\alpha+\alpha r.
    \label{eq:rho-def}
\end{equation}
Hence $\rho_0(r)=1$ and $\rho_1(r)=r$.
\end{definition}

\subsection{Related derivative formulae}

These coefficient arrays occur naturally in closed-form formulae for
higher derivatives.  For example, the formulae of Cvijovi\'c
\cite[Corollary~2]{Cvijovic2009}, specialized to the orders considered in
this section and written in the present normalization, give
\begin{align}
\dfrac{d^{2n}}{dx^{2n}}\tanh x
&=
\sum_{j=0}^{n}
\dfrac{(-1)^{n+j}}{2j+1}\,
T(2n+1,2j+1)\tanh^{2j+1}x,
\label{eq:literature-tanh-cvijovic}\\
\dfrac{d^{2n-1}}{dx^{2n-1}}\operatorname{sech}x
&=
\operatorname{sech}x
\sum_{j=0}^{n-1}
(-1)^{n+j}\,
S(2n-1,2j+1)\tanh^{2j+1}x.
\label{eq:literature-sech-cvijovic}
\end{align}
Franssens obtained a different explicit representation in terms of the
Eulerian and MacMahon number triangles.  Specializing his formulae to the
two orders considered here gives, for $n\geq 1$,
\begin{align}
\dfrac{d^{2n}}{dx^{2n}}\tanh x
&=
-\dfrac{1}{\cosh^{2n+1}x}
\sum_{k=0}^{2n-1}
(-1)^k A_{2n,k+1}
\sinh\bigl((2n-1-2k)x\bigr),
\label{eq:literature-tanh-franssens}\\
\dfrac{d^{2n-1}}{dx^{2n-1}}\operatorname{sech}x
&=
-\dfrac{1}{2^{2n-1}\cosh^{2n}x}
\sum_{k=0}^{2n-1}
(-1)^k B_{2n,k+1}
\sinh\bigl((2n-1-2k)x\bigr),
\label{eq:literature-sech-franssens}
\end{align}
where \(A_{m,k}\) and \(B_{m,k}\) denote, respectively, the ordinary Eulerian numbers and the MacMahon numbers (equivalently, the Eulerian numbers of type~\(B\)) in the notation of Franssens; see \cite[Sec.~5.3.1]{Franssens2007}.  Thus, these formulae express the
relevant derivatives as rational functions of hyperbolic functions, with
coefficients from the Eulerian and MacMahon triangles.  Related
derivative-polynomial representations were obtained by Hoffman and
Boyadzhiev \cite{Hoffman1995,boyadzhiev2007}.

More recently, Kyrion obtained alternative formulae precisely for even
derivatives of $\tanh x$ and odd derivatives of $\operatorname{sech}x$;
in the present notation, his formulae include
\begin{align}
\dfrac{d^{2n}}{dx^{2n}}\tanh x
&=
\tanh x\sum_{k=1}^{n}
\left(
\dfrac{2}{4^k}\sum_{j=1}^{k}
\binom{2k}{k-j}(-1)^j(2j)^{2n}
\right)
\operatorname{sech}^{2k}x,
\label{eq:literature-tanh-kyrion}\\
\dfrac{d^{2n-1}}{dx^{2n-1}}\operatorname{sech}x
&=
\tanh x\sum_{k=0}^{n-1}
\left(
\dfrac{1}{4^k}\sum_{j=0}^{k}
\binom{2k+1}{k-j}(-1)^{j+1}(2j+1)^{2n-1}
\right)
\operatorname{sech}^{2k+1}x.
\label{eq:literature-sech-kyrion}
\end{align}
See \cite[Lemma~3.3 and Proposition~4.3]{Kyrion2025}.  The identities
proved below have a different form: instead of expanding the derivatives
only as polynomials in $\tanh x$ and $\operatorname{sech}x$, they express
them as finite linear combinations of derivatives of powers of $\tanh x$
and of products $\operatorname{sech}x\,\tanh^p x$, with coefficients
given by the same higher-order tangent and secant numbers.

\subsection{Coefficient and differential recurrences}

\begin{lemma}
For $m,k\geq 0$ and $\alpha\in\{0,1\}$,
\begin{equation}
    C_\alpha(m+1,k)
    =
    k\,C_\alpha(m,k-1)
    +(k+\alpha)\,C_\alpha(m,k+1).
    \label{eq:unified-coefficient-recurrence}
\end{equation}
\end{lemma}

\begin{proof}
Differentiating \eqref{eq:unified-coefficients} and using
$\sec^2z=1+\tan^2z$ gives
\[
\dfrac{d}{dz}\bigl(\sec^\alpha z\tan^kz\bigr)
=
k\sec^\alpha z\tan^{k-1}z
+(k+\alpha)\sec^\alpha z\tan^{k+1}z.
\]
Comparison of exponential generating-function coefficients proves
\eqref{eq:unified-coefficient-recurrence}.
\end{proof}

\medskip
For the hyperbolic side, define
\begin{equation}
    F_{\alpha,k}(x)
    :=
    \operatorname{sech}^{\alpha}x\,\tanh^k x,
    \qquad k\geq 0.
    \label{eq:F-def}
\end{equation}
Since $\dfrac{d}{dx}\tanh x=1-\tanh^2x$ and
$\dfrac{d}{dx}\operatorname{sech}x=-\operatorname{sech}x\tanh x$, we have the
companion relation
\begin{equation}
    \dfrac{d}{dx}F_{\alpha,k}
    =
    kF_{\alpha,k-1}
    -(k+\alpha)F_{\alpha,k+1}.
    \label{eq:F-recurrence}
\end{equation}
For $p\geq 1$, set
\begin{equation}
    G_{\alpha,p}
    :=
    \dfrac{d^{p-\alpha}}{dx^{p-\alpha}}F_{\alpha,p+1-\alpha}.
    \label{eq:G-def}
\end{equation}
Thus $G_{0,p}=\dfrac{d^p}{dx^p}\!\left(\tanh^{p+1}x\right)$ and
$G_{1,p}=\dfrac{d^{p-1}}{dx^{p-1}}\!\left(\operatorname{sech}x\,\tanh^p x\right)$.

\begin{lemma}
For every $p\geq 1$ and $\alpha\in\{0,1\}$,
\begin{equation}
    \dfrac{d^2}{dx^2}G_{\alpha,p}
    =
    -(p+1-\alpha)!
    \sum_{q=1}^{p+1}
    \dfrac{\rho_\alpha(q)}{(q-1)!}\,
    G_{\alpha,q}.
    \label{eq:unified-G-second-derivative}
\end{equation}
\end{lemma}

\begin{proof}
For $p=1$, direct differentiation gives
\[
\dfrac{d^2}{dx^2}G_{0,1}=-2\bigl(G_{0,1}+G_{0,2}\bigr),
\qquad
\dfrac{d^2}{dx^2}G_{1,1}=-G_{1,1}-2G_{1,2},
\]
which is exactly \eqref{eq:unified-G-second-derivative}.

For $p\geq 2$, \eqref{eq:F-recurrence} yields
\begin{align}
\dfrac{d^2}{dx^2}G_{\alpha,p}
&=
\dfrac{d^{p+2-\alpha}}{dx^{p+2-\alpha}}F_{\alpha,p+1-\alpha}
\nonumber\\
&=
(p+1-\alpha)\dfrac{d^2}{dx^2}G_{\alpha,p-1}
-(p+1)G_{\alpha,p+1}.
\label{eq:G-recursion}
\end{align}
Assuming \eqref{eq:unified-G-second-derivative} for $p-1$ and using
$(p+1-\alpha)(p-\alpha)!=(p+1-\alpha)!$ and
$(p+1)=(p+1-\alpha)!\,\rho_\alpha(p+1)/p!$, we obtain
\eqref{eq:unified-G-second-derivative} for $p$.
\end{proof}

\subsection{The unified main identity}

\begin{theorem}[Unified tangent--secant identity]
Let $n\geq 1$ and $\alpha\in\{0,1\}$. Then
\begin{equation}
\boxed{
\dfrac{d^{2n-\alpha}}{dx^{2n-\alpha}}F_{\alpha,1-\alpha}
=
(-1)^n
\sum_{p=1}^{n}
\rho_\alpha(p)\,
C_\alpha(2n-p-\alpha,p-\alpha)\,
G_{\alpha,p}.
}
\label{eq:unified-main-identity}
\end{equation}
\end{theorem}

\begin{proof}
We proceed by induction on $n$.  Since
$C_\alpha(1-\alpha,1-\alpha)=1$, the case $n=1$ follows from $\dfrac{d^2}{dx^2}\tanh x=-\dfrac{d}{dx}(\tanh^2x)$ and
$\dfrac{d}{dx}\operatorname{sech}x=-\operatorname{sech}x\tanh x$.

Assume \eqref{eq:unified-main-identity} for some $n\geq 1$.
Applying $\dfrac{d^2}{dx^2}$ and using
\eqref{eq:unified-G-second-derivative}, we get
\begin{align}
\dfrac{d^{2n+2-\alpha}}{dx^{2n+2-\alpha}}F_{\alpha,1-\alpha}
&=
(-1)^{n+1}
\sum_{q=1}^{n+1}
\dfrac{\rho_\alpha(q)}{(q-1)!}G_{\alpha,q}
\nonumber\\
&\quad\times
\sum_{p=q-1}^{n}
\rho_\alpha(p)(p+1-\alpha)!
C_\alpha(2n-p-\alpha,p-\alpha).
\label{eq:unified-double-sum}
\end{align}
For $q=1$, the formally occurring $p=0$ term vanishes: if
$\alpha=0$, then $C_0(2n,0)=0$, while if $\alpha=1$, then
$\rho_1(0)=0$.

To evaluate the inner sum, define
$X_p:=p!\,C_\alpha(2n-p+1-\alpha,p+1-\alpha)$. Using
\eqref{eq:unified-coefficient-recurrence} with $m=2n-p-\alpha$ and
$k=p+1-\alpha$, we obtain
\begin{equation}
X_p-X_{p+1}
=
\rho_\alpha(p)(p+1-\alpha)!
C_\alpha(2n-p-\alpha,p-\alpha).
\label{eq:unified-telescoping}
\end{equation}
Hence
\begin{align*}
&\sum_{p=q-1}^{n}
\rho_\alpha(p)(p+1-\alpha)!
C_\alpha(2n-p-\alpha,p-\alpha)
\\
&\qquad
=X_{q-1}-X_{n+1}
\\
&\qquad
=(q-1)!
C_\alpha(2n-q+2-\alpha,q-\alpha),
\end{align*}
because $X_{n+1}=(n+1)!\,C_\alpha(n-\alpha,n+2-\alpha)=0$.
Substituting this into \eqref{eq:unified-double-sum} gives
\[
\dfrac{d^{2n+2-\alpha}}{dx^{2n+2-\alpha}}F_{\alpha,1-\alpha}
=
(-1)^{n+1}
\sum_{q=1}^{n+1}
\rho_\alpha(q)
C_\alpha(2n-q+2-\alpha,q-\alpha)
G_{\alpha,q},
\]
which is \eqref{eq:unified-main-identity} with $n$ replaced by
$n+1$.
\end{proof}

\begin{corollary}
For every $n\geq 1$,
\begin{equation}
\boxed{
\dfrac{d^{2n}}{dx^{2n}}\tanh x
=
(-1)^n
\sum_{p=1}^{n}
T(2n-p,p)\,
\dfrac{d^p}{dx^p}\!\left(\tanh^{p+1}x\right)
}
\label{eq:main-tanh-identity}
\end{equation}
and
\begin{equation}
\boxed{
\dfrac{d^{2n-1}}{dx^{2n-1}}\operatorname{sech}x
=
(-1)^n
\sum_{p=1}^{n}
p\,S(2n-p-1,p-1)\,
\dfrac{d^{p-1}}{dx^{p-1}}\!\left(\operatorname{sech}x\,\tanh^p x\right).
}
\label{eq:main-sech-identity}
\end{equation}
\end{corollary}

\begin{proof}
Take $\alpha=0$ and $\alpha=1$, respectively, in
\eqref{eq:unified-main-identity}.
\end{proof}

\section{Integral representations for odd zeta and even beta values}
\subsection{Derivative integrals for hyperbolic functions}

\begin{proposition}
\label[proposition]{prop:derivative-integrals}
For every $n\in\mathbb{N}$,
\begin{equation}
\begin{dcases}
\displaystyle
\int_{0}^{\infty}
\frac{1}{x}
\frac{d^{\,2n-1}}{dx^{\,2n-1}}
\operatorname{sech}x\,dx
=
(-1)^n
2^{2n}(2n-1)!
\frac{\beta(2n)}{\pi^{2n-1}},
\\[12pt]
\displaystyle
\int_{0}^{\infty}
\frac{1}{x}
\frac{d^{\,2n}}{dx^{\,2n}}
\tanh x\,dx
=
(-1)^n
\bigl(2^{2n+1}-1\bigr)(2n)!
\frac{\zeta(2n+1)}{\pi^{2n}}.
\end{dcases}
\label{eq:sech-tanh-derivative-integrals}
\end{equation}
\end{proposition}

\begin{proof}
We start from the finite hyperbolic representations obtained in
\cite[Eq.~(6.11)]{TallaWaffo2026}:
\begin{equation}
\begin{dcases}
\displaystyle
\int_{0}^{1}
\frac{\Im\!\left(\operatorname{Li}_{-2n}(ix)\right)}{x}
\ln\!\ln\frac{1}{x}\,dx
=
\frac{1}{2^{2n-1}}
\sum_{k=0}^{n-1}
(-1)^k
\left\langle{2n-1\atop k}\right\rangle^{B}
\int_{0}^{\infty}
\frac{\sinh((2n-2k-1)t)}
{t\cosh^{2n}t}\,dt,
\\[17pt]
\displaystyle
\int_{0}^{1}
\frac{\operatorname{Li}_{-2n-1}(-x^2)}{x}
\ln\!\ln\frac{1}{x}\,dx
=
-\frac{1}{2^{2n+1}}
\sum_{k=0}^{n-1}
(-1)^k
\left\langle{2n\atop k}\right\rangle
\int_{0}^{\infty}
\frac{\sinh((2n-2k-1)t)}
{t\cosh^{2n+1}t}\,dt.
\end{dcases}
\label{eq:polylog-hyperbolic-representations}
\end{equation}

On the other hand, the derivative formulae of Franssens
\cite[eqs.~(5.24) and~(5.27)]{Franssens2007}, specialized to the
orders $2n-1$ and $2n$, respectively, yield after pairing the
symmetric terms
\begin{equation}
\begin{dcases}
\displaystyle
\sum_{k=0}^{n-1}
(-1)^k
\left\langle{2n-1\atop k}\right\rangle^{B}
\sinh((2n-2k-1)t)
=
-2^{2n-2}\cosh^{2n}t\,
\frac{d^{\,2n-1}}{dt^{\,2n-1}}
\operatorname{sech}t,
\\[14pt]
\displaystyle
\sum_{k=0}^{n-1}
(-1)^k
\left\langle{2n\atop k}\right\rangle
\sinh((2n-2k-1)t)
=
-\frac{1}{2}\cosh^{2n+1}t\,
\frac{d^{\,2n}}{dt^{\,2n}}
\tanh t.
\end{dcases}
\label{eq:franssens-paired-derivative-sums}
\end{equation}

Since the sums in
\eqref{eq:polylog-hyperbolic-representations} are finite, we may
insert \eqref{eq:franssens-paired-derivative-sums} directly into
the corresponding integrals.  The powers of $\cosh t$ cancel,
and we obtain
\begin{equation}
\begin{dcases}
\displaystyle
\int_{0}^{1}
\frac{\Im\!\left(\operatorname{Li}_{-2n}(ix)\right)}{x}
\ln\!\ln\frac{1}{x}\,dx
=
-\frac{1}{2}
\int_{0}^{\infty}
\frac{1}{t}
\frac{d^{\,2n-1}}{dt^{\,2n-1}}
\operatorname{sech}t\,dt,
\\[17pt]
\displaystyle
\int_{0}^{1}
\frac{\operatorname{Li}_{-2n-1}(-x^2)}{x}
\ln\!\ln\frac{1}{x}\,dx
=
\frac{1}{2^{2n+2}}
\int_{0}^{\infty}
\frac{1}{t}
\frac{d^{\,2n}}{dt^{\,2n}}
\tanh t\,dt.
\end{dcases}
\label{eq:polylog-derivative-integrals}
\end{equation}

Finally, by \cite[Eq.~(6.9)]{TallaWaffo2026}\cite{TallaWaffo2025arxiv2511.02843},
\begin{equation}
\begin{dcases}
\displaystyle
\int_{0}^{1}
\frac{\Im\!\left(\operatorname{Li}_{-2n}(ix)\right)}{x}
\ln\!\ln\frac{1}{x}\,dx
=
(-1)^{n-1}
2^{2n-1}(2n-1)!
\frac{\beta(2n)}{\pi^{2n-1}},
\\[17pt]
\displaystyle
\int_{0}^{1}
\frac{\operatorname{Li}_{-2n-1}(-x^2)}{x}
\ln\!\ln\frac{1}{x}\,dx
=
(-1)^n
\frac{\bigl(2^{2n+1}-1\bigr)(2n)!}{2^{2n+2}}
\frac{\zeta(2n+1)}{\pi^{2n}}.
\end{dcases}
\label{eq:polylog-evaluations-for-derivatives}
\end{equation}

Comparing the first lines of
\eqref{eq:polylog-derivative-integrals} and
\eqref{eq:polylog-evaluations-for-derivatives} and simplifying algebraically give \eqref{eq:sech-tanh-derivative-integrals}.
\end{proof}

\begin{remark}
Strictly speaking, \cref{prop:derivative-integrals} may be found in the second version of this preprint \textnormal{\cite{Kyrion2025}} by Tobias Kyrion .
\end{remark}

\subsection{Power-integral representations}

\begin{proposition}
\label[proposition]{prop:zeta-beta-power-integrals}
For every $n\geq 1$,
\begin{equation}
\begin{dcases}
\displaystyle
\frac{\zeta(2n+1)}{\pi^{2n}}
=
\sum_{p=1}^{n}
\frac{p!\,T(2n-p,p)}
{\bigl(2^{2n+1}-1\bigr)(2n)!}
\int_{0}^{\infty}
\frac{\tanh^{p+1}x}{x^{p+1}}\,dx,
\\[16pt]
\displaystyle
\frac{\beta(2n)}{\pi^{2n-1}}
=
\sum_{p=1}^{n}
\frac{p!\,S(2n-p-1,p-1)}
{2^{2n}(2n-1)!}
\int_{0}^{\infty}
\frac{\operatorname{sech}x\,\tanh^{p}x}{x^{p}}\,dx.
\end{dcases}
\label{eq:zeta-beta-power-integrals}
\end{equation}
\end{proposition}

\begin{proof}
Multiplying \eqref{eq:main-tanh-identity} and
\eqref{eq:main-sech-identity} by $x^{-1}$ and integrating over
$(0,\infty)$ gives
\begin{equation}
\begin{dcases}
\displaystyle
\int_{0}^{\infty}
\frac{1}{x}
\frac{d^{2n}}{dx^{2n}}\tanh x\,dx
=
(-1)^n
\sum_{p=1}^{n}
T(2n-p,p)
\int_{0}^{\infty}
\frac{1}{x}
\frac{d^{p}}{dx^{p}}
\bigl(\tanh^{p+1}x\bigr)\,dx,
\\[17pt]
\displaystyle
\int_{0}^{\infty}
\frac{1}{x}
\frac{d^{2n-1}}{dx^{2n-1}}
\operatorname{sech}x\,dx
=
(-1)^n
\sum_{p=1}^{n}
p\,S(2n-p-1,p-1)
\int_{0}^{\infty}
\frac{1}{x}
\frac{d^{p-1}}{dx^{p-1}}
\bigl(\operatorname{sech}x\,\tanh^{p}x\bigr)\,dx.
\end{dcases}
\label{eq:integrated-compact-identities}
\end{equation}

For $0\leq j\leq p$, as $x\to0^+$,
\(\frac{d^j}{dx^j}\bigl(\tanh^{p+1}x\bigr)=O\bigl(x^{p+1-j}\bigr)\),
whereas, for $0\leq j\leq p-1$,
\(\frac{d^j}{dx^j}\bigl(\operatorname{sech}x\,\tanh^{p}x\bigr)=O\bigl(x^{p-j}\bigr)\).
At infinity, all positive-order derivatives involved above decay
exponentially; moreover, $\tanh^{p+1}x$ is bounded and
$\operatorname{sech}x\,\tanh^p x$ decays exponentially.  Hence all
boundary terms arising below vanish.

Repeated integration by parts therefore yields
\begin{equation}
\begin{dcases}
\displaystyle
\int_{0}^{\infty}
\frac{1}{x}
\frac{d^{p}}{dx^{p}}
\bigl(\tanh^{p+1}x\bigr)\,dx
=
p!
\int_{0}^{\infty}
\frac{\tanh^{p+1}x}{x^{p+1}}\,dx,
\\[17pt]
\displaystyle
\int_{0}^{\infty}
\frac{1}{x}
\frac{d^{p-1}}{dx^{p-1}}
\bigl(\operatorname{sech}x\,\tanh^{p}x\bigr)\,dx
=
(p-1)!
\int_{0}^{\infty}
\frac{\operatorname{sech}x\,\tanh^{p}x}{x^{p}}\,dx.
\end{dcases}
\label{eq:repeated-integration-by-parts}
\end{equation}

Substituting \eqref{eq:repeated-integration-by-parts} into
\eqref{eq:integrated-compact-identities}, and using
$p(p-1)!=p!$, gives
\begin{equation}
\begin{dcases}
\displaystyle
\int_{0}^{\infty}
\frac{1}{x}
\frac{d^{2n}}{dx^{2n}}\tanh x\,dx
=
(-1)^n
\sum_{p=1}^{n}
p!\,T(2n-p,p)
\int_{0}^{\infty}
\frac{\tanh^{p+1}x}{x^{p+1}}\,dx,
\\[17pt]
\displaystyle
\int_{0}^{\infty}
\frac{1}{x}
\frac{d^{2n-1}}{dx^{2n-1}}
\operatorname{sech}x\,dx
=
(-1)^n
\sum_{p=1}^{n}
p!\,S(2n-p-1,p-1)
\int_{0}^{\infty}
\frac{\operatorname{sech}x\,\tanh^{p}x}{x^{p}}\,dx.
\end{dcases}
\label{eq:derivative-integrals-power-form}
\end{equation}

On the other hand, \cref{prop:derivative-integrals} gives
\begin{equation}
\begin{dcases}
\displaystyle
\int_{0}^{\infty}
\frac{1}{x}
\frac{d^{2n}}{dx^{2n}}\tanh x\,dx
=
(-1)^n
\bigl(2^{2n+1}-1\bigr)(2n)!
\frac{\zeta(2n+1)}{\pi^{2n}},
\\[17pt]
\displaystyle
\int_{0}^{\infty}
\frac{1}{x}
\frac{d^{2n-1}}{dx^{2n-1}}
\operatorname{sech}x\,dx
=
(-1)^n
2^{2n}(2n-1)!
\frac{\beta(2n)}{\pi^{2n-1}}.
\end{dcases}
\label{eq:derivative-integrals-evaluated}
\end{equation}

Comparing the corresponding lines of
\eqref{eq:derivative-integrals-power-form} and
\eqref{eq:derivative-integrals-evaluated}, cancelling $(-1)^n$,
and dividing by the respective prefactors gives
\eqref{eq:zeta-beta-power-integrals}.
\end{proof}

\section{The polynomial families $P_n$ and $Q_n$}
\subsection{Definition and integral representation}

The finite sums in \cref{prop:zeta-beta-power-integrals} naturally give rise to two polynomial families.  For every $n\geq1$, define
\[
P_n(y):=\sum_{p=1}^{n}p!\,T(2n-p,p)y^{p-1},
\qquad
Q_n(y):=\sum_{p=1}^{n}p!\,S(2n-p-1,p-1)y^{p-1}.
\]
For example,
\[
\begin{aligned}
P_1(y)&=1,
& P_2(y)&=2+4y,
& P_3(y)&=16+32y+36y^2,\\
Q_1(y)&=1,
& Q_2(y)&=1+2y,
& Q_3(y)&=5+10y+12y^2.
\end{aligned}
\]
We shall return later to the structural and coefficient properties of these
polynomial families.  With this notation, \cref{eq:zeta-beta-power-integrals}
takes the form
\begin{equation}
\begin{dcases}
\displaystyle
\frac{\zeta(2n+1)}{\pi^{2n}}
=
\frac{1}{\bigl(2^{2n+1}-1\bigr)(2n)!}
\int_{0}^{\infty}
\left(\frac{\tanh x}{x}\right)^2
P_n\!\left(\frac{\tanh x}{x}\right)\,dx,
\\[16pt]
\displaystyle
\frac{\beta(2n)}{\pi^{2n-1}}
=
\frac{1}{2^{2n}(2n-1)!}
\int_{0}^{\infty}
\operatorname{sech}x\,\frac{\tanh x}{x}
Q_n\!\left(\frac{\tanh x}{x}\right)\,dx.
\end{dcases}
\label{eq:zeta-beta-polynomial-integrals}
\end{equation}
Thus \cref{prop:zeta-beta-power-integrals} exhibits the normalized odd zeta
and even beta values through the two polynomial families $P_n$ and $Q_n$.

\section{Arithmetic and coefficient properties of $P_n$ and $Q_n$}

\subsection{Edge coefficients}

\begin{proposition}[The first and last coefficients of \(P_n\) and \(Q_n\)]
\label[proposition]{prop:PQ-edge-coefficients}
For every \(n\geq2\), the first two and the last two coefficients of
\(P_n\) are
\begin{equation}
\begin{aligned}
P_n(0)
&=T(2n-1,1)
=(-1)^{n-1}\frac{2^{2n}(2^{2n}-1)}{2n}\,B_{2n},\\
[y]P_n(y)
&=2P_n(0),\\
[y^{n-2}]P_n(y)
&=\frac{n!(n-1)!(n^2-1)}{3},\\
[y^{n-1}]P_n(y)
&=(n!)^2,
\end{aligned}
\label{eq:P-edge-coefficients}
\end{equation}
and those of \(Q_n\) are
\begin{equation}
\begin{aligned}
Q_n(0)
&=S(2n-2,0)
=(-1)^{n-1}E_{2n-2},\\
[y]Q_n(y)
&=2Q_n(0),\\
[y^{n-2}]Q_n(y)
&=\frac{n!(n-1)!(2n-1)}{6},\\
[y^{n-1}]Q_n(y)
&=n!(n-1)!.
\end{aligned}
\label{eq:Q-edge-coefficients}
\end{equation}
Here \(B_{2n}\) and \(E_{2n}\) denote the Bernoulli and Euler numbers,
respectively, with
\[
\sec x=\sum_{m\geq0}(-1)^mE_{2m}\frac{x^{2m}}{(2m)!}.
\]
In particular, all these edge coefficients are positive integers.
For \(n=1\), one simply has \(P_1(y)=Q_1(y)=1\).
\end{proposition}

\begin{proof}
The constant terms follow immediately from the definitions:
\(P_n(0)=T(2n-1,1)\) and \(Q_n(0)=S(2n-2,0)\).
Using the coefficient recurrence with \((m,k)=(2n-2,1)\) and
\((m,k)=(2n-3,0)\), respectively, gives
\(T(2n-1,1)=T(2n-2,2)\) and
\(S(2n-2,0)=S(2n-3,1)\). Hence
\([y]P_n(y)=2!\,T(2n-2,2)=2P_n(0)\) and
\([y]Q_n(y)=2!\,S(2n-3,1)=2Q_n(0)\).
The classical expansions
\[
\tan x
=\sum_{m\geq1}(-1)^{m-1}
\frac{2^{2m}(2^{2m}-1)B_{2m}}{(2m)!}\,x^{2m-1},
\qquad
\sec x
=\sum_{m\geq0}(-1)^mE_{2m}\frac{x^{2m}}{(2m)!},
\]
then yield the stated Bernoulli and Euler forms of \(P_n(0)\) and
\(Q_n(0)\).

For the last two coefficients, use
\[
\tan x=x+\frac{x^3}{3}+O(x^5),
\qquad
\sec x=1+\frac{x^2}{2}+O(x^4).
\]
It follows that
\[
\begin{aligned}
\tan^n x
&=x^n+O(x^{n+2}),\\
\tan^{n-1}x
&=x^{n-1}+\frac{n-1}{3}x^{n+1}+O(x^{n+3}),\\
\sec x\,\tan^{n-1}x
&=x^{n-1}+O(x^{n+1}),\\
\sec x\,\tan^{n-2}x
&=x^{n-2}+\frac{2n-1}{6}x^n+O(x^{n+2}).
\end{aligned}
\]
Comparison with the defining exponential generating functions gives
\[
\begin{aligned}
T(n,n)&=n!,
&T(n+1,n-1)&=\frac{(n+1)!(n-1)}{3},\\
S(n-1,n-1)&=(n-1)!,
&S(n,n-2)&=\frac{n!(2n-1)}{6}.
\end{aligned}
\]
Multiplying by the factorial prefactors in the definitions of \(P_n\) and
\(Q_n\) proves \eqref{eq:P-edge-coefficients} and
\eqref{eq:Q-edge-coefficients}.
\end{proof}

\subsection{Exact $2$-adic content of $P_n$}

For a formal power series \(f(x)\), we write \([x^m]f(x)\) for the
coefficient of \(x^m\).  For constant terms we use the more natural
notation \(f(0)\) whenever this improves readability.

The strong divisibility of the classical tangent numbers by powers of two is
closely related to the arithmetic of the Genocchi numbers; see, for example,
\cite{knuth-buckholtz1967,Gessel1981}.  The following result determines the
exact common power of two in the coefficients of $P_n$.  Its proof combines
the classical tangent--Bernoulli formula with the von Staudt--Clausen theorem
and a parity argument based on Lagrange inversion.

\begin{theorem}[Exact $2$-adic content of $P_n$]
\label[theorem]{thm:Pn-exact-2-adic-content}
Let $\nu_2$ denote the $2$-adic valuation on $\mathbb{Q}^{\times}$, normalized
by $\nu_2(2)=1$, and let $s_2(m)$ be the sum of the binary digits of the
nonnegative integer $m$. Put \(L_n:=\lfloor\log_2 n\rfloor+1\). Then, for every $n\geq1$,
\begin{equation}
\boxed{
\min_{1\leq p\leq n}
\nu_2\!\left(p!\,T(2n-p,p)\right)
=
2n-s_2(n)-L_n.
}
\label{eq:Pn-exact-2-adic-content}
\end{equation}
Moreover, the minimum is attained at
\begin{equation}
    p_n:=2n-2^{L_n}+1.
    \label{eq:Pn-minimizing-index}
\end{equation}
Consequently, $2^{\,2n-s_2(n)-L_n}$ is the largest power of two dividing
every coefficient of $P_n(y)$.
\end{theorem}

\begin{proof}
Write \(\tan x=\sum_{m\geq1}c_mx^{2m-1}\), where
\(c_m=T(2m-1,1)/(2m-1)!\). The classical tangent--Bernoulli formula, together with the
von Staudt--Clausen theorem, gives
\[
T(2m-1,1)
=
(-1)^{m-1}\frac{2^{2m}(2^{2m}-1)}{2m}B_{2m},
\qquad
\nu_2(B_{2m})=-1;
\]
see \cite{knuth-buckholtz1967} and
\cite[pp.~233--237]{IrelandRosen1990}, respectively.  Hence
\begin{equation}
\nu_2\!\left(T(2m-1,1)\right)
=2m-2-\nu_2(m).
\label{eq:tangent-number-2-adic-valuation}
\end{equation}
Legendre's formula in base two gives
\(\nu_2(r!)=\sum_{j\geq1}\lfloor r/2^j\rfloor=r-s_2(r)\). Therefore
\begin{align}
\nu_2(c_m)
&=2m-2-\nu_2(m)-\bigl((2m-1)-s_2(2m-1)\bigr)\nonumber\\
&=s_2(m-1)-\nu_2(m),
\label{eq:ordinary-tangent-coefficient-valuation}
\end{align}
where we used $s_2(2m-1)=s_2(m-1)+1$, since
$2m-1=2(m-1)+1$. If $m=2^a u$ with $u$ odd, then
\(m-1=2^a(u-1)+(2^a-1)\), and the two summands occupy disjoint blocks
of binary digits. Thus \(s_2(m-1)=a+s_2(u-1)\), so that
\begin{equation}
\boxed{
\nu_2(c_m)=s_2(u-1)\geq0.
}
\label{eq:ordinary-tangent-coefficient-unit-criterion}
\end{equation}
In particular, $c_m$ is a $2$-adic unit if and only if $m$ is a power of
$2$. If \(\mathbb Z_{(2)}:=\{a/b\in\mathbb Q:b\text{ is odd}\}\), then $\tan x\in\mathbb Z_{(2)}[[x]]$, and reduction modulo $2$ gives
\begin{equation}
\tan x\equiv
F(x):=\sum_{j\geq0}x^{2^{j+1}-1}
=x+x^3+x^7+x^{15}+\cdots
\pmod{2\mathbb Z_{(2)}[[x]]}.
\label{eq:tangent-mod-2-series}
\end{equation}

Now fix $1\leq p\leq n$ and put $N:=2n-p$. Since
\(T(N,p)/N!=[x^N]\tan^p x\in\mathbb Z_{(2)}\), we have
\(\nu_2(T(N,p))\geq\nu_2(N!)=N-s_2(N)\). Thus, with $a_{n,p}:=p!\,T(2n-p,p)$,
\begin{equation}
\nu_2(a_{n,p})
\geq
2n-s_2(p)-s_2(2n-p).
\label{eq:Pn-coefficient-lower-bound-digit-sums}
\end{equation}

We next bound the two binary digit sums.  Set $q:=2n-p$, so that
$p+q=2n$.  In a binary addition, each carry decreases the total digit sum by
one. Consequently, \(s_2(p)+s_2(q)-s_2(2n)\) is the total number of carries in the addition of $p$ and $q$.  Since
$p<2^{L_n}$, $q<2^{L_n+1}$, and $p+q=2n<2^{L_n+1}$, there is no carry
from binary position $L_n$ to position $L_n+1$.  Thus carries can occur only
out of positions $0,1,\ldots,L_n-1$, and there are at most $L_n$ of them.  As
$s_2(2n)=s_2(n)$, it follows that
\begin{equation}
    s_2(p)+s_2(2n-p)\leq s_2(n)+L_n.
    \label{eq:binary-digit-sum-bound}
\end{equation}
Combining \eqref{eq:Pn-coefficient-lower-bound-digit-sums} and
\eqref{eq:binary-digit-sum-bound} yields
\begin{equation}
\nu_2(a_{n,p})\geq 2n-s_2(n)-L_n
\qquad (1\leq p\leq n).
\label{eq:Pn-uniform-2-adic-lower-bound}
\end{equation}

It remains to show that equality occurs.  Put
\[
p_*:=2n-2^{L_n}+1,
\qquad
N_*:=2n-p_*=2^{L_n}-1.
\]
Since $2^{L_n-1}\leq n\leq2^{L_n}-1$, one has
$1\leq p_*\leq n\leq N_*$.
By \eqref{eq:tangent-mod-2-series}, it suffices to prove that
$[x^{N_*}]F(x)^{p_*}$ is nonzero over $\mathbb F_2$.  In
$\mathbb F_2[[x]]$ the series $F$ satisfies
\begin{equation}
    F=x(1+F^2).
    \label{eq:F-functional-equation}
\end{equation}
Indeed, squaring $F$ in characteristic two simply doubles all exponents.
Let $H\in x\mathbb Z[[x]]$ be the unique formal power series satisfying
\(H=x(1+H^2)\). Its reduction modulo $2$ satisfies \eqref{eq:F-functional-equation}; by
uniqueness of the solution in $x\mathbb F_2[[x]]$, it is precisely $F$.
Lagrange inversion, in the form of
\cite[Theorem~2.1.1]{Gessel2016}, therefore gives, for $N\geq p\geq1$
with $N-p$ even,
\begin{equation}
[x^N]H^p
=
\frac{p}{N}[t^{N-p}](1+t^2)^N
=
\frac{p}{N}\binom{N}{(N-p)/2}.
\label{eq:Lagrange-H-power}
\end{equation}

For $N=N_*=2^{L_n}-1$, every binomial coefficient
$\binom{N_*}{j}$ is odd. Indeed, in $\mathbb F_2[t]$,
\((1+t)^{N_*}(1+t)=(1+t)^{2^{L_n}}=1+t^{2^{L_n}}\), whence
\((1+t)^{N_*}=1+t+\cdots+t^{N_*}\). Both $p_*$ and $N_*$ are odd.  Hence the right-hand side of
\eqref{eq:Lagrange-H-power}, with $(N,p)=(N_*,p_*)$, is a $2$-adic unit.
Since $H\equiv F\equiv\tan x\pmod2$, we obtain
\(\nu_2([x^{N_*}]\tan^{p_*}x)=0\), and therefore
\begin{equation}
\nu_2\!\left(T(N_*,p_*)\right)=\nu_2(N_*!).
\label{eq:higher-tangent-equality-case}
\end{equation}

Finally, write \(n=2^{L_n-1}+r\), where
\(0\leq r<2^{L_n-1}\). Then $p_*=2r+1$, so
\(s_2(p_*)=1+s_2(r)=s_2(n)\) and \(s_2(N_*)=L_n\). Using \eqref{eq:higher-tangent-equality-case} and Legendre's formula once
more, we find
\begin{align*}
\nu_2\!\left(p_*!\,T(N_*,p_*)\right)
&=\nu_2(p_*!)+\nu_2(N_*!)\\
&=p_*-s_2(p_*)+N_*-s_2(N_*)\\
&=2n-s_2(n)-L_n.
\end{align*}
Together with \eqref{eq:Pn-uniform-2-adic-lower-bound}, this proves
\eqref{eq:Pn-exact-2-adic-content} and shows that the minimum is attained at
$p_*=p_n$.
\end{proof}

\begin{corollary}
\label[corollary]{cor:Pn-2-adic-content-recurrence}
If \(e_n:=2n-s_2(n)-\lfloor\log_2 n\rfloor-1\), then $e_1=0$ and, for every $n\geq2$,
\begin{equation}
    e_n=n-1+e_{\lfloor n/2\rfloor}.
    \label{eq:Pn-2-adic-content-recurrence}
\end{equation}
\end{corollary}

\begin{proof}
Write $n=2m+\varepsilon$, where $m=\lfloor n/2\rfloor$ and
$\varepsilon\in\{0,1\}$. For $n\geq2$,
\(s_2(n)=s_2(m)+\varepsilon\) and
\(\lfloor\log_2 n\rfloor=\lfloor\log_2 m\rfloor+1\).
Substitution into the closed formula gives \(e_n=n-1+e_m\).
\end{proof}

\subsection{Odd common divisors and the exact content of $P_n$}

We now determine the odd part of the content of $P_n$.  The argument is
parallel in spirit to the proof of the primitivity of $Q_n$, but uses the
ordinary tangent numbers and Kummer's congruence for Bernoulli numbers.

\begin{proposition}[No odd common prime divisor for $P_n$]
\label[proposition]{prop:Pn-no-odd-common-prime}
For every $n\geq1$, no odd prime divides all coefficients of $P_n$.
Equivalently, if \(a_{n,p}:=p!\,T(2n-p,p)\) for $1\leq p\leq n$, then
\(q\nmid\gcd_{1\leq p\leq n}\{a_{n,p}\}\) for every odd prime $q$.
\end{proposition}

\begin{proof}
Suppose that an odd prime $q$ divides every $a_{n,p}$. By
\cref{prop:PQ-edge-coefficients}, \(a_{n,n}=(n!)^2\), and hence
$q\leq n$. Put \(h:=(q-1)/2\). For $0\leq r\leq h-1$, choose $p=2r+1$.  Then $p<q\leq n$, so $p!$ is
invertible modulo $q$.  Therefore
\begin{equation}
T(2n-2r-1,2r+1)\equiv0\pmod q
\qquad(0\leq r\leq h-1).
\label{eq:Pn-higher-tangent-congruences}
\end{equation}

Let $D_m(t)$ be the tangent derivative polynomials defined by
\(d^m(\tan x)/dx^m=D_m(\tan x)\). They satisfy
\(D_0(t)=t\) and \(D_{m+1}(t)=(1+t^2)D_m'(t)\). Consequently, $D_{2j}(t)$ is an odd polynomial of degree $2j+1$ with
leading coefficient $(2j)!$.  Thus, for every $r\geq0$, there are unique
rational numbers $c_{r,j}$ such that
\begin{equation}
t^{2r+1}=\sum_{j=0}^{r}c_{r,j}D_{2j}(t).
\label{eq:odd-monomial-tangent-derivative-basis}
\end{equation}
Moreover, $D_m(t)$ is divisible by $1+t^2$ for every $m\geq1$.  Evaluating
\eqref{eq:odd-monomial-tangent-derivative-basis} at $t=i$ therefore gives
\(i^{2r+1}=c_{r,0}i\), and hence
\begin{equation}
c_{r,0}=(-1)^r.
\label{eq:tangent-basis-cr0}
\end{equation}

Fix $0\leq r\leq h-1$. Since $2r\leq q-3$, the triangular
change-of-basis matrix between
\(t,t^3,\ldots,t^{2r+1}\) and
\(D_0(t),D_2(t),\ldots,D_{2r}(t)\) has diagonal entries
\(0!,2!,\ldots,(2r)!\), all of which are units in
\(\mathbb Z_{(q)}:=\{a/b\in\mathbb Q:q\nmid b\}\). Hence
\begin{equation}
c_{r,j}\in\mathbb Z_{(q)}
\qquad(0\leq j\leq r),
\label{eq:tangent-basis-q-integrality}
\end{equation}
while \eqref{eq:tangent-basis-cr0} shows that $c_{r,0}$ is a unit in
$\mathbb Z_{(q)}$.

Substituting $t=\tan x$ in
\eqref{eq:odd-monomial-tangent-derivative-basis} yields
\[
\tan^{2r+1}x
=
\sum_{j=0}^{r}c_{r,j}\frac{d^{2j}}{dx^{2j}}\tan x.
\]
Comparison of exponential Taylor coefficients gives
\begin{equation}
T(M,2r+1)
=
\sum_{j=0}^{r}c_{r,j}T(M+2j,1)
\qquad(M\geq0).
\label{eq:higher-tangent-to-ordinary-tangent}
\end{equation}
Write \(\tau_m:=T(2m-1,1)\) for $m\geq1$. Taking $M=2n-2r-1$ in
\eqref{eq:higher-tangent-to-ordinary-tangent} and using
\eqref{eq:Pn-higher-tangent-congruences}, we obtain
\begin{equation}
0\equiv
\sum_{j=0}^{r}c_{r,j}\tau_{n-r+j}
\pmod q
\qquad(0\leq r\leq h-1).
\label{eq:Pn-consecutive-tangent-induction}
\end{equation}
For $r=0$ this gives $\tau_n\equiv0\pmod q$.  Since $c_{r,0}=(-1)^r$,
\eqref{eq:Pn-consecutive-tangent-induction} then yields inductively
\begin{equation}
\tau_{n-h+1}\equiv
\tau_{n-h+2}\equiv\cdots\equiv
\tau_n\equiv0\pmod q.
\label{eq:Pn-block-tangent-divisibility}
\end{equation}

Assume first that $q>3$. Among the $h$ consecutive integers
$n-h+1,n-h+2,\ldots,n$ there is an integer $m$ satisfying
\(m\equiv1\pmod h\). Since $2h=q-1$, we have
\(2m\equiv2\not\equiv0\pmod{q-1}\). Kummer's congruence for Bernoulli numbers therefore gives
\[
\frac{B_{2m}}{2m}
\equiv
\frac{B_2}{2}
=
\frac1{12}
\pmod q;
\]
see \cite[pp.~239--241]{IrelandRosen1990}.  Using the classical
tangent--Bernoulli formula,
\[
\tau_m
=
(-1)^{m-1}
2^{2m}(2^{2m}-1)\frac{B_{2m}}{2m},
\]
and Fermat's theorem, we obtain
\[
\tau_m
\equiv
(-1)^{m-1}\,4(4-1)\frac1{12}
=
(-1)^{m-1}
\not\equiv0\pmod q.
\]
This contradicts \eqref{eq:Pn-block-tangent-divisibility}.

It remains to exclude $q=3$. From \((\tan x)'=1+\tan^2x\) one obtains, for $m\geq1$,
\begin{equation}
\tau_{m+1}
=
\sum_{j=1}^{m}
\binom{2m}{2j-1}\tau_j\tau_{m+1-j}.
\label{eq:tangent-number-quadratic-recurrence}
\end{equation}
We claim that \(\tau_m\equiv(-1)^{m-1}\pmod3\) for $m\geq1$.
This is clear for $m=1$.  If it holds up to $m$, then
\eqref{eq:tangent-number-quadratic-recurrence} gives
\[
\tau_{m+1}
\equiv
(-1)^{m-1}
\sum_{j=1}^{m}\binom{2m}{2j-1}
=
(-1)^{m-1}2^{2m-1}
\equiv
(-1)^m
\pmod3.
\]
Thus no $\tau_m$ is divisible by $3$, contradicting
\eqref{eq:Pn-block-tangent-divisibility} also in this case.  Hence no odd
prime can divide all coefficients of $P_n$.
\end{proof}

\begin{theorem}[Exact content of $P_n$]
\label[theorem]{thm:Pn-content}
With $s_2$ and $L_n$ as in
\cref{thm:Pn-exact-2-adic-content}, one has, for every $n\geq1$,
\begin{equation}
\boxed{
\operatorname{cont}(P_n)
=
\gcd_{1\leq p\leq n}
\bigl\{p!\,T(2n-p,p)\bigr\}
=
2^{\,2n-s_2(n)-L_n}
=
2^{\,2n-s_2(n)-\lfloor\log_2 n\rfloor-1}.
}
\label{eq:Pn-content}
\end{equation}
Here $\operatorname{cont}(P_n)$ denotes the positive greatest common divisor
of the coefficients of $P_n$.
\end{theorem}

\begin{proof}
Let \(d_n:=\gcd_{1\leq p\leq n}\{p!\,T(2n-p,p)\}\). By
\cref{thm:Pn-exact-2-adic-content},
\(\nu_2(d_n)=2n-s_2(n)-L_n\). By \cref{prop:Pn-no-odd-common-prime}, no odd prime divides $d_n$.
Therefore $d_n$ is a power of two, and its exact exponent is the one given
above.  This proves \eqref{eq:Pn-content}.
\end{proof}

\subsection{Primitivity of $Q_n$}

The arithmetic behavior of the secant companion is markedly different from
that of $P_n$: no nontrivial integer divides all coefficients of $Q_n$.

\begin{theorem}[Primitivity of $Q_n$]
\label[theorem]{thm:Qn-primitive}
For every $n\geq1$,
\begin{equation}
\boxed{
\operatorname{cont}(Q_n)
=
\gcd_{1\leq p\leq n}
\bigl\{p!\,S(2n-p-1,p-1)\bigr\}
=
1,
}
\label{eq:Qn-primitive}
\end{equation}
where $\operatorname{cont}(Q_n)$ denotes the positive greatest common
divisor of the coefficients of $Q_n$.
\end{theorem}

\begin{proof}
Put \(b_{n,p}:=p!\,S(2n-p-1,p-1)\) for $1\leq p\leq n$.
We first exclude the prime $2$. Let \(\mathcal E_m:=S(2m,0)\) for $m\geq0$.
The coefficient recurrence
\(S(m+1,k)=kS(m,k-1)+(k+1)S(m,k+1)\) gives, for $m\geq1$,
\(\mathcal E_m=S(2m-1,1)=S(2m-2,0)+2S(2m-2,2)
\equiv\mathcal E_{m-1}\pmod2\). Since $\mathcal E_0=S(0,0)=1$,
all $\mathcal E_m$ are odd. In particular,
\(b_{n,1}=S(2n-2,0)=\mathcal E_{n-1}\) is odd, so $2$ cannot divide all coefficients of $Q_n$.

Suppose now that an odd prime $q$ divides every $b_{n,p}$. By
\cref{prop:PQ-edge-coefficients}, \(b_{n,n}=n!(n-1)!\), hence $q\leq n$.
Set \(h:=(q-1)/2\). For $0\leq r\leq h-1$, choose $p=2r+1$.  Then
$p<q\leq n$, so $p!$ is invertible modulo $q$, and therefore
\begin{equation}
S(2n-2r-2,2r)\equiv0\pmod q
\qquad (0\leq r\leq h-1).
\label{eq:Qn-S-congruences}
\end{equation}

We next convert these congruences into congruences for consecutive secant
numbers. Define the secant derivative polynomials $R_m(t)$ by
\(d^m(\sec x)/dx^m=\sec x\,R_m(\tan x)\). The addition formula for the
secant gives
\[
\sum_{m\geq0}R_m(t)\frac{z^m}{m!}
=\frac{\sec z}{1-t\tan z}.
\]
Comparison with the defining generating function of $S(m,k)$ yields
\begin{equation}
R_m(t)=\sum_{k=0}^{m}S(m,k)t^k.
\label{eq:secant-derivative-polynomial-S}
\end{equation}
Taking the even part in $z$ gives
\begin{equation}
\sum_{j\geq0}R_{2j}(t)\frac{z^{2j}}{(2j)!}
=
\frac{\sec z}{1-t^2\tan^2z}.
\label{eq:even-secant-derivative-generating-function}
\end{equation}
In particular, $R_{2j}$ is an even polynomial of degree $2j$ with leading
coefficient \(S(2j,2j)=(2j)!\). Hence, for every $r\geq0$, there are unique rational numbers $c_{r,j}$ such
that
\begin{equation}
t^{2r}=\sum_{j=0}^{r}c_{r,j}R_{2j}(t).
\label{eq:monomial-in-secant-derivative-basis}
\end{equation}

We shall need the coefficient $c_{r,0}$. Let $\Lambda$ be the linear
functional on the even polynomial ring determined by
\(\Lambda(R_0)=1\) and \(\Lambda(R_{2j})=0\) for $j\geq1$. Then $\Lambda(t^{2r})=c_{r,0}$.  Applying $\Lambda$ coefficientwise to
\eqref{eq:even-secant-derivative-generating-function} gives
\[
1
=
\sec z\sum_{r\geq0}c_{r,0}\tan^{2r}z.
\]
Consequently,
\[
\sum_{r\geq0}c_{r,0}\tan^{2r}z
=
\cos z
=
(1+\tan^2z)^{-1/2}.
\]
Since $\tan^2z=z^2+O(z^4)$, composition with $\tan^2z$ is injective on
formal power series, and hence
\begin{equation}
c_{r,0}
=
(-1)^r\frac{1}{4^r}\binom{2r}{r}.
\label{eq:cr0-central-binomial}
\end{equation}

Fix now $0\leq r\leq h-1$. Since $2r\leq q-3$, the triangular
change-of-basis matrix between \(1,t^2,\ldots,t^{2r}\) and
\(R_0(t),R_2(t),\ldots,R_{2r}(t)\) has diagonal entries
\(0!,2!,\ldots,(2r)!\), all of which are units in the localization
\(\mathbb Z_{(q)}:=\{a/b\in\mathbb Q:q\nmid b\}\). Thus every $c_{r,j}$ belongs to $\mathbb Z_{(q)}$.  Moreover,
\eqref{eq:cr0-central-binomial} and $2r<q$ show that
\begin{equation}
c_{r,0}\in\mathbb Z_{(q)}^\times.
\label{eq:cr0-q-adic-unit}
\end{equation}

Multiplying \eqref{eq:monomial-in-secant-derivative-basis} by $\sec x$,
putting $t=\tan x$, and using the definition of $R_{2j}$ gives
\[
\sec x\,\tan^{2r}x
=
\sum_{j=0}^{r}c_{r,j}\frac{d^{2j}}{dx^{2j}}\sec x.
\]
Comparison of the coefficients of $x^M/M!$ yields
\begin{equation}
S(M,2r)
=
\sum_{j=0}^{r}c_{r,j}S(M+2j,0)
\qquad(M\geq0).
\label{eq:S-to-secant-triangular-relation}
\end{equation}
Taking $M=2n-2r-2$ and using
\eqref{eq:Qn-S-congruences}, we obtain
\begin{equation}
0
\equiv
\sum_{j=0}^{r}
c_{r,j}\,
\mathcal E_{\,n-r-1+j}
\pmod q.
\label{eq:Qn-consecutive-secant-induction}
\end{equation}
For $r=0$ this gives $\mathcal E_{n-1}\equiv0\pmod q$.
Using \eqref{eq:cr0-q-adic-unit}, equation
\eqref{eq:Qn-consecutive-secant-induction} then yields inductively
\begin{equation}
\mathcal E_{n-h}\equiv
\mathcal E_{n-h+1}\equiv\cdots\equiv
\mathcal E_{n-1}\equiv0\pmod q.
\label{eq:Qn-block-secant-divisibility}
\end{equation}

We now use the classical Kummer-type congruence for Euler numbers
\[
E_a\equiv E_b\pmod q
\qquad
\text{if }a\equiv b\pmod{q-1}
\quad\text{and}\quad a,b\geq1;
\]
see, for example, \cite[Corollary~3.9]{Chen2004Euler}.
Among the $h$ consecutive integers $n-h,n-h+1,\ldots,n-1$ there is
an integer $m$ satisfying \(m\equiv1\pmod h\). Hence
\(2m\equiv2\pmod{q-1}\), and therefore
\(E_{2m}\equiv E_2=-1\pmod q\). On the other hand,
\(\mathcal E_m=S(2m,0)=(-1)^mE_{2m}\), so $\mathcal E_m\not\equiv0\pmod q$, contradicting
\eqref{eq:Qn-block-secant-divisibility}.  Thus no odd prime divides all
coefficients of $Q_n$.  Since the prime $2$ was already excluded, the
content of $Q_n$ is $1$.
\end{proof}

\section{Structural properties of $P_n$ and $Q_n$}

\subsection{Real zeros}

\begin{proposition}[The real zeros of $P_{2n}$ and $Q_{2n}$]
\label[proposition]{prop:PQ-even-real-zero}
For every $n\geq1$, the polynomials $P_{2n}$ and $Q_{2n}$ each have
exactly one real zero, and this zero is negative. Equivalently, there exist unique numbers $\rho_n,\sigma_n>0$ such that
\(P_{2n}(-\rho_n)=0\) and \(Q_{2n}(-\sigma_n)=0\).
\end{proposition}

The sign of the real zero is immediate from the positivity of the
coefficients: neither polynomial can vanish on $[0,\infty)$. Moreover,
$P_{2n}$ and $Q_{2n}$ both have degree $2n-1$ with positive leading
coefficient, so each has at least one negative real zero. Thus the
substantive assertion in \cref{prop:PQ-even-real-zero} is uniqueness.

\begin{proof}
We prove the stronger assertion
\begin{equation}
\label{eq:PQ-even-monotonicity}
P_{2n}'(y)>0,
\qquad
Q_{2n}'(y)>0
\qquad (y\in\mathbb R).
\end{equation}
The uniqueness of the real zero will then follow immediately.

Put \(\mathcal A(u):=\tan\sqrt u/\sqrt u\) and
\(\mathcal B(u):=\sec\sqrt u\). By the defining exponential generating functions of the higher-order
tangent and secant numbers,
\begin{equation}
\label{eq:PQ-coefficient-extraction}
\frac{T(2N-p,p)}{(2N-p)!}
=[u^{N-p}]\mathcal A(u)^p,
\qquad
\frac{S(2N-p-1,p-1)}{(2N-p-1)!}
=[u^{N-p}]\mathcal B(u)\mathcal A(u)^{p-1}.
\end{equation}

We first derive coefficient representations for the derivatives. For
$N\geq2$, the beta integral gives
\(p!(2N-p)!/(2N+1)!=\int_0^1 t^p(1-t)^{2N-p}\,dt\). Using \eqref{eq:PQ-coefficient-extraction}, we obtain
\begin{align*}
\frac{P_N'(y)}{(2N+1)!}
&=
\sum_{p=2}^{N}(p-1)y^{p-2}
\int_0^1
t^p(1-t)^{2N-p}
[u^{N-p}]\mathcal A(u)^p\,dt
\\
&=
\int_0^1
t^2(1-t)^{2N-2}
[u^{N-2}]
\mathcal A(u)^2
\sum_{p\geq2}(p-1)
\left(
y\frac{t}{1-t}u\mathcal A(u)
\right)^{p-2}
dt.
\end{align*}
Since \(\sum_{p\geq2}(p-1)z^{p-2}=(1-z)^{-2}\), it follows that
\begin{equation}
\label{eq:PN-derivative-coefficient}
\frac{P_N'(y)}{(2N+1)!}
=
\int_0^1
t^2(1-t)^{2N-2}
[u^{N-2}]
\frac{\mathcal A(u)^2}
{\left(
1-y\frac{t}{1-t}u\mathcal A(u)
\right)^2}
\,dt.
\end{equation}

Similarly,
\(p!(2N-p-1)!/(2N)!=\int_0^1 t^p(1-t)^{2N-p-1}\,dt\), and therefore
\begin{equation}
\label{eq:QN-derivative-coefficient}
\frac{Q_N'(y)}{(2N)!}
=
\int_0^1
t^2(1-t)^{2N-3}
[u^{N-2}]
\frac{\mathcal A(u)\mathcal B(u)}
{\left(
1-y\frac{t}{1-t}u\mathcal A(u)
\right)^2}
\,dt.
\end{equation}

For $y\geq0$, the inequalities in
\eqref{eq:PQ-even-monotonicity} are immediate from the positivity of
the coefficients. It therefore remains to consider $y<0$. Write $y=-x$ with $x>0$ and,
for $0<t<1$, put \(a:=xt/(1-t)>0\). Define
\begin{equation}
\label{eq:Da-definition}
D_a(u):=
\cos\sqrt u+a\sqrt u\,\sin\sqrt u
\end{equation}
and
\begin{equation}
\label{eq:Ga-Ka-definition}
G_a(u):=
\frac{\mathcal A(u)}
     {1+au\mathcal A(u)}
=
\frac{\sin\sqrt u}
     {\sqrt u\,D_a(u)},
\qquad
K_a(u):=
\frac{\mathcal B(u)}
     {1+au\mathcal A(u)}
=
\frac{1}{D_a(u)}.
\end{equation}
Thus the coefficients occurring in
\eqref{eq:PN-derivative-coefficient} and
\eqref{eq:QN-derivative-coefficient} reduce to
\begin{equation}
\label{eq:Ga-coefficients}
[u^{N-2}]G_a(u)^2,
\qquad
[u^{N-2}]G_a(u)K_a(u),
\end{equation}
respectively.

We shall show that all even Taylor coefficients in
\eqref{eq:Ga-coefficients} are strictly positive.

Consider on $L^2(0,1)$ the regular Sturm--Liouville operator
\(L_af=-f''\) with separated boundary conditions
\begin{equation}
\label{eq:Sturm-Liouville-boundary}
f'(0)=0,
\qquad
f(1)-af'(1)=0.
\end{equation}
This is a self-adjoint regular Sturm--Liouville problem. Hence its
spectrum consists of real, simple eigenvalues
$\{\lambda_j\}_{j\geq1}$ with no finite accumulation point, and its
normalized eigenfunctions form an orthonormal basis of $L^2(0,1)$;
see, for example, \cite{Zettl2005}.

The eigenvalue equation
\[
-f''=\lambda f
\]
together with $f'(0)=0$ gives
\[
f(s)=c\cos(\sqrt\lambda\,s),
\]
and the second boundary condition in
\eqref{eq:Sturm-Liouville-boundary} becomes
\[
\cos\sqrt\lambda
+
a\sqrt\lambda\,\sin\sqrt\lambda
=0.
\]
Consequently,
\begin{equation}
\label{eq:Da-eigenvalues}
D_a(\lambda_j)=0.
\end{equation}
Moreover, $D_a(0)=1$, so $0$ is not an eigenvalue.

Let $\phi_j$ be normalized real eigenfunctions and set
\begin{equation}
\label{eq:wj-definition}
w_j:=
\left|\langle 1,\phi_j\rangle\right|^2.
\end{equation}
Completeness gives
\begin{equation}
\label{eq:wj-sum}
\sum_{j\geq1}w_j
=
\|1\|_{L^2(0,1)}^2
=
1.
\end{equation}
Furthermore,
\begin{equation}
\label{eq:wj-positive}
w_j>0.
\end{equation}
Indeed, if $\lambda_j>0$, then
\(\langle1,\phi_j\rangle=c_j\sin\sqrt{\lambda_j}/\sqrt{\lambda_j}\).
This quantity cannot vanish, since
$\sin\sqrt{\lambda_j}=0$ together with
\eqref{eq:Da-eigenvalues} would imply
$\cos\sqrt{\lambda_j}=0$, which is impossible. The argument for a
negative eigenvalue is identical after replacing the trigonometric
functions by their hyperbolic counterparts.

We next identify $G_a$ spectrally. For $u$ outside the spectrum, let
\[
h=(L_a-u)^{-1}1.
\]
Then
\[
-h''-uh=1
\]
with the boundary conditions
\eqref{eq:Sturm-Liouville-boundary}. Direct solution gives
\[
h(s)
=
\frac{\cos(\sqrt u\,s)}
     {uD_a(u)}
-\frac1u.
\]
Hence
\begin{align*}
\left\langle
1,L_a(L_a-u)^{-1}1
\right\rangle
&=
1+u\int_0^1h(s)\,ds
\\
&=
\frac{\sin\sqrt u}
     {\sqrt u\,D_a(u)}
=
G_a(u).
\end{align*}
By the spectral theorem,
\begin{equation}
\label{eq:Ga-spectral-representation}
G_a(u)
=
\sum_{j\geq1}
\frac{w_j}{1-u/\lambda_j}.
\end{equation}

We also require a product representation for $K_a$. The function
$D_a$ is an entire function of $u$ of order $1/2$. Its zeros are
precisely the eigenvalues $\lambda_j$, and the standard eigenvalue
asymptotics for a regular Sturm--Liouville problem imply
\begin{equation}
\label{eq:eigenvalue-reciprocal-sum}
\sum_{j\geq1}\frac{1}{|\lambda_j|}<\infty.
\end{equation}
Therefore, by the Hadamard factorization theorem
\cite{Conway1995} and the normalization $D_a(0)=1$,
\begin{equation}
\label{eq:Da-Hadamard-product}
D_a(u)
=
\prod_{j\geq1}
\left(
1-\frac{u}{\lambda_j}
\right).
\end{equation}
Consequently,
\begin{equation}
\label{eq:Ka-product}
K_a(u)
=
\prod_{j\geq1}
\left(
1-\frac{u}{\lambda_j}
\right)^{-1}.
\end{equation}

Set \(\tau_j:=1/\lambda_j\). Let $E,E_1,E_2,\ldots$ be independent
standard exponentially distributed random variables with
\(\mathbb P(E\in ds)=e^{-s}\,ds\) for $s>0$. Let $J$, independent of $E$,
satisfy \(\mathbb P(J=j)=w_j\), which is possible by \eqref{eq:wj-sum}, and define
\begin{equation}
\label{eq:X-definition}
X:=\tau_JE.
\end{equation}
Then, for $u$ sufficiently close to $0$,
\begin{equation}
\label{eq:Ga-mgf}
\mathbb E(e^{uX})
=
\sum_{j\geq1}
w_j\frac{1}{1-\tau_ju}
=
G_a(u).
\end{equation}

By \eqref{eq:eigenvalue-reciprocal-sum},
\(\sum_{j\geq1}|\tau_j|<\infty\). Hence the random series
\begin{equation}
\label{eq:Y-definition}
Y:=
\sum_{j\geq1}\tau_jE_j
\end{equation}
converges absolutely almost surely and in $L^1$. Independence and
\eqref{eq:Ka-product} yield, again for $u$ sufficiently close to
$0$,
\begin{equation}
\label{eq:Ka-mgf}
\mathbb E(e^{uY})
=
\prod_{j\geq1}
\frac{1}{1-\tau_ju}
=
K_a(u).
\end{equation}

Let $X_1$ and $X_2$ be independent copies of $X$, and take all
random variables above to be mutually independent. From \eqref{eq:Ga-mgf} and \eqref{eq:Ka-mgf},
\[
G_a(u)^2=\mathbb E\!\left(e^{u(X_1+X_2)}\right),
\qquad
G_a(u)K_a(u)=\mathbb E\!\left(e^{u(X+Y)}\right).
\]
Therefore, for every $m\geq0$,
\begin{equation}
\label{eq:Ga-square-even-positive}
[u^{2m}]G_a(u)^2
=
\frac{
\mathbb E\!\left((X_1+X_2)^{2m}\right)
}{(2m)!}
>0,
\end{equation}
and
\begin{equation}
\label{eq:GaKa-even-positive}
[u^{2m}]G_a(u)K_a(u)
=
\frac{
\mathbb E\!\left((X+Y)^{2m}\right)
}{(2m)!}
>0.
\end{equation}
For $m=0$, both coefficients equal $1$. For $m\geq1$, the
inequalities are strict because the corresponding real random
variables are not almost surely zero.

We now return to
\eqref{eq:PN-derivative-coefficient} and
\eqref{eq:QN-derivative-coefficient}. Taking $N=2n$, the coefficient
index is $N-2=2n-2$, which is even. Hence
\eqref{eq:Ga-square-even-positive} and
\eqref{eq:GaKa-even-positive} imply, for every $x>0$,
\[
P_{2n}'(-x)>0,
\qquad
Q_{2n}'(-x)>0.
\]
Together with the already observed positivity for $y\geq0$, this
proves \eqref{eq:PQ-even-monotonicity}.

Thus $P_{2n}$ and $Q_{2n}$ are strictly increasing on $\mathbb R$.
Since $P_{2n}(0)>0$ and $Q_{2n}(0)>0$, whereas both polynomials have odd degree $2n-1$ and positive leading
coefficient,
\[
\lim_{y\to-\infty}P_{2n}(y)
=
\lim_{y\to-\infty}Q_{2n}(y)
=
-\infty.
\]
The intermediate value theorem therefore gives a negative real zero
for each polynomial, and strict monotonicity shows that this zero is
unique.
\end{proof}

\begin{proposition}[Absence of real zeros of $P_{2n-1}$ and $Q_{2n-1}$]
\label[proposition]{prop:PQ-odd-no-real-zero}
For every $n\geq1$, the polynomials $P_{2n-1}$ and $Q_{2n-1}$
have no real zeros. More precisely,
\[
P_{2n-1}(y)>0,
\qquad
Q_{2n-1}(y)>0
\qquad (y\in\mathbb R).
\]
\end{proposition}

\begin{proof}
For $y\geq0$, the assertion follows immediately from the positivity
of the coefficients. It remains to consider $y<0$.

We use the notation introduced in the proof of
\cref{prop:PQ-even-real-zero}. Recall that
\(\mathcal A(u)=\tan\sqrt u/\sqrt u\) and \(\mathcal B(u)=\sec\sqrt u\).
By \eqref{eq:PQ-coefficient-extraction},
\(T(2N-p,p)/(2N-p)!=[u^{N-p}]\mathcal A(u)^p\). Using the beta integral
\(p!(2N-p)!/(2N+1)!=\int_0^1t^p(1-t)^{2N-p}\,dt\), we obtain
\begin{align*}
\frac{P_N(y)}{(2N+1)!}
&=
\sum_{p=1}^{N}y^{p-1}
\int_0^1
t^p(1-t)^{2N-p}
[u^{N-p}]\mathcal A(u)^p\,dt
\\
&=
\int_0^1
t(1-t)^{2N-1}
[u^{N-1}]
\mathcal A(u)
\sum_{p\geq1}
\left(
y\frac{t}{1-t}u\mathcal A(u)
\right)^{p-1}
dt.
\end{align*}
Hence
\begin{equation}
\label{eq:PN-value-coefficient}
\frac{P_N(y)}{(2N+1)!}
=
\int_0^1
t(1-t)^{2N-1}
[u^{N-1}]
\frac{\mathcal A(u)}
{1-y\frac{t}{1-t}u\mathcal A(u)}
\,dt.
\end{equation}

Similarly, from
\(S(2N-p-1,p-1)/(2N-p-1)!=[u^{N-p}]\mathcal B(u)\mathcal A(u)^{p-1}\)
and \(p!(2N-p-1)!/(2N)!=\int_0^1t^p(1-t)^{2N-p-1}\,dt\), we obtain
\begin{equation}
\label{eq:QN-value-coefficient}
\frac{Q_N(y)}{(2N)!}
=
\int_0^1
t(1-t)^{2N-2}
[u^{N-1}]
\frac{\mathcal B(u)}
{1-y\frac{t}{1-t}u\mathcal A(u)}
\,dt.
\end{equation}

Now write $y=-x$ with $x>0$ and, for $0<t<1$, put
\(a=xt/(1-t)>0\). By \eqref{eq:Ga-Ka-definition},
\[
\frac{\mathcal A(u)}
{1+au\mathcal A(u)}
=
G_a(u),
\qquad
\frac{\mathcal B(u)}
{1+au\mathcal A(u)}
=
K_a(u).
\]
Thus \eqref{eq:PN-value-coefficient} and
\eqref{eq:QN-value-coefficient} become
\[
\begin{aligned}
\frac{P_N(-x)}{(2N+1)!}
&=\int_0^1t(1-t)^{2N-1}[u^{N-1}]G_a(u)\,dt,\\
\frac{Q_N(-x)}{(2N)!}
&=\int_0^1t(1-t)^{2N-2}[u^{N-1}]K_a(u)\,dt.
\end{aligned}
\]

From \eqref{eq:Ga-mgf} and \eqref{eq:Ka-mgf},
\(G_a(u)=\mathbb E(e^{uX})\) and \(K_a(u)=\mathbb E(e^{uY})\). Hence,
for every $m\geq0$,
\[
[u^{2m}]G_a(u)=\frac{\mathbb E(X^{2m})}{(2m)!}>0,
\qquad
[u^{2m}]K_a(u)=\frac{\mathbb E(Y^{2m})}{(2m)!}>0.
\]
For $m=0$, both coefficients equal $1$, while for $m\geq1$
the inequalities are strict since $X$ and $Y$ are not almost surely
zero.

Finally, take $N=2n-1$. Then $N-1=2n-2$ is even. Therefore, for every $x>0$,
\[
P_{2n-1}(-x)>0,
\qquad
Q_{2n-1}(-x)>0.
\]
Together with the positivity for $y\geq0$, this proves
\[
P_{2n-1}(y)>0,
\qquad
Q_{2n-1}(y)>0
\qquad (y\in\mathbb R).
\]
Hence neither $P_{2n-1}$ nor $Q_{2n-1}$ has a real zero.
\end{proof}

\begin{lemma}[Dominant spectral rates]
\label[lemma]{lem:PQ-dominant-spectral-rates}
Let $c>0$. For $a>0$, put
\[
D_a(u):=\cos\sqrt u+a\sqrt u\,\sin\sqrt u,
\]
and
\[
G_a(u):=
\frac{\sin\sqrt u}{\sqrt u\,D_a(u)},
\qquad
K_a(u):=\frac1{D_a(u)}.
\]
Let $-\kappa(a)^2$ denote the unique negative zero of $D_a$, and let
$\mu_1(a)^2$ denote its smallest positive zero. Then
\[
a=\frac{\coth\kappa}{\kappa},
\qquad
a=-\frac{\cot\mu_1}{\mu_1},
\qquad
\frac{\pi}{2}<\mu_1<\pi.
\]

For $N=2n$, write
\[
\begin{cases}
\dfrac{P_{2n}(-c)}{(4n+1)!}
=
\mathcal A_n^{P}(c)-\mathcal B_n^{P}(c)
+\mathcal E_n^{P}(c),\\[2ex]
\dfrac{Q_{2n}(-c)}{(4n)!}
=
\mathcal A_n^{Q}(c)-\mathcal B_n^{Q}(c)
+\mathcal E_n^{Q}(c),
\end{cases}
\]
where $\mathcal A_n^{P},\mathcal A_n^{Q}>0$ are the contributions
of the smallest positive pole $\mu_1(a)^2$, and
$\mathcal B_n^{P},\mathcal B_n^{Q}>0$ are the contributions of the
negative pole $-\kappa(a)^2$. Then
\begin{equation}
\label{eq:PQ-positive-spectral-rate}
\lim_{n\to\infty}
\left(\mathcal A_n^{P}(c)\right)^{1/(4n)}
=
\lim_{n\to\infty}
\left(\mathcal A_n^{Q}(c)\right)^{1/(4n)}
=
\frac{2}{\pi},
\end{equation}
whereas
\begin{equation}
\label{eq:PQ-negative-spectral-rate}
\lim_{n\to\infty}
\left(\mathcal B_n^{P}(c)\right)^{1/(4n)}
=
\lim_{n\to\infty}
\left(\mathcal B_n^{Q}(c)\right)^{1/(4n)}
=
\frac1{R_-(c)},
\end{equation}
with
\begin{equation}
\label{eq:Rminus-definition}
R_-(c):=
\min_{\kappa>0}
\left(
\kappa+\frac{\coth\kappa}{c}
\right).
\end{equation}
Moreover,
\begin{equation}
\label{eq:PQ-spectral-remainder}
\limsup_{n\to\infty}
\left|\mathcal E_n^{P}(c)\right|^{1/(4n)}
\leq\frac{2}{3\pi}
<
\frac{2}{\pi},
\qquad
\limsup_{n\to\infty}
\left|\mathcal E_n^{Q}(c)\right|^{1/(4n)}
\leq\frac{2}{3\pi}
<
\frac{2}{\pi}.
\end{equation}
\end{lemma}

\begin{proof}
From the coefficient representations established above, for $N\geq1$,
\begin{equation}
\label{eq:PQ-negative-value-integrals}
\begin{cases}
\dfrac{P_N(-c)}{(2N+1)!}
=
\displaystyle
\int_0^1
t(1-t)^{2N-1}
[u^{N-1}]G_a(u)\,dt,\\[3ex]
\dfrac{Q_N(-c)}{(2N)!}
=
\displaystyle
\int_0^1
t(1-t)^{2N-2}
[u^{N-1}]K_a(u)\,dt,
\end{cases}
\qquad
a=\frac{ct}{1-t}.
\end{equation}

The zeros of $D_a$ are the eigenvalues of the regular
Sturm--Liouville problem
\[
-f''=\lambda f,
\qquad f'(0)=0,
\qquad f(1)-af'(1)=0.
\]
There is exactly one negative eigenvalue, \(\lambda_-(a)=-\kappa(a)^2\),
whose defining equation is \(\cosh\kappa-a\kappa\sinh\kappa=0\). Hence
\begin{equation}
\label{eq:a-kappa-relation}
a=\frac{\coth\kappa}{\kappa}.
\end{equation}

The positive eigenvalues are $\mu_j(a)^2$, where
\(\cos\mu_j+a\mu_j\sin\mu_j=0\). For every $j\geq1$,
\((j-\tfrac12)\pi<\mu_j(a)<j\pi\). In particular,
\(\pi/2<\mu_1<\pi\) and \(a=-\cot\mu_1/\mu_1\), and
\begin{equation}
\label{eq:second-positive-pole}
\mu_j(a)>\frac{3\pi}{2},
\qquad j\geq2.
\end{equation}

For $G_a$, the spectral theorem gives
\begin{equation}
\label{eq:G-spectral-representation}
G_a(u)
=
\frac{w_-(a)}{1+u/\kappa(a)^2}
+
\sum_{j\geq1}
\frac{w_j^+(a)}{1-u/\mu_j(a)^2},
\end{equation}
where \(w_-(a)>0\), \(w_j^+(a)>0\), and
\(w_-(a)+\sum_{j\geq1}w_j^+(a)=1\). Consequently, for odd $m$,
\begin{equation}
\label{eq:G-odd-coefficients}
[u^m]G_a(u)
=
\sum_{j\geq1}
\frac{w_j^+(a)}{\mu_j(a)^{2m}}
-
\frac{w_-(a)}{\kappa(a)^{2m}}.
\end{equation}

We next isolate the two relevant poles of $K_a$. Since the zeros of
$D_a$ are simple, the normalized coefficient associated with a zero
$\lambda$ is \(-1/(\lambda D_a'(\lambda))\). At the negative zero,
\[
\begin{aligned}
D_a'(-\kappa^2)
&=\frac{\sinh\kappa}{2\kappa}\left(a+a^2\kappa^2-1\right),\\
a^2\kappa^2
&=\coth^2\kappa=1+\frac1{\sinh^2\kappa},
\end{aligned}
\]
where the second identity follows from \eqref{eq:a-kappa-relation}; hence
\(D_a'(-\kappa^2)>0\). Thus
\begin{equation}
\label{eq:vminus}
v_-(a)
:=
\frac1{\kappa(a)^2D_a'(-\kappa(a)^2)}
=
\frac{2}
{\cosh\kappa+\kappa/\sinh\kappa}
>0.
\end{equation}

At a positive zero $\mu^2$,
\[
D_a'(\mu^2)
=\frac{\sin\mu}{2\mu}\left(a-1-a^2\mu^2\right),
\qquad
-\frac1{\mu^2D_a'(\mu^2)}
=\frac{2\sin\mu}{\mu+\sin\mu\cos\mu},
\]
where the second identity uses $a=-\cot\mu/\mu$. Thus, at the smallest positive pole,
\begin{equation}
\label{eq:vplus}
v_1(a)
:=
-\frac1{\mu_1(a)^2D_a'(\mu_1(a)^2)}
=
\frac{2\sin\mu_1}
{\mu_1+\sin\mu_1\cos\mu_1}
>0.
\end{equation}

Define
\begin{equation}
\label{eq:Ka-remainder-function}
H_a(u)
:=
K_a(u)
-
\frac{v_1(a)}{1-u/\mu_1(a)^2}
-
\frac{v_-(a)}{1+u/\kappa(a)^2}.
\end{equation}
After removal of these two simple poles, $H_a$ is analytic in
\(|u|<\mu_2(a)^2\). By \eqref{eq:second-positive-pole},
\(\mu_2(a)^2>9\pi^2/4\) for $a>0$. Fix
\(\pi^2<r<9\pi^2/4\). We claim that
\begin{equation}
\label{eq:Ha-uniform-bound}
M_r
:=
\sup_{a>0}\sup_{|u|=r}|H_a(u)|
<\infty.
\end{equation}
On every compact interval $[\delta,A]\subset(0,\infty)$ this follows
from the continuous dependence of the simple eigenvalues and their
residues on $a$. It remains to control the limits $a\to0^+$ and
$a\to\infty$.

As $a\to0^+$,
\[
\kappa(a)\to\infty,
\qquad
\mu_1(a)\to\frac{\pi}{2},
\qquad
K_a(u)\longrightarrow\sec\sqrt u,
\]
with the convergence of $K_a$ uniform on $|u|=r$, and moreover
\[
v_-(a)\to0,
\qquad
v_1(a)\to\frac4\pi.
\]
Since $|u|=r$ does not meet any pole of $\sec\sqrt u$, the convergence
of all terms in \eqref{eq:Ka-remainder-function} is uniform on that
circle. Hence $H_a$ is uniformly bounded there for all sufficiently
small $a$.

As $a\to\infty$,
\[
\kappa(a)\to0,
\qquad
\mu_1(a)\to\pi,
\qquad
v_1(a)\to0.
\]
Because \(\pi^2<r<9\pi^2/4\), the function $\sqrt u\,\sin\sqrt u$ has no
zero on $|u|=r$, and therefore \(K_a(u)\to0\) uniformly there. Furthermore,
\[
\frac{v_-(a)}{1+u/\kappa(a)^2}\longrightarrow0,
\qquad
\frac{v_1(a)}{1-u/\mu_1(a)^2}\longrightarrow0
\]
uniformly on $|u|=r$. Hence $H_a$ is uniformly bounded on $|u|=r$ for
all sufficiently large $a$. Together with the compact-$a$ argument,
this proves \eqref{eq:Ha-uniform-bound}.

Cauchy's estimate now gives, uniformly in $a>0$,
\begin{equation}
\label{eq:Ha-Cauchy}
\boxed{
\bigl|[u^m]H_a(u)\bigr|
\leq
M_r r^{-m}.
}
\end{equation}

We now set $N=2n$ and $m=N-1=2n-1$. For the smallest positive pole,
put \(\Phi_+(t):=(1-t)/\mu_1(a)\), where \(a=ct/(1-t)\). Since
\(a=-\cot\mu_1/\mu_1\), we obtain
\begin{equation}
\label{eq:positive-phase}
\Phi_+(t)
=
\frac1{
\mu_1-\dfrac{\cot\mu_1}{c}
}.
\end{equation}
Moreover,
\(\frac{d}{d\mu}(\mu-\cot\mu/c)=1+\csc^2\mu/c>0\). Therefore
\begin{equation}
\label{eq:positive-phase-sup}
\sup_{0<t<1}\Phi_+(t)
=
\left[
\inf_{\pi/2<\mu<\pi}
\left(
\mu-\frac{\cot\mu}{c}
\right)
\right]^{-1}
=
\frac2\pi.
\end{equation}

The corresponding contributions are
\[
\begin{aligned}
\mathcal A_n^P(c)&=\int_0^1t(1-t)w_1^+(a)\,\Phi_+(t)^{4n-2}\,dt,\\
\mathcal A_n^Q(c)&=\int_0^1t\,v_1(a)\,\Phi_+(t)^{4n-2}\,dt.
\end{aligned}
\]
Their amplitudes are nonnegative and bounded on $(0,1)$ and are
strictly positive on every compact subinterval of $(0,1)$. Although
the supremum of $\Phi_+$ is attained only in the limit $t\downarrow0$,
this causes no difficulty: for every $\varepsilon>0$ one may choose
$t_\varepsilon>0$ such that
\(\Phi_+(t_\varepsilon)>2/\pi-\varepsilon\). By continuity there is a compact interval
$I_\varepsilon\Subset(0,1)$ containing $t_\varepsilon$ on which
$\Phi_+>2/\pi-2\varepsilon$, while the corresponding amplitude has
a strictly positive minimum. This gives the required lower
exponential bound. The upper bound follows from
$\Phi_+(t)\le2/\pi$. Hence the root form of Laplace's principle yields
\[
\lim_{n\to\infty}
\left(\mathcal A_n^P(c)\right)^{1/(4n)}
=
\lim_{n\to\infty}
\left(\mathcal A_n^Q(c)\right)^{1/(4n)}
=
\frac2\pi.
\]
This proves \eqref{eq:PQ-positive-spectral-rate}.

For the negative pole, put \(\Phi_-(t):=(1-t)/\kappa(a)\). Using
\(a=ct/(1-t)=\coth\kappa/\kappa\), we obtain
\begin{equation}
\label{eq:negative-phase}
\Phi_-(t)
=
\frac1{
\kappa+\dfrac{\coth\kappa}{c}
}.
\end{equation}
Thus \(\sup_{0<t<1}\Phi_-(t)=1/R_-(c)\), where
\(R_-(c)=\min_{\kappa>0}(\kappa+\coth\kappa/c)\). Now
\(\frac{d}{d\kappa}(\kappa+\coth\kappa/c)=1-1/(c\sinh^2\kappa)\).
Since this derivative tends to $-\infty$ as $\kappa\downarrow0$ and
to $1$ as $\kappa\to\infty$, it has exactly one zero. Thus the
minimum is attained at a unique point $x_c>0$, characterized by
\begin{equation}
\label{eq:xc-equation}
\sinh^2x_c=\frac1c.
\end{equation}
Furthermore,
\(\frac{d^2}{d\kappa^2}(\kappa+\coth\kappa/c)=2\coth\kappa/(c\sinh^2\kappa)>0\),
so the minimum is strict and nondegenerate.

The two negative-pole contributions can be written as
\[
\begin{aligned}
\mathcal B_n^P(c)&=\int_0^1t(1-t)w_-(a)\,\Phi_-(t)^{4n-2}\,dt,\\
\mathcal B_n^Q(c)&=\int_0^1t\,v_-(a)\,\Phi_-(t)^{4n-2}\,dt.
\end{aligned}
\]
The unique maximizer of $\Phi_-$ corresponds to the finite value
$\kappa=x_c$ and therefore to a point $t_c\in(0,1)$. Both amplitudes
are continuous and strictly positive at $t_c$. Hence the interior
form of Laplace's principle gives
\[
\lim_{n\to\infty}
\left(\mathcal B_n^P(c)\right)^{1/(4n)}
=
\lim_{n\to\infty}
\left(\mathcal B_n^Q(c)\right)^{1/(4n)}
=
\frac1{R_-(c)}.
\]
This proves \eqref{eq:PQ-negative-spectral-rate}.

It remains to estimate the remainders.

For $P_{2n}$, \eqref{eq:G-odd-coefficients},
\eqref{eq:second-positive-pole}, and \(\sum_{j\geq2}w_j^+(a)\leq1\) give
\[
\left|
\sum_{j\geq2}
\frac{w_j^+(a)}{\mu_j(a)^{4n-2}}
\right|
\leq
\left(\frac{2}{3\pi}\right)^{4n-2}.
\]
Consequently,
\[
|\mathcal E_n^P(c)|
\leq
\left(\frac{2}{3\pi}\right)^{4n-2}
\int_0^1t(1-t)^{4n-1}\,dt.
\]
Since \(\int_0^1t(1-t)^{4n-1}\,dt=1/((4n)(4n+1))\), it follows that
\begin{equation}
\label{eq:EP-bound}
\limsup_{n\to\infty}
|\mathcal E_n^P(c)|^{1/(4n)}
\leq
\frac{2}{3\pi}.
\end{equation}

For $Q_{2n}$, the remainder is precisely the contribution of
$H_a$. By \eqref{eq:Ha-Cauchy},
\(\lvert[u^{2n-1}]H_a(u)\rvert\leq M_r r^{-(2n-1)}\) uniformly in $a>0$.
Hence
\[
|\mathcal E_n^Q(c)|
\leq
M_r r^{-(2n-1)}
\int_0^1t(1-t)^{4n-2}\,dt.
\]
Since \(\int_0^1t(1-t)^{4n-2}\,dt=1/((4n-1)(4n))\), we obtain
\[
\limsup_{n\to\infty}
|\mathcal E_n^Q(c)|^{1/(4n)}
\leq
r^{-1/2}.
\]
Because this holds for every \(\pi^2<r<9\pi^2/4\), letting $r\uparrow9\pi^2/4$ yields
\begin{equation}
\label{eq:EQ-bound}
\limsup_{n\to\infty}
|\mathcal E_n^Q(c)|^{1/(4n)}
\leq
\frac{2}{3\pi}.
\end{equation}

Combining \eqref{eq:EP-bound} and \eqref{eq:EQ-bound} gives
\[
\limsup_{n\to\infty}|\mathcal E_n^P(c)|^{1/(4n)}
\leq\frac{2}{3\pi}<\frac2\pi,
\qquad
\limsup_{n\to\infty}|\mathcal E_n^Q(c)|^{1/(4n)}
\leq\frac{2}{3\pi}<\frac2\pi.
\]
This proves \eqref{eq:PQ-spectral-remainder} and completes the proof.
\end{proof}

\begin{proposition}[Common limit of the real zeros]
\label[proposition]{prop:PQ-real-zero-limit}
Let $\rho_n,\sigma_n>0$ be defined by
\(P_{2n}(-\rho_n)=0\) and \(Q_{2n}(-\sigma_n)=0\). There exists a constant $c_*>0$ such that
\[
\lim_{n\to\infty}\rho_n
=
\lim_{n\to\infty}\sigma_n
=
c_*.
\]
More precisely, let $x_*>0$ be the unique solution of
\begin{equation}
\label{eq:xstar-equation}
x_*+\sinh x_*\cosh x_*
=
\frac{\pi}{2},
\end{equation}
equivalently, \(x_*+\tfrac12\sinh(2x_*)=\pi/2\). Then
\begin{equation}
\label{eq:cstar-definition}
c_*=\frac{1}{\sinh^2x_*}
=
1.897670091183219518\ldots,
\end{equation}
and consequently
\[
\begin{cases}
\displaystyle
\lim_{n\to\infty}
\bigl(\text{the real zero of }P_{2n}\bigr)
=-c_*,\\[1ex]
\displaystyle
\lim_{n\to\infty}
\bigl(\text{the real zero of }Q_{2n}\bigr)
=-c_*.
\end{cases}
\]
\end{proposition}

\begin{proof}
By \eqref{eq:xc-equation}, \(c=1/\sinh^2x_c\). Substituting this into \eqref{eq:Rminus-definition} gives
\begin{align}
R_-(c)
&=
x_c+\frac{\coth x_c}{c}
\nonumber\\
&=
x_c+\sinh x_c\cosh x_c.
\label{eq:Rminus-xc}
\end{align}

Consider \(F(x):=x+\sinh x\cosh x\) for $x>0$. Then
\[
F'(x)=1+\cosh^2x+\sinh^2x=2\cosh^2x>0,
\qquad
\lim_{x\downarrow0}F(x)=0,
\qquad
\lim_{x\to\infty}F(x)=\infty.
\]
Hence $F$ is strictly increasing on $(0,\infty)$ and there exists a unique
$x_*>0$ such that \(F(x_*)=\pi/2\), which is precisely \eqref{eq:xstar-equation}.

On the other hand, the map \(x\mapsto1/\sinh^2x\) is strictly decreasing
from $(0,\infty)$ onto $(0,\infty)$. Therefore
\(c_*:=1/\sinh^2x_*\) is the unique positive value of $c$ for which
\(R_-(c_*)=\pi/2\); numerically,
\(c_*=1.897670091183219518\ldots\). Furthermore, since
\(c=1/\sinh^2x_c\) is strictly decreasing as a function of $x_c$, we have
\(c<c_*\iff x_c>x_*\) and \(c>c_*\iff x_c<x_*\). Since $F$ is strictly increasing and
$R_-(c)=F(x_c)$, it follows that
\begin{equation}
\label{eq:Rminus-threshold}
\begin{cases}
R_-(c)>\dfrac{\pi}{2}, & 0<c<c_*,\\[1.5ex]
R_-(c)=\dfrac{\pi}{2}, & c=c_*,\\[1.5ex]
R_-(c)<\dfrac{\pi}{2}, & c>c_*.
\end{cases}
\end{equation}

Fix first $0<c<c_*$. By
\eqref{eq:PQ-positive-spectral-rate},
\eqref{eq:PQ-negative-spectral-rate},
\eqref{eq:PQ-spectral-remainder}, and
\eqref{eq:Rminus-threshold},
\[
\frac2\pi
>
\max\left\{
\frac1{R_-(c)},
\frac{2}{3\pi}
\right\}.
\]
Thus the contribution of the smallest positive pole dominates both
the negative-pole contribution and the remainder exponentially.
Consequently,
\[
\frac{
\mathcal B_n^P(c)+|\mathcal E_n^P(c)|
}{
\mathcal A_n^P(c)
}
\longrightarrow0,
\qquad
\frac{
\mathcal B_n^Q(c)+|\mathcal E_n^Q(c)|
}{
\mathcal A_n^Q(c)
}
\longrightarrow0.
\]
Hence, for all sufficiently large $n$,
\begin{equation}
\label{eq:PQ-sign-below-cstar}
\begin{cases}
P_{2n}(-c)>0,\\
Q_{2n}(-c)>0.
\end{cases}
\end{equation}

Now fix $c>c_*$. Then \(1/R_-(c)>2/\pi>2/(3\pi)\). Hence the negative-pole contribution dominates both the smallest
positive-pole contribution and the remainder exponentially. Thus
\[
\frac{
\mathcal A_n^P(c)+|\mathcal E_n^P(c)|
}{
\mathcal B_n^P(c)
}
\longrightarrow0,
\qquad
\frac{
\mathcal A_n^Q(c)+|\mathcal E_n^Q(c)|
}{
\mathcal B_n^Q(c)
}
\longrightarrow0.
\]
Consequently, for all sufficiently large $n$,
\begin{equation}
\label{eq:PQ-sign-above-cstar}
\begin{cases}
P_{2n}(-c)<0,\\
Q_{2n}(-c)<0.
\end{cases}
\end{equation}

Let $\varepsilon>0$ be such that $c_*-\varepsilon>0$. Applying
\eqref{eq:PQ-sign-below-cstar} with
$c=c_*-\varepsilon$ and
\eqref{eq:PQ-sign-above-cstar} with
$c=c_*+\varepsilon$, we obtain, for all sufficiently large $n$,
\[
\begin{cases}
P_{2n}(-c_*+\varepsilon)>0,
&
P_{2n}(-c_*-\varepsilon)<0,\\[1ex]
Q_{2n}(-c_*+\varepsilon)>0,
&
Q_{2n}(-c_*-\varepsilon)<0.
\end{cases}
\]
Since $P_{2n}$ and $Q_{2n}$ are strictly increasing on $\mathbb R$
and have unique real zeros $-\rho_n$ and $-\sigma_n$, respectively,
we obtain \(-c_*-\varepsilon<-\rho_n<-c_*+\varepsilon\) and
\(-c_*-\varepsilon<-\sigma_n<-c_*+\varepsilon\). Equivalently,
\[
\begin{cases}
c_*-\varepsilon<\rho_n<c_*+\varepsilon,\\
c_*-\varepsilon<\sigma_n<c_*+\varepsilon.
\end{cases}
\]
Thus, for every sufficiently small $\varepsilon>0$, these inequalities
hold for all sufficiently large $n$. For an arbitrary
$\varepsilon>0$, the conclusion is immediate as well, since one may
replace it by $\min\{\varepsilon,c_*/2\}$. Therefore
\[
\lim_{n\to\infty}\rho_n
=
\lim_{n\to\infty}\sigma_n
=
c_*.
\]
The corresponding real zeros consequently converge to $-c_*$.
\end{proof}

\subsection{Differential recurrences}

\begin{proposition}[Differential recurrences for $P_n$ and $Q_n$]
\label[proposition]{prop:PQ-polynomial-recurrences}
For $n\geq2$, the polynomials $P_n$ and $Q_n$ satisfy
\begin{equation}
\begin{dcases}
\displaystyle
P_n(y)=\frac{d}{dy}\left[
\frac{y}{1-y}
\left\{
(-1)^{n-1}\frac{2^{2n}(2^{2n}-1)}{2n}B_{2n}
-y^2\bigl(yP_{n-1}'(y)+2P_{n-1}(y)\bigr)
\right\}
\right],
\\[12pt]
\displaystyle
Q_n(y)=\frac{d}{dy}\left[
\frac{y}{1-y}
\left\{
(-1)^{n-1}E_{2n-2}
-y^2\bigl(yQ_{n-1}'(y)+Q_{n-1}(y)\bigr)
\right\}
\right],
\end{dcases}
\label{eq:PQ-polynomial-recurrences}
\end{equation}
with $P_1(y)=Q_1(y)=1$.
\end{proposition}

\begin{proof}
Write \(P_n(y)=\sum_{p=1}^{n}a_{n,p}y^{p-1}\) and
\(Q_n(y)=\sum_{p=1}^{n}b_{n,p}y^{p-1}\), where
$a_{n,p}=p!\,T(2n-p,p)$ and $b_{n,p}=p!\,S(2n-p-1,p-1)$. From
\[
T(m+1,p)=pT(m,p-1)+pT(m,p+1),
\qquad
S(m+1,p)=pS(m,p-1)+(p+1)S(m,p+1),
\]
we obtain
\[
\begin{dcases}
\displaystyle
\frac{a_{n,p}}{p}-\frac{a_{n,p+1}}{p+1}
=p\,a_{n-1,p-1},
\\[8pt]
\displaystyle
\frac{b_{n,p}}{p}-\frac{b_{n,p+1}}{p+1}
=(p-1)b_{n-1,p-1},
\end{dcases}
\]
where $a_{n,n+1}=b_{n,n+1}=0$. Put
\(\mathcal P_n(y):=y^{-1}\int_0^yP_n(t)\,dt\) and
\(\mathcal Q_n(y):=y^{-1}\int_0^yQ_n(t)\,dt\). Then
\[
\mathcal P_n(y)=\sum_{p=1}^{n}\frac{a_{n,p}}{p}y^{p-1},
\qquad
\mathcal Q_n(y)=\sum_{p=1}^{n}\frac{b_{n,p}}{p}y^{p-1}.
\]
Multiplying the preceding recurrences by $y^{p-1}$ and summing over
$1\leq p\leq n$ gives
\[
\begin{dcases}
\displaystyle
(y-1)\mathcal P_n(y)+P_n(0)
=
y^2\bigl(yP_{n-1}'(y)+2P_{n-1}(y)\bigr),
\\[6pt]
\displaystyle
(y-1)\mathcal Q_n(y)+Q_n(0)
=
y^2\bigl(yQ_{n-1}'(y)+Q_{n-1}(y)\bigr).
\end{dcases}
\]
By \cref{prop:PQ-edge-coefficients},
\[
\begin{dcases}
\displaystyle
P_n(0)=(-1)^{n-1}\frac{2^{2n}(2^{2n}-1)}{2n}B_{2n},
\\[8pt]
\displaystyle
Q_n(0)=(-1)^{n-1}E_{2n-2}.
\end{dcases}
\]
Consequently,
\[
\begin{dcases}
\displaystyle
\mathcal P_n(y)
=
\frac{
(-1)^{n-1}\dfrac{2^{2n}(2^{2n}-1)}{2n}B_{2n}
-y^2\bigl(yP_{n-1}'(y)+2P_{n-1}(y)\bigr)}
{1-y},
\\[14pt]
\displaystyle
\mathcal Q_n(y)
=
\frac{
(-1)^{n-1}E_{2n-2}
-y^2\bigl(yQ_{n-1}'(y)+Q_{n-1}(y)\bigr)}
{1-y}.
\end{dcases}
\]
Finally, \(P_n(y)=\frac{d}{dy}(y\mathcal P_n(y))\) and
\(Q_n(y)=\frac{d}{dy}(y\mathcal Q_n(y))\), which gives \eqref{eq:PQ-polynomial-recurrences}.
\end{proof}

\subsection{Strict log-concavity and unimodality}

\begin{proposition}[Strict log-concavity of the coefficients]
Let \(n\in\mathbb N\), and write
\(P_n(y)=\sum_{p=1}^{n}a_{n,p}y^{p-1}\) and
\(Q_n(y)=\sum_{p=1}^{n}b_{n,p}y^{p-1}\), where
\(a_{n,p}=p!\,T(2n-p,p)\) and
\(b_{n,p}=p!\,S(2n-p-1,p-1)\). Then, for \(2\le p\le n-1\),
\[
\begin{cases}
a_{n,p}^{\,2}>a_{n,p-1}a_{n,p+1},\\
b_{n,p}^{\,2}>b_{n,p-1}b_{n,p+1}.
\end{cases}
\]
In particular, the coefficient sequences of \(P_n\) and \(Q_n\) are unimodal.
\end{proposition}

\begin{proof}
Put \(A_{n,p}=a_{n,p}/p\) and \(B_{n,p}=b_{n,p}/p\). From
\(T(m+1,p)=pT(m,p-1)+pT(m,p+1)\) and
\(S(m+1,p)=pS(m,p-1)+(p+1)S(m,p+1)\), we obtain
\begin{equation}
\begin{cases}
A_{n,p}-A_{n,p+1}=p(p-1)A_{n-1,p-1},\\
B_{n,p}-B_{n,p+1}=(p-1)^2B_{n-1,p-1}.
\end{cases}
\label{eq:AB-recurrence}
\end{equation}
Since \(A_{n,n+1}=B_{n,n+1}=0\), summation of \eqref{eq:AB-recurrence} gives
\begin{equation}
\begin{cases}
A_{n,p}=\displaystyle\sum_{j=p}^{n}j(j-1)A_{n-1,j-1},\\[6pt]
B_{n,p}=\displaystyle\sum_{j=p}^{n}(j-1)^2B_{n-1,j-1}.
\end{cases}
\label{eq:AB-tail-sums}
\end{equation}

We use the elementary fact that tail sums preserve log-concavity. Let
\((x_j)\) be a positive log-concave sequence and \(X_p=\sum_{j=p}^{N}x_j\).
Since \(X_{p-1}=x_{p-1}+X_p\) and \(X_{p+1}=X_p-x_p\),
\[
X_p^2-X_{p-1}X_{p+1}
=X_p(x_p-x_{p-1})+x_{p-1}x_p.
\]
If \(x_{p-1}\le x_p\), this is nonnegative. If \(x_{p-1}>x_p\), put
\(r=x_p/x_{p-1}<1\). Log-concavity implies \(x_{p+k}\le x_pr^k\), and hence
\(X_p\le x_p/(1-r)=x_{p-1}x_p/(x_{p-1}-x_p)\), which again gives
\(X_p^2\ge X_{p-1}X_{p+1}\).

We now proceed by induction on \(n\). The assertion is immediate for \(n=1\).
Assume that \((A_{n-1,p})\) and \((B_{n-1,p})\) are log-concave. The sequences
\(j(j-1)\) and \((j-1)^2\), \(j\ge2\), are log-concave, and therefore so are
\[
\begin{cases}
\bigl(j(j-1)A_{n-1,j-1}\bigr)_{j=2}^{n},\\
\bigl((j-1)^2B_{n-1,j-1}\bigr)_{j=2}^{n}.
\end{cases}
\]
By \eqref{eq:AB-tail-sums} and the preceding tail-sum property,
\((A_{n,p})_{p=2}^{n}\) and \((B_{n,p})_{p=2}^{n}\) are log-concave.
Moreover, \eqref{eq:AB-recurrence} with \(p=1\) gives
\(A_{n,1}=A_{n,2}\) and \(B_{n,1}=B_{n,2}\), so both full sequences are log-concave.

Since \(a_{n,p}=pA_{n,p}\) and \(b_{n,p}=pB_{n,p}\), for \(2\le p\le n-1\),
\[
\begin{cases}
a_{n,p}^{\,2}
=p^2A_{n,p}^{\,2}
\ge p^2A_{n,p-1}A_{n,p+1}
>(p-1)(p+1)A_{n,p-1}A_{n,p+1}
=a_{n,p-1}a_{n,p+1},\\[4pt]
b_{n,p}^{\,2}
=p^2B_{n,p}^{\,2}
\ge p^2B_{n,p-1}B_{n,p+1}
>(p-1)(p+1)B_{n,p-1}B_{n,p+1}
=b_{n,p-1}b_{n,p+1}.
\end{cases}
\]
Thus both coefficient sequences are strictly log-concave.

By \cref{prop:PQ-edge-coefficients}, for \(n\ge2\),
\begin{equation}
\begin{cases}
\dfrac{a_{n,2}}{a_{n,1}}=2,\\[6pt]
\dfrac{b_{n,2}}{b_{n,1}}=2,
\end{cases}
\label{eq:first-ratios}
\end{equation}
whereas at the opposite endpoint
\begin{equation}
\begin{cases}
\dfrac{a_{n,n}}{a_{n,n-1}}
=\dfrac{3n}{n^2-1},\\[8pt]
\dfrac{b_{n,n}}{b_{n,n-1}}
=\dfrac{6}{2n-1}.
\end{cases}
\label{eq:last-ratios}
\end{equation}
For \(n\ge4\), both quantities in \eqref{eq:last-ratios} are \(<1\), whereas
those in \eqref{eq:first-ratios} are \(>1\). Strict log-concavity gives
\[
\begin{cases}
\dfrac{a_{n,p+1}}{a_{n,p}}
<
\dfrac{a_{n,p}}{a_{n,p-1}},\\[8pt]
\dfrac{b_{n,p+1}}{b_{n,p}}
<
\dfrac{b_{n,p}}{b_{n,p-1}},
\end{cases}
\qquad 2\le p\le n-1.
\]
Thus, in both sequences, the successive quotients can pass through \(1\) at
most once. The coefficients therefore increase up to their maximum and
decrease thereafter. The cases \(n=1,2,3\) are immediate.
\end{proof}

\subsection{Generating functions}

\begin{proposition}[Generating functions for $P_n$ and $Q_n$]
\label[proposition]{prop:PQ-generating-functions}
For $y\in\mathbb C$ and $|z|$ sufficiently small,
\begin{equation}
\begin{dcases}
\displaystyle
\sum_{n=1}^{\infty}P_n(y)\frac{z^{2n}}{(2n)!}
=
\int_0^z
\frac{\tan x}
{\bigl(1-y(z-x)\tan x\bigr)^2}\,dx,
\\[12pt]
\displaystyle
\sum_{n=1}^{\infty}Q_n(y)\frac{z^{2n-1}}{(2n-1)!}
=
\int_0^z
\frac{\sec x}
{\bigl(1-y(z-x)\tan x\bigr)^2}\,dx.
\end{dcases}
\label{eq:PQ-generating-functions}
\end{equation}
\end{proposition}

\begin{proof}
By definition,
\(P_n(y)=\sum_{p=1}^{n}p!\,T(2n-p,p)y^{p-1}\) and
\(Q_n(y)=\sum_{p=1}^{n}p!\,S(2n-p-1,p-1)y^{p-1}\).
Setting respectively $m=2n-p$ and $m=2n-p-1$, and using the vanishing of \(T(m,p)\) and \(S(m,p-1)\) outside the
admissible range and parity, we obtain
\[
\begin{dcases}
\displaystyle
\sum_{n=1}^{\infty}P_n(y)\frac{z^{2n}}{(2n)!}
=
\sum_{p=1}^{\infty}\sum_{m=p}^{\infty}
p!\,T(m,p)y^{p-1}\frac{z^{m+p}}{(m+p)!},
\\[12pt]
\displaystyle
\sum_{n=1}^{\infty}Q_n(y)\frac{z^{2n-1}}{(2n-1)!}
=
\sum_{p=1}^{\infty}\sum_{m=p-1}^{\infty}
p!\,S(m,p-1)y^{p-1}\frac{z^{m+p}}{(m+p)!}.
\end{dcases}
\]

Now \(p!/(m+p)!=(p/m!)\int_0^1t^m(1-t)^{p-1}\,dt\). Therefore,
\[
\begin{dcases}
\displaystyle
\sum_{n=1}^{\infty}P_n(y)\frac{z^{2n}}{(2n)!}
=
z\int_0^1
\sum_{p=1}^{\infty}
p\bigl(zy(1-t)\bigr)^{p-1}
\sum_{m=p}^{\infty}
T(m,p)\frac{(zt)^m}{m!}\,dt,
\\[14pt]
\displaystyle
\sum_{n=1}^{\infty}Q_n(y)\frac{z^{2n-1}}{(2n-1)!}
=
z\int_0^1
\sum_{p=1}^{\infty}
p\bigl(zy(1-t)\bigr)^{p-1}
\sum_{m=p-1}^{\infty}
S(m,p-1)\frac{(zt)^m}{m!}\,dt.
\end{dcases}
\]

Using
\[
\begin{dcases}
\displaystyle
\tan^p x
=
\sum_{m=p}^{\infty}
T(m,p)\frac{x^m}{m!},
\\[10pt]
\displaystyle
\sec x\,\tan^{p-1}x
=
\sum_{m=p-1}^{\infty}
S(m,p-1)\frac{x^m}{m!},
\end{dcases}
\]
we get
\[
\begin{dcases}
\displaystyle
\sum_{n=1}^{\infty}P_n(y)\frac{z^{2n}}{(2n)!}
=
z\int_0^1
\tan(zt)
\sum_{p=1}^{\infty}
p\bigl(zy(1-t)\tan(zt)\bigr)^{p-1}\,dt,
\\[14pt]
\displaystyle
\sum_{n=1}^{\infty}Q_n(y)\frac{z^{2n-1}}{(2n-1)!}
=
z\int_0^1
\sec(zt)
\sum_{p=1}^{\infty}
p\bigl(zy(1-t)\tan(zt)\bigr)^{p-1}\,dt.
\end{dcases}
\]

Since \(\sum_{p=1}^{\infty}p r^{p-1}=(1-r)^{-2}\) for $|r|<1$,
it follows that
\[
\begin{dcases}
\displaystyle
\sum_{n=1}^{\infty}P_n(y)\frac{z^{2n}}{(2n)!}
=
z\int_0^1
\frac{\tan(zt)}
{\bigl(1-zy(1-t)\tan(zt)\bigr)^2}\,dt,
\\[14pt]
\displaystyle
\sum_{n=1}^{\infty}Q_n(y)\frac{z^{2n-1}}{(2n-1)!}
=
z\int_0^1
\frac{\sec(zt)}
{\bigl(1-zy(1-t)\tan(zt)\bigr)^2}\,dt.
\end{dcases}
\]

Finally, the substitution \(x=zt\) yields
\[
\begin{dcases}
\displaystyle
\sum_{n=1}^{\infty}P_n(y)\frac{z^{2n}}{(2n)!}
=
\int_0^z
\frac{\tan x}
{\bigl(1-y(z-x)\tan x\bigr)^2}\,dx,
\\[14pt]
\displaystyle
\sum_{n=1}^{\infty}Q_n(y)\frac{z^{2n-1}}{(2n-1)!}
=
\int_0^z
\frac{\sec x}
{\bigl(1-y(z-x)\tan x\bigr)^2}\,dx,
\end{dcases}
\]
which proves \eqref{eq:PQ-generating-functions}.
\end{proof}

\section{Complex zeros and the common zero attractor}
\label[section]{sec:complex-zero-attractor}

The generating functions in \cref{prop:PQ-generating-functions} reduce the
complex-zero problem for both polynomial families to the same singularity
geometry.  The purpose of this section is to make that geometry global and
to prove that it is exactly the limiting geometry of the zeros.  The central
result is that the zero sets of both $P_n$ and $Q_n$ converge, in the
Hausdorff sense, to one and the same compact Jordan curve.  We also determine
the limiting zero distribution on this curve.

We use throughout the set conventions stated in the Notation paragraph.  In
particular, $A^\circ$, $\overline A$, and $\partial A$ denote interior,
closure, and boundary; $L\Subset U$ denotes compact containment;
$\operatorname{Comp}_0$ selects the connected component containing the
origin; and $\mathbb D_r$ is the open disc of radius $r$ centered at $0$.
The Riemann sphere $\widehat{\mathbb C}$ is used when the exterior domain is
viewed as a domain containing $\infty$.  The geometric roles of the sets
introduced below are: $\mathcal S$ is the central strip in the critical
$x$-plane, $\Omega$ is the principal critical component, $E$ is the exterior
domain in the $y$-plane, $K=\widehat{\mathbb C}\setminus E$ is its compact
complement, and $\Gamma=\partial E=\partial K$ is the common boundary and
zero attractor.

Throughout this section we put
\begin{equation}
\rho:=\frac{\pi}{2},
\qquad
R:=\rho^2=\frac{\pi^2}{4},
\qquad
\mathfrak c:=\frac{4}{R}=\frac{16}{\pi^2}.
\label{eq:complex-rho-R-c}
\end{equation}
For $y,z,x\in\mathbb C$ we use the common denominator
\begin{equation}
D_y(x,z):=1-y(z-x)\tan x.
\label{eq:complex-D-def}
\end{equation}

\subsection{The principal critical map and the curve $\Gamma$}

The algebraic critical-point equations attached to
\eqref{eq:complex-D-def} admit a particularly simple global
parametrization.

\begin{lemma}[Critical-point parametrization]
\label[lemma]{lem:complex-critical-parametrization}
Let $y\neq0$.  A pair $(x,z)$ with $\tan x$ finite satisfies
\[
D_y(x,z)=0,
\qquad
\partial_xD_y(x,z)=0,
\]
if and only if
\begin{equation}
y=\csc^2x,
\qquad
z=f(x),
\qquad
f(x):=x+\sin x\cos x
=x+\frac12\sin(2x).
\label{eq:complex-f-def}
\end{equation}
At such a point,
\begin{equation}
\partial_x^2D_y(x,z)=2y,
\qquad
\partial_zD_y(x,z)=-y\tan x\neq0.
\label{eq:complex-critical-nondegenerate}
\end{equation}
Moreover,
\begin{equation}
f'(x)=2\cos^2x.
\label{eq:complex-f-derivative}
\end{equation}
\end{lemma}

\begin{proof}
From $D_y(x,z)=0$,
\[
z-x=\frac{1}{y\tan x}.
\]
Substitution into
\[
\partial_xD_y(x,z)
=y\tan x-y(z-x)\sec^2x
\]
gives
\[
y\tan^2x=\sec^2x,
\qquad
\sin^2x=\frac1y.
\]
Thus $y=\csc^2x$, and then
\[
z-x=\frac{1}{y\tan x}=\sin x\cos x,
\]
which is \eqref{eq:complex-f-def}.  Differentiating once more gives
\[
\partial_x^2D_y
=2y\sec^2x-2y(z-x)\sec^2x\tan x.
\]
Since $(z-x)\tan x=1/y$ at a critical point,
\eqref{eq:complex-critical-nondegenerate} follows.  Finally,
\[
f'(x)=1+\cos(2x)=2\cos^2x.
\]
\end{proof}

The function $f$ contains the global information that is obscured if one
works directly with the multivalued expression involving $\arcsin$.  We now
single out the component that is analytically connected to the Taylor germ
of the generating functions at $z=0$.

\begin{proposition}[Principal critical domain]
\label[proposition]{prop:complex-principal-domain}
Put
\begin{equation}
\mathcal S:=\{x\in\mathbb C:|\Re x|<\rho\},
\qquad
\Omega:=\operatorname{Comp}_0\{x\in\mathcal S:|f(x)|<\rho\},
\label{eq:complex-Omega-def}
\end{equation}
where $\operatorname{Comp}_0$ denotes the connected component containing
$0$.  Then $\Omega$ is bounded, $f'$ has no zero in $\Omega$, and
\begin{equation}
f:\Omega\longrightarrow\mathbb D_\rho
:=\{w\in\mathbb C:\ |w|<\rho\}
\label{eq:complex-f-biholo}
\end{equation}
is biholomorphic.  In addition, $\Omega$ is the only connected component of
$\{|f|<\rho\}$ contained in $\mathcal S$.
\end{proposition}

\begin{proof}
Write $x=a+ib$.  On the two boundary lines of the strip,
\begin{equation}
\begin{aligned}
f(\rho+ib)
&=\rho+i\bigl(b-\sinh b\cosh b\bigr),\\
f(-\rho+ib)
&=-\rho+i\bigl(b-\sinh b\cosh b\bigr),
\end{aligned}
\label{eq:complex-strip-boundary-f}
\end{equation}
so $|f(x)|\geq\rho$ when $|\Re x|=\rho$.  Equality can occur there only
at $x=\pm\rho$: indeed, for
$h(b):=b-\sinh b\cosh b$,
\[
h(0)=0,
\qquad
h'(b)=1-\cosh(2b)=-2\sinh^2b\leq0,
\]
and $h(b)=0$ only for $b=0$.

We next show that $0$ is the only zero of $f$ in $\mathcal S$.  Since
\begin{equation}
\begin{aligned}
\Re f(a+ib)
&=a+\frac12\sin(2a)\cosh(2b),\\
\Im f(a+ib)
&=b+\frac12\cos(2a)\sinh(2b),
\end{aligned}
\label{eq:complex-f-real-imag}
\end{equation}
$\Re f(a+ib)$ has the same sign as $a$ for $0<|a|<\rho$.  Hence
$f(a+ib)=0$ forces $a=0$, and then
\[
f(ib)=i\left(b+\frac12\sinh(2b)\right)=0
\]
forces $b=0$.

Moreover,
\[
|\sin x\cos x|^2
=\frac14\bigl(\sin^2(2a)+\sinh^2(2b)\bigr),
\]
so $|f(a+ib)|\to\infty$ uniformly for $|a|\leq\rho$ as
$|b|\to\infty$.  Consequently every component of
$\{x\in\mathcal S:|f(x)|<\rho\}$ is bounded.  For the component $\Omega$,
continuity and \eqref{eq:complex-strip-boundary-f} give the useful boundary
facts
\begin{equation}
\partial\Omega\cap\mathcal S\subset\{|f|=\rho\},
\qquad
\partial\Omega\cap\partial\mathcal S=\{-\rho,\rho\}.
\label{eq:complex-Omega-boundary-facts}
\end{equation}
By \eqref{eq:complex-f-derivative}, the only zeros of $f'$ in the closed
central strip are $\pm\rho$; hence $f'$ has no zero in $\Omega$.

We now justify properness explicitly.  If $C\Subset\mathbb D_\rho$ and
$x_j\in f^{-1}(C)\cap\Omega$ had no convergent subsequence in $\Omega$,
boundedness of $\Omega$ would give a subsequence converging to
$x_0\in\partial\Omega$.  But \eqref{eq:complex-Omega-boundary-facts} would
then imply $|f(x_0)|=\rho$, contradicting $f(x_j)\in C$.  Thus
\[
f:\Omega\longrightarrow\mathbb D_\rho
\]
is a proper local biholomorphism.  Its image is nonempty, open and closed in
$\mathbb D_\rho$, hence is the whole disc; equivalently, $f$ is a finite
unbranched covering of $\mathbb D_\rho$.  The covering degree is the number
of preimages of any point.  Since $0$ has exactly one preimage in
$\mathcal S$, namely $x=0$, that degree is one.  Therefore
\eqref{eq:complex-f-biholo} is biholomorphic.

Finally, the same argument on any other component of
$\{x\in\mathcal S:|f(x)|<\rho\}$ would again give a proper covering of
$\mathbb D_\rho$ and therefore a preimage of $0$ in that component.  This is
impossible because $0$ is the unique zero of $f$ in $\mathcal S$.  Hence
$\Omega$ is the unique component.
\end{proof}

Let
\begin{equation}
\chi:\mathbb D_\rho\longrightarrow\Omega
\label{eq:complex-chi-def}
\end{equation}
be the inverse of $f$.  Since $f$ is odd, $\chi$ is odd.  The meromorphic
function
\[
w\longmapsto\csc^2\chi(w)
\]
is therefore even.  Hence there is a unique meromorphic function $Y$ on
$\mathbb D_R$ such that
\begin{equation}
Y(w^2)=\csc^2\chi(w),
\qquad |w|<\rho.
\label{eq:complex-Y-def}
\end{equation}
The point $s=0$ is a simple pole of $Y$, because
\[
f(x)=2x+O(x^3),
\qquad
\chi(w)=\frac{w}{2}+O(w^3),
\qquad
Y(s)=\frac4s+O(1).
\]

\begin{proposition}[The exterior map and the zero-attractor curve]
\label[proposition]{prop:complex-exterior-map}
The function
\[
Y:\mathbb D_R\longrightarrow\widehat{\mathbb C}
\]
is univalent as a meromorphic map.  Its image
\begin{equation}
E:=Y(\mathbb D_R)
\label{eq:complex-E-def}
\end{equation}
is a simply connected domain containing $\infty$.  Put
\begin{equation}
K:=\widehat{\mathbb C}\setminus E,
\qquad
\Gamma:=\partial K=\partial E.
\label{eq:complex-K-Gamma-def}
\end{equation}
Then $K\subset\mathbb C$ is compact and $\Gamma$ is a Jordan curve,
symmetric with respect to the real axis.  The inverse map
\begin{equation}
\Psi:=Y^{-1}:E\longrightarrow\mathbb D_R
\label{eq:complex-Psi-def}
\end{equation}
is conformal and satisfies, on the branch obtained by analytic continuation
from $y>1$,
\begin{equation}
\Psi(y)=\Phi(y)^2,
\qquad
\Phi(y):=
\arcsin\frac1{\sqrt y}
+\frac1{\sqrt y}\sqrt{1-\frac1y}.
\label{eq:complex-Psi-Phi}
\end{equation}
Consequently,
\begin{equation}
\Gamma
=
\left\{y:\ |\Psi(y)|=R\right\}
=
\left\{y:\ |\Phi(y)|=\rho\right\},
\label{eq:complex-Gamma-equimodular}
\end{equation}
where the second description is interpreted on the principal analytic
continuation determined by $E$.
\end{proposition}

\begin{proof}
Suppose first that $s_1,s_2\in\mathbb D_R$ and $Y(s_1)=Y(s_2)$.
Choose $w_j^2=s_j$ and put $x_j=\chi(w_j)\in\Omega$.  Then
\[
\csc^2x_1=\csc^2x_2,
\qquad\text{hence}\qquad
\sin^2x_1=\sin^2x_2.
\]
The complete set of solutions of the latter equation is
$x_1=\pm x_2+k\pi$, $k\in\mathbb Z$.  Since both points lie in the open
central strip, necessarily $k=0$.  Therefore
\[
x_1=\pm x_2,
\qquad
w_1=f(x_1)=\pm f(x_2)=\pm w_2,
\qquad
s_1=s_2.
\]
Thus $Y$ is univalent as a meromorphic map.  Its pole at $s=0$ shows that
$E=Y(\mathbb D_R)$ contains $\infty$; hence
$K=\widehat{\mathbb C}\setminus E$ is compact in $\mathbb C$.

We next establish the boundary extension.  Let $|w_0|=\rho$ with
$w_0\neq\pm\rho$, and let $w_j\to w_0$ from $\mathbb D_\rho$.  The points
$x_j:=\chi(w_j)$ remain in the bounded domain $\Omega$, so every subsequence
has a further subsequence converging to some $x_0\in\partial\Omega$.  By
\eqref{eq:complex-Omega-boundary-facts}, $x_0\in\mathcal S$ and
$f(x_0)=w_0$.  Since $f'(x_0)\neq0$, the level set $|f|=\rho$ is locally a
real-analytic arc through $x_0$, and $f$ maps the two sides conformally to
the two sides of $|w|=\rho$.  The side belonging to $\Omega$ is mapped to
$\mathbb D_\rho$.  Consequently the local inverse of $f$ near $w_0$ agrees
with $\chi$ on the disc side and gives an analytic continuation of $\chi$
through $w_0$.  In particular, the boundary limit is unique.

At $w_0=\rho$, the real inverse branch satisfies
$\chi(w)\uparrow\rho$ as $w\uparrow\rho$.  Writing $x=\rho+u$ gives
\begin{equation}
\begin{aligned}
f(\rho+u)
&=\rho+\frac23u^3-\frac2{15}u^5+O(u^7),\\
\csc^2(\rho+u)
&=1+u^2+\frac23u^4+O(u^6).
\end{aligned}
\label{eq:complex-boundary-Puiseux}
\end{equation}
The first relation has three local inverse branches.  For $w$ on the disc
side and sufficiently close to $\rho$, exactly one of them satisfies
$\Re x<\rho$; it is the continuation of the real branch and hence the
branch belonging to $\Omega$.  Thus $\chi$ has a unique cubic-root boundary
continuation at $w=\rho$.  The case $w=-\rho$ follows by oddness.  Because
the two points $w=\pm\rho$ correspond to the same value $s=R$, the even
composition in \eqref{eq:complex-Y-def} has a $2/3$-power continuation at
$s=R$ and extends continuously to $|s|=R$.

We now check injectivity on the boundary, including the only subtle case
created by the closed strip.  If two boundary points have the same
$Y$-value, their associated points $x_1,x_2\in\partial\Omega$ satisfy
$x_1=\pm x_2+k\pi$.  In the closed central strip, $k\neq0$ is possible only
when the points lie on opposite boundary lines.  By
\eqref{eq:complex-Omega-boundary-facts} those points are necessarily
$\pm\rho$, and both give the single boundary point $s=R$.  Hence the
boundary map $Y|_{|s|=R}$ is one-to-one.  The same argument shows that a
boundary value cannot equal an interior value: otherwise a point of
$\partial\Omega$ would be congruent up to sign and an integral multiple of
$\pi$ to a point of $\Omega$; the exceptional value $\csc^2x=1$ would force
an interior point to be $\pm\rho$, which is impossible.

It follows that $Y$ extends to a continuous injection
\[
Y:\overline{\mathbb D_R}\longrightarrow\widehat{\mathbb C}
\]
when the pole at $0$ is interpreted as the value $\infty$.  Therefore
$Y(\partial\mathbb D_R)=\partial E$ is a Jordan curve.  Complex conjugation
commutes with $f$, $\chi$ and $Y$, so the curve is symmetric with respect to
the real axis.

Finally, for $y>1$ put
\[
x=\arcsin(y^{-1/2})\in(0,\rho).
\]
Then $y=\csc^2x$ and
\[
f(x)=x+\sin x\cos x
=\arcsin\frac1{\sqrt y}
 +\frac1{\sqrt y}\sqrt{1-\frac1y}
=\Phi(y).
\]
By \eqref{eq:complex-Y-def},
$\Psi(y)=f(x)^2=\Phi(y)^2$ for $y>1$.  Analytic continuation on $E$ gives
\eqref{eq:complex-Psi-Phi}; the boundary identity
\eqref{eq:complex-Gamma-equimodular} follows from the continuous extension
just proved.
\end{proof}

Away from $x=\pm\rho$, the boundary parametrization is regular.  Indeed,
$f'(x)\neq0$ there and
\[
\frac{d}{dx}\csc^2x=-2\csc^2x\cot x,
\]
whose only zeros in the closed central strip occur at $x=\pm\rho$.  Thus
$y=1$ is the only nonregular point of the boundary parametrization.  If
$y=1+\delta$ and the branches are continued from $\delta>0$, write
$\delta=\tan^2u$ with $u=\arctan\sqrt\delta$.  Then
\begin{equation}
\begin{aligned}
\Phi(1+\delta)
&=\rho-u+\sin u\cos u
 =\rho-\frac23\delta^{3/2}+\frac45\delta^{5/2}
 -\frac67\delta^{7/2}+O(\delta^{9/2}),\\
\Psi(1+\delta)
&=R-\frac{2\pi}{3}\delta^{3/2}+O(\delta^{5/2}).
\end{aligned}
\label{eq:complex-Phi-near-one}
\end{equation}
Hence $\Gamma$ is real analytic away from $1$.  Since
$\Psi(y)-R$ is proportional to $-(y-1)^{3/2}$ to leading order, the two
boundary tangents at $1$ satisfy
\[
\arg(y-1)=\pm\frac\pi3.
\]
Thus $1$ is an algebraic corner of $\Gamma$.

\begin{proposition}[Real intersections of $\Gamma$]
\label[proposition]{prop:complex-real-intersections}
The curve $\Gamma$ meets the real axis exactly at
\begin{equation}
\Gamma\cap\mathbb R=\{1,-c_*\},
\label{eq:complex-real-intersections}
\end{equation}
where $c_*$ is the constant in \cref{prop:PQ-real-zero-limit}.
\end{proposition}

\begin{proof}
Let $x=a+ib\in\partial\Omega$.  Since
\[
\Im\bigl(\sin^2(a+ib)\bigr)
=\frac12\sin(2a)\sinh(2b),
\]
the quantity $\csc^2x$ can be real in the closed central strip only when
$b=0$, $a=0$, or $|a|=\rho$.  On the real interval $[-\rho,\rho]$,
\eqref{eq:complex-f-derivative} gives $f'(x)\geq0$ and
\[
f(-\rho)=-\rho,
\qquad
f(\rho)=\rho.
\]
Hence $|f(x)|=\rho$ occurs there only at $x=\pm\rho$, giving $y=1$.

For $x=it$,
\begin{equation}
y=\csc^2(it)=-\frac1{\sinh^2t},
\qquad
f(it)=i\bigl(t+\sinh t\cosh t\bigr).
\label{eq:complex-negative-axis-param}
\end{equation}
The boundary condition $|f(it)|=\rho$ is therefore
\[
t+\sinh t\cosh t=\frac\pi2.
\]
By \cref{prop:PQ-real-zero-limit}, this equation has the unique positive
solution $t=x_*$ and $c_*=\sinh^{-2}x_*$.  Hence the negative intersection
is $-c_*$.  The cases $|a|=\rho$ reduce, by
\eqref{eq:complex-strip-boundary-f}, to $b=0$ and again give $y=1$.
\end{proof}

\subsection{Dominant singularities and uniform coefficient asymptotics}

We now pass from the geometry to coefficient asymptotics.  For later local
analysis it is convenient to name the two $z$-plane germs from
\cref{prop:PQ-generating-functions}:
\begin{equation}
\begin{aligned}
F_P(y,z)&:=\int_0^z\frac{\tan x}{D_y(x,z)^2}\,dx,
&
F_Q(y,z)&:=\int_0^z\frac{\sec x}{D_y(x,z)^2}\,dx.
\end{aligned}
\label{eq:complex-z-germs}
\end{equation}
Thus $F_P(y,\cdot)$ is even and $F_Q(y,\cdot)$ is odd.  Removing this
symmetry by $s=z^2$, define
\begin{equation}
\mathcal P(y,s):=\sum_{n\geq1}P_n(y)\frac{s^n}{(2n)!},
\qquad
\mathcal Q(y,s):=\sum_{n\geq1}Q_n(y)\frac{s^n}{(2n-1)!}.
\label{eq:complex-s-generating-functions}
\end{equation}
For $s=z^2$ near the origin,
\begin{equation}
\mathcal P(y,z^2)=F_P(y,z),
\qquad
\mathcal Q(y,z^2)=zF_Q(y,z),
\label{eq:complex-s-z-relation}
\end{equation}
and therefore
\begin{equation}
\begin{aligned}
\mathcal P(y,s)
&=\int_0^{\sqrt s}
\frac{\tan x}{\bigl(1-y(\sqrt s-x)\tan x\bigr)^2}\,dx,\\
\mathcal Q(y,s)
&=\sqrt s\int_0^{\sqrt s}
\frac{\sec x}{\bigl(1-y(\sqrt s-x)\tan x\bigr)^2}\,dx.
\end{aligned}
\label{eq:complex-s-integral-representations}
\end{equation}
The parity in \eqref{eq:complex-s-z-relation} shows that these are
single-valued analytic germs in $s$.

We shall also use the standard dented domains of singularity analysis.  For
$s_0\neq0$, $\eta>0$ and $0<\theta<\pi/2$, put
\begin{equation}
\Delta(s_0;\eta,\theta)
:=\left\{s:\ |s|<(1+\eta)|s_0|,\ s\neq s_0,\
\left|\arg\left(\frac{s}{s_0}-1\right)\right|>\theta\right\}.
\label{eq:complex-Delta-domain}
\end{equation}
The next proposition identifies the two possible dominant obstructions and,
in the form needed below, records the uniform $\Delta$-analyticity of the
principal continuation.

\begin{proposition}[Two singularity regimes and uniform continuation]
\label[proposition]{prop:complex-two-singularity-regimes}
Let $y\notin\Gamma$.
\begin{enumerate}
\item If $y\in E$, the principal analytic continuation of
$\mathcal P(y,\cdot)$ and $\mathcal Q(y,\cdot)$ has a unique singularity of
smallest modulus at
\begin{equation}
s=\Psi(y),
\qquad
|\Psi(y)|<R.
\label{eq:complex-moving-singularity}
\end{equation}
It is produced by a nondegenerate quadratic coalescence and is of algebraic
type $(1-s/\Psi(y))^{-3/2}$.

\item If $y\in K^\circ$, no principal moving critical value occurs for
$|s|<R$.  The unique dominant singularity of the principal branch is the
fixed endpoint singularity
\begin{equation}
s=R,
\label{eq:complex-fixed-singularity}
\end{equation}
and it is logarithmic.
\end{enumerate}
Moreover, these assertions hold uniformly on compact subsets.  More
precisely, for every $L\Subset E$ there exist $\eta>0$ and
$0<\theta<\pi/2$, independent of $y\in L$, such that the two functions are
holomorphic in
\[
\Delta\bigl(\Psi(y);\eta,\theta\bigr),
\]
and for every $L\Subset K^\circ$ there exist $\eta>0$ and
$0<\theta<\pi/2$ such that they are holomorphic in
\[
\Delta(R;\eta,\theta).
\]
In both cases the indicated dominant singularity is uniformly separated
from every other obstruction to the chosen principal continuation.
\end{proposition}

\begin{proof}
We first make the contour-continuation step explicit.  As long as the
endpoint $z$ remains away from the poles of $\tan$ and $\sec$, both
numerators in \eqref{eq:complex-z-germs} are holomorphic in the $x$-plane.
The possible moving poles are the zeros of $D_y(x,z)$.  Notice that
\begin{equation}
D_y(0,z)=D_y(z,z)=1,
\label{eq:complex-D-endpoints-nonzero}
\end{equation}
so a denominator zero never collides with either endpoint of the contour.
If a zero is simple, the implicit-function theorem makes it a holomorphic
function of the parameters locally, and the integration path can be
transported by a small isotopy so as to avoid it.  On a compact parameter
set free of multiple zeros, only finitely many such zero branches meet a
fixed compact tube around the transported contour, and the isotopy may be
chosen uniformly.  Therefore a genuine obstruction to the principal
continuation before an endpoint reaches a trigonometric pole can occur only
when a denominator zero becomes multiple, that is,
\begin{equation}
D_y(x,z)=\partial_xD_y(x,z)=0.
\label{eq:complex-pinch-equations}
\end{equation}
This is the one-dimensional pinch mechanism used below; it also shows that
all separation statements obtained below give the required uniform
$\Delta$-domains.

Suppose first that $y\in E$.  Since $\Psi(y)\in\mathbb D_R$, choose
$z_0$ with $z_0^2=\Psi(y)$ and put $x_0=\chi(z_0)\in\Omega$; changing the
sign of $z_0$ changes $x_0$ to $-x_0$.  By construction,
\begin{equation}
y=\csc^2x_0,
\qquad
z_0=f(x_0),
\qquad
|z_0|<\rho.
\label{eq:complex-principal-critical-pair}
\end{equation}
Conversely, any solution of \eqref{eq:complex-pinch-equations} that is
reachable from the Taylor germ while $|z|<\rho$ must lie in the central
strip $\mathcal S$.  By \cref{lem:complex-critical-parametrization}, its
critical value is $z=f(x)$.  If $|z|<\rho$, then
$x\in\{x\in\mathcal S:|f(x)|<\rho\}=\Omega$ by
\cref{prop:complex-principal-domain}.  Hence the only reachable critical
points are $\pm x_0$, giving the pair of $z$-values $\pm z_0$ and therefore
the single $s$-value $s=\Psi(y)$.  The first fixed endpoint poles occur at
$z=\pm\rho$, so the moving singularity is strictly dominant.

For later use, note also that all algebraic solutions of the critical-point
equations are
\begin{equation}
x=\pm x_0+k\pi,
\qquad
z=\pm z_0+k\pi,
\qquad k\in\mathbb Z.
\label{eq:complex-all-critical-translates}
\end{equation}
For $k\neq0$,
\[
|\pm z_0+k\pi|\geq |k|\pi-|z_0|
\geq\pi-|z_0|>\rho,
\]
so the principal pair is isolated even among all algebraic critical values.
If $L\Subset E$, then
\[
q_L:=\max_{y\in L}\frac{|\Psi(y)|}{R}<1,
\qquad
|z_0(y)|\leq\rho\sqrt{q_L}<\rho.
\]
Consequently the principal points $\pm z_0(y)$ stay a uniform positive
distance from the fixed poles $\pm\rho$ and from every translated critical
value.  Choosing $\eta>0$ so small that
$(1+\eta)q_L<1$, compactness, the contour argument above and
\eqref{eq:complex-critical-nondegenerate} give a common
$\Delta(\Psi(y);\eta,\theta)$ and a uniform quadratic local model for every
$y\in L$.

Now let $y\in K^\circ$.  If a reachable multiple denominator zero existed
with $|z|<\rho$, then its $x$-coordinate would lie in $\mathcal S$ and
\cref{lem:complex-critical-parametrization} would give
\[
|f(x)|<\rho,
\qquad
x\in\Omega,
\qquad
 y=\csc^2x\in E,
\]
a contradiction.  Equality $|f(x)|=\rho$ cannot occur either.  Indeed, if
$x\in\mathcal S$ and $f'(x)\neq0$, the regular level arc $|f|=\rho$ bounds
the unique sublevel component $\Omega$, hence $x\in\partial\Omega$ and
$\csc^2x\in\Gamma$; if $f'(x)=0$, then $x=\pm\rho$ and again
$y=1\in\Gamma$.  Thus for $y\in K^\circ$ every reachable critical value is
strictly outside $|z|=\rho$.

This separation is uniform on compact subsets of $K^\circ$.  Otherwise
there would be $y_j$ in a compact $L\Subset K^\circ$ and reachable critical
points $x_j\in\mathcal S$ with $|f(x_j)|\downarrow\rho$.  The uniform growth
$|f(a+ib)|\to\infty$ for $|a|\leq\rho$ forces $(x_j)$ to be bounded, so a
subsequence converges to $x_\infty\in\overline{\mathcal S}$ with
$|f(x_\infty)|=\rho$ and
$y_j=\csc^2x_j\to\csc^2x_\infty\in\Gamma$, contradicting
$L\cap\Gamma=\varnothing$.  Hence the first obstruction of the principal
branch is the endpoint pole at $z=\pm\rho$, which becomes the single point
$s=R$.  The next trigonometric poles are a positive distance away, and the
same compact contour argument yields a uniform dented neighborhood
$\Delta(R;\eta,\theta)$.

Finally, \eqref{eq:complex-critical-nondegenerate} gives
$\partial_x^2D_y=2y\neq0$ and $\partial_zD_y\neq0$ at every principal
moving critical point, so the coalescence is quadratic.  The local integral
of a squared denominator at such a quadratic pinch has exponent $-3/2$;
the exact expansion is computed in
\cref{lem:complex-moving-expansion}.  The endpoint expansion at $s=R$ is
computed in \cref{lem:complex-fixed-expansion} and is logarithmic.
\end{proof}

We record the local singular expansions.  They will also show explicitly
that the leading amplitudes never vanish.

\begin{lemma}[Fixed endpoint expansion]
\label[lemma]{lem:complex-fixed-expansion}
Uniformly for $y$ in compact subsets of $K^\circ$,
\begin{equation}
\begin{aligned}
\mathcal P(y,s)
&=-\frac1{(1-y)^2}
\log\left(1-\frac{s}{R}\right)
+H_P(y,s)
+O\left(\left(1-\frac{s}{R}\right)
\log\left(1-\frac{s}{R}\right)\right),\\
\mathcal Q(y,s)
&=-\frac{\rho}{(1-y)^2}
\log\left(1-\frac{s}{R}\right)
+H_Q(y,s)
+O\left(\left(1-\frac{s}{R}\right)
\log\left(1-\frac{s}{R}\right)\right),
\end{aligned}
\label{eq:complex-fixed-local-expansion}
\end{equation}
as $s\to R$ in a dented neighborhood of $R$, where $H_P$ and $H_Q$ are
holomorphic in a full neighborhood of $s=R$.
\end{lemma}

\begin{proof}
Fix $L\Subset K^\circ$.  By
\cref{prop:complex-two-singularity-regimes}, a single dented neighborhood of
$s=R$ works for all $y\in L$.  Split the transported contour into a short
endpoint piece and a remaining piece.  The latter stays in a compact set on
which the integrands are jointly holomorphic in $(y,s)$, so it contributes
to the holomorphic functions $H_P$ and $H_Q$.

On the endpoint piece write
\[
z=\rho-\delta,
\qquad
x=z-u,
\qquad
\nu:=\delta+u.
\]
Uniformly for $y\in L$ and $u,\delta\to0$ in the chosen dented sector,
\begin{equation}
\tan(z-u)=\frac1\nu+O(\nu),
\qquad
\sec(z-u)=\frac1\nu+O(\nu),
\label{eq:complex-tan-sec-endpoint}
\end{equation}
and
\begin{equation}
D_y(z-u,z)
=\frac{\delta+(1-y)u}{\nu}+O_L(\nu^2).
\label{eq:complex-D-endpoint}
\end{equation}
By the uniform continuation statement in
\cref{prop:complex-two-singularity-regimes}, the short endpoint contours
may be chosen, after a finite subdivision of $L$ if necessary, so that
\[
|\delta+(1-y)u|\geq c_L\bigl(|\delta|+|u|\bigr)
\]
throughout the local analysis.  Hence inversion of
\eqref{eq:complex-D-endpoint} is uniform there, and the common singular
model for the two local kernels is
\[
M_y(\delta,u):=
\frac{\delta+u}{\bigl(\delta+(1-y)u\bigr)^2}.
\]
Put $a:=1-y$.  Since $L$ stays a positive distance from $y=1$, $a$ is
uniformly bounded away from zero.  For a fixed small $\varepsilon>0$,
\begin{equation}
\int_0^\varepsilon M_y(\delta,u)\,du
=-\frac1{a^2}\log\delta+G(y,\delta),
\label{eq:complex-endpoint-model-integral}
\end{equation}
where $G$ is holomorphic near $L\times\{0\}$.  Indeed, direct integration
gives
\[
\frac1{a^2}\log(\delta+a\varepsilon)
-\frac1{a^2}\log\delta
+\frac{a-1}{a^2}
\left(1-\frac{\delta}{\delta+a\varepsilon}\right).
\]
Subtracting the model from the exact kernels and using
\eqref{eq:complex-tan-sec-endpoint}--\eqref{eq:complex-D-endpoint} shows,
after integration, that the remaining nonholomorphic contribution is
$O_L(\delta\log\delta)$.  Thus
\[
F_P(y,z)
=-\frac1{(1-y)^2}\log\delta+\widetilde H_P(y,z)
+O_L(\delta\log\delta),
\]
while multiplication by the prefactor $z$ in
\eqref{eq:complex-s-z-relation} gives
\[
\mathcal Q(y,z^2)
=-\frac{\rho}{(1-y)^2}\log\delta+\widetilde H_Q(y,z)
+O_L(\delta\log\delta).
\]
Finally,
\begin{equation}
1-\frac{s}{R}
=1-\left(1-\frac\delta\rho\right)^2
=\frac{2\delta}{\rho}+O(\delta^2),
\label{eq:complex-s-delta-relation}
\end{equation}
so replacing $\log\delta$ by $\log(1-s/R)$ changes only the holomorphic
part, and $\delta\log\delta=O((1-s/R)\log(1-s/R))$.  This yields
\eqref{eq:complex-fixed-local-expansion}, uniformly for $y\in L$.
\end{proof}

\begin{lemma}[Moving critical expansion]
\label[lemma]{lem:complex-moving-expansion}
Let $L\Subset E$.  Cover $L$ by finitely many simply connected domains
$U_j\Subset E$.  On each $U_j$, choose the branches in
\eqref{eq:complex-Psi-Phi} by analytic continuation from $y>1$ along a
fixed path in $E$.  Then, uniformly on $L\cap U_j$,
\begin{equation}
\begin{aligned}
\mathcal P(y,s)
&=C_P(y)\left(1-\frac{s}{\Psi(y)}\right)^{-3/2}
+O_L\left(\left|1-\frac{s}{\Psi(y)}\right|^{-1}\right),\\
\mathcal Q(y,s)
&=C_Q(y)\left(1-\frac{s}{\Psi(y)}\right)^{-3/2}
+O_L\left(\left|1-\frac{s}{\Psi(y)}\right|^{-1}\right),
\end{aligned}
\label{eq:complex-moving-local-expansion}
\end{equation}
as $s\to\Psi(y)$ in the corresponding uniform dented domains, where
\begin{equation}
C_P(y)=
\frac{\pi\sqrt2\,(y-1)^{1/4}}{y^2\Psi(y)^{3/4}},
\qquad
C_Q(y)=
\frac{\pi\sqrt2\,(y-1)^{1/4}}{y^{3/2}\Psi(y)^{1/4}}.
\label{eq:complex-moving-C-amplitudes}
\end{equation}
All fractional powers in each formula are taken compatibly on the chosen
$U_j$.  In particular, neither amplitude vanishes.  Their branch-independent
squares are
\begin{equation}
C_P(y)^2=-\pi^2\frac{\Psi'(y)}{y^2\Psi(y)^2},
\qquad
C_Q(y)^2=-\pi^2\frac{\Psi'(y)}{y\Psi(y)}.
\label{eq:complex-moving-C-squares}
\end{equation}
\end{lemma}

\begin{proof}
Fix one of the domains $U_j$.  Let $z_0(y)$ be the compatible square root of
$\Psi(y)$ and put
\[
x_0(y):=\chi(z_0(y)),
\qquad
z_0=f(x_0),
\qquad
y=\csc^2x_0.
\]
For $y>1$ these are the positive real branches, with
$0<x_0<z_0<\rho$.  The real contour is then pinched at $x=x_0$ as
$z\uparrow z_0$, which fixes the sign of the leading coefficient; the
formula on $U_j$ is its analytic continuation.

Set
\[
\eta:=1-\frac z{z_0},
\qquad
\xi:=x-x_0.
\]
Using \eqref{eq:complex-critical-nondegenerate} and
$z-z_0=-z_0\eta$, Taylor expansion at $(x_0,z_0)$ gives, uniformly on
compact subsets of $U_j$,
\begin{equation}
D_y(x,z)
=y\bigl(\xi^2+z_0\tan x_0\,\eta\bigr)
+O\bigl(\xi^3+|\eta\xi|+\eta^2\bigr).
\label{eq:complex-D-Morse-local}
\end{equation}
The analytic Morse lemma straightens the local contour and the quadratic
part simultaneously.  The singular coefficient is therefore the one of
the canonical integral
\begin{equation}
\int_{-\infty}^{\infty}\frac{d\xi}{(\xi^2+A)^2}
=\frac{\pi}{2A^{3/2}},
\qquad
A=z_0\tan x_0\,\eta,
\label{eq:complex-canonical-coalescence-integral}
\end{equation}
with the branch fixed by $A>0$ in the real case $y>1$.  Since the numerator
of $F_P$ is $\tan x_0+O(\xi)$, we obtain
\begin{equation}
F_P(y,z)
=\frac{\pi}{2y^2z_0^{3/2}\sqrt{\tan x_0}}
\left(1-\frac z{z_0}\right)^{-3/2}
+O_L\left(\left|1-\frac z{z_0}\right|^{-1}\right).
\label{eq:complex-FP-local-z}
\end{equation}
Now
\[
1-\frac{s}{\Psi(y)}
=1-\left(\frac z{z_0}\right)^2
=2\left(1-\frac z{z_0}\right)
+O\left(\left(1-\frac z{z_0}\right)^2\right),
\]
so the passage from $z$ to $s=z^2$ multiplies the leading coefficient by
$2^{3/2}$.  The critical equations also give
\[
\tan^2x_0=\frac1{y-1},
\qquad
z_0^2=\Psi(y),
\]
which yields the first formula in
\eqref{eq:complex-moving-C-amplitudes}.

For $\mathcal Q$, the prefactor $z$ in
\eqref{eq:complex-s-z-relation} contributes $z_0$, while the local
numerator changes from $\tan x_0$ to $\sec x_0$.  Hence
\[
\frac{C_Q(y)}{C_P(y)}
=\frac{z_0\sec x_0}{\tan x_0}
=\frac{z_0}{\sin x_0}
=\sqrt{y\Psi(y)},
\]
with compatible square roots, proving the second formula.

Differentiating \eqref{eq:complex-Psi-Phi} on the same branch gives
\[
\Phi'(y)=-\frac{\sqrt{y-1}}{y^2},
\qquad
\Psi'(y)=-\frac{2\sqrt{\Psi(y)}\sqrt{y-1}}{y^2}.
\]
Squaring \eqref{eq:complex-moving-C-amplitudes} therefore gives
\eqref{eq:complex-moving-C-squares}.  Since $\Psi$ is conformal on $E$,
$\Psi'(y)\neq0$, and $y\Psi(y)\neq0$ at every finite point of $E$; hence the
amplitudes are nonzero.  Finally, the uniform $\Delta$-domains from
\cref{prop:complex-two-singularity-regimes}, together with compactness of
the finite cover and the analytic Morse reduction, give the stated uniform
remainder.
\end{proof}

We now transfer these local expansions to the polynomial coefficients.  We
use the logarithmic and algebraic transfer theorems of singularity analysis;
see \cite[Chapter~VI]{FlajoletSedgewick2009}.  The uniform dented domains in
\cref{prop:complex-two-singularity-regimes} are precisely the hypothesis
needed for uniform transfer on compact parameter sets.

\begin{proposition}[Uniform two-phase asymptotics]
\label[proposition]{prop:complex-uniform-two-phase-asymptotics}
Define the monic polynomials
\begin{equation}
\widehat P_n(y):=\frac{P_n(y)}{(n!)^2},
\qquad
\widehat Q_n(y):=\frac{Q_n(y)}{n!(n-1)!}.
\label{eq:complex-monic-normalization}
\end{equation}
If $L\Subset K^\circ$, then uniformly for $y\in L$,
\begin{equation}
\begin{aligned}
\widehat P_n(y)
&=\frac{\mathfrak c^n}{\sqrt\pi\,(1-y)^2n^{3/2}}
\bigl(1+O_L(n^{-1})\bigr),\\
\widehat Q_n(y)
&=\frac{\sqrt\pi\,\mathfrak c^n}{4(1-y)^2n^{3/2}}
\bigl(1+O_L(n^{-1})\bigr).
\end{aligned}
\label{eq:complex-interior-asymptotics}
\end{equation}
If $L\Subset E$, choose a finite simply connected cover and compatible
branches as in \cref{lem:complex-moving-expansion}.  Then, uniformly on the
members of that finite cover and hence uniformly on $L$,
\begin{equation}
\begin{aligned}
\widehat P_n(y)
&=A_P(y)\left(\frac4{\Psi(y)}\right)^n
\bigl(1+O_L(n^{-1/2})\bigr),\\
\widehat Q_n(y)
&=A_Q(y)\left(\frac4{\Psi(y)}\right)^n
\bigl(1+O_L(n^{-1/2})\bigr).
\end{aligned}
\label{eq:complex-exterior-asymptotics}
\end{equation}
where, on each patch,
\begin{equation}
A_P(y)=\frac{2\sqrt2\,(y-1)^{1/4}}{y^2\Psi(y)^{3/4}},
\qquad
A_Q(y)=\frac{\sqrt2\,(y-1)^{1/4}}{y^{3/2}\Psi(y)^{1/4}}.
\label{eq:complex-A-amplitudes}
\end{equation}
Equivalently,
\begin{equation}
A_P(y)^2=-4\frac{\Psi'(y)}{y^2\Psi(y)^2},
\qquad
A_Q(y)^2=-\frac{\Psi'(y)}{y\Psi(y)}.
\label{eq:complex-A-squares}
\end{equation}
In particular, every leading amplitude occurring above is nonzero.
\end{proposition}

\begin{proof}
From \cref{lem:complex-fixed-expansion} and
\[
[s^n]\left(-\log\left(1-\frac sR\right)\right)=\frac{R^{-n}}n,
\]
uniform logarithmic transfer gives
\begin{equation}
\frac{P_n(y)}{(2n)!}
=\frac{R^{-n}}{(1-y)^2n}\bigl(1+O_L(n^{-1})\bigr),
\qquad
\frac{Q_n(y)}{(2n-1)!}
=\frac{\rho R^{-n}}{(1-y)^2n}\bigl(1+O_L(n^{-1})\bigr).
\label{eq:complex-fixed-coefficient-asymptotics}
\end{equation}
Stirling's formula in the forms
\begin{equation}
\frac{(2n)!}{(n!)^2}
=\frac{4^n}{\sqrt{\pi n}}\bigl(1+O(n^{-1})\bigr),
\qquad
\frac{(2n-1)!}{n!(n-1)!}
=\frac{4^n}{2\sqrt{\pi n}}\bigl(1+O(n^{-1})\bigr)
\label{eq:complex-central-binomial-asymptotics}
\end{equation}
and $4/R=\mathfrak c$, $\rho=\sqrt R=\pi/2$, yield
\eqref{eq:complex-interior-asymptotics}.

For the moving singularity,
\begin{equation}
[s^n]\left(1-\frac{s}{\Psi}\right)^{-3/2}
=\frac{2}{\sqrt\pi}\,n^{1/2}\Psi^{-n}
\bigl(1+O(n^{-1})\bigr).
\label{eq:complex-transfer-three-halves}
\end{equation}
The remainder $O((1-s/\Psi)^{-1})$ in
\eqref{eq:complex-moving-local-expansion} contributes
$O_L(|\Psi|^{-n})$, which is smaller than the leading coefficient by the
relative factor $O_L(n^{-1/2})$.  Combining
\eqref{eq:complex-transfer-three-halves} with
\eqref{eq:complex-central-binomial-asymptotics} gives
\[
A_P=\frac{2}{\pi}C_P,
\qquad
A_Q=\frac1\pi C_Q,
\]
and hence \eqref{eq:complex-exterior-asymptotics}--
\eqref{eq:complex-A-amplitudes}.  The square identities
\eqref{eq:complex-A-squares} follow from
\eqref{eq:complex-moving-C-squares}; their right-hand sides are nonzero
because $\Psi$ is conformal on $E$.  This also proves the asserted
nonvanishing of the leading terms.
\end{proof}

The exterior estimate must also control the point at infinity; otherwise a
sequence of zeros could in principle escape every compact subset of the
plane.  The following rescaling removes this possibility.

\begin{lemma}[Uniform control at infinity]
\label[lemma]{lem:complex-infinity-control}
There exists $\varepsilon_0>0$ such that, for $|u|\leq\varepsilon_0$,
\begin{equation}
\begin{aligned}
u^{n-1}\widehat P_n(1/u)
&=a_P(u)h(u)^n\left(1+O(n^{-1/2})\right),\\
u^{n-1}\widehat Q_n(1/u)
&=a_Q(u)h(u)^n\left(1+O(n^{-1/2})\right),
\end{aligned}
\label{eq:complex-reversed-asymptotics}
\end{equation}
uniformly in $u$, where $a_P,a_Q,h$ are holomorphic and nonzero on
$|u|\leq\varepsilon_0$ and
\begin{equation}
a_P(0)=a_Q(0)=h(0)=1.
\label{eq:complex-reversed-values-zero}
\end{equation}
Consequently, there is $M>0$ such that, for all sufficiently large $n$,
\begin{equation}
P_n(y)Q_n(y)\neq0
\qquad (|y|\geq M).
\label{eq:complex-uniform-zero-bound}
\end{equation}
\end{lemma}

\begin{proof}
Put $y=1/u$, $z=\sqrt u\,\xi$ and $x=\sqrt u\,\eta$.  The functions
\begin{equation}
A_u(\eta):=\frac{\tan(\sqrt u\,\eta)}{\sqrt u},
\qquad
B_u(\eta):=\sec(\sqrt u\,\eta)
\label{eq:complex-scaled-A-B}
\end{equation}
are jointly holomorphic in $(u,\eta)$ near $u=0$, because their power-series
expansions contain only integral powers of $u$; their removable values are
$A_0(\eta)=\eta$ and $B_0(\eta)=1$.  After the substitution in
\cref{prop:PQ-generating-functions},
\begin{equation}
\begin{aligned}
\sum_{n\geq1}u^{n-1}P_n(1/u)\frac{\xi^{2n}}{(2n)!}
&=\int_0^\xi
\frac{A_u(\eta)}{\bigl(1-(\xi-\eta)A_u(\eta)\bigr)^2}\,d\eta,\\
\sum_{n\geq1}u^{n-1}Q_n(1/u)\frac{\xi^{2n-1}}{(2n-1)!}
&=\int_0^\xi
\frac{B_u(\eta)}{\bigl(1-(\xi-\eta)A_u(\eta)\bigr)^2}\,d\eta.
\end{aligned}
\label{eq:complex-scaled-generating-functions}
\end{equation}
At $u=0$ the denominator is
\[
1-(\xi-\eta)\eta=1-\xi\eta+\eta^2,
\]
and the two critical pairs are
\begin{equation}
(\eta,\xi)=(1,2),
\qquad
(\eta,\xi)=(-1,-2),
\label{eq:complex-scaled-critical-points}
\end{equation}
with nonzero quadratic second derivative.

There are no hidden critical points remaining at bounded $\xi$ when
$u\to0$.  Indeed, for $u\neq0$ the critical equations are exactly the
rescaled equations from \cref{lem:complex-critical-parametrization} and
therefore
\begin{equation}
\sin^2(\sqrt u\,\eta)=u.
\label{eq:complex-scaled-critical-equation}
\end{equation}
On any local choice of $\sqrt u$, put $\alpha(u)=\arcsin\sqrt u$.  The
principal critical values are
\begin{equation}
\eta_\pm(u)=\pm\frac{\alpha(u)}{\sqrt u},
\qquad
\xi_\pm(u)=\pm\kappa(u),
\qquad
\kappa(u):=\frac{f(\alpha(u))}{\sqrt u}=2+O(u),
\label{eq:complex-scaled-principal-critical}
\end{equation}
where both $\alpha(u)/\sqrt u$ and $\kappa(u)$ are single-valued
holomorphic functions of $u$ near $0$.  All other critical pairs have
\begin{equation}
\eta_{k,\pm}(u)=\frac{k\pi\pm\alpha(u)}{\sqrt u},
\qquad
\xi_{k,\pm}(u)=\frac{k\pi\pm f(\alpha(u))}{\sqrt u},
\qquad k\neq0,
\label{eq:complex-scaled-nonprincipal-critical}
\end{equation}
so their $\xi$-coordinates tend to infinity uniformly as $u\to0$.  The
rescaled poles of $\tan$ and $\sec$ are likewise located at
\[
\eta=\frac{\pi/2+k\pi}{\sqrt u},
\]
and hence escape every bounded $\eta$-set.

It follows that, after shrinking $|u|\leq\varepsilon_0$, the two quadratic
coalescences in \eqref{eq:complex-scaled-principal-critical} are the only
obstructions in a fixed bounded $\xi$-neighborhood of the dominant circle.
The analytic implicit-function and Morse lemmas therefore give a uniform
quadratic normal form in the parameter $u$, including $u=0$.  Applying the
same uniform algebraic transfer as in
\cref{prop:complex-uniform-two-phase-asymptotics} to
\eqref{eq:complex-scaled-generating-functions} yields
\eqref{eq:complex-reversed-asymptotics}, with holomorphic functions
$a_P,a_Q,h$ near $u=0$.  At $u=0$ the model in
\eqref{eq:complex-scaled-critical-points} gives
\[
a_P(0)=a_Q(0)=h(0)=1.
\]

For $u\neq0$ the scaled singularity satisfies
$\Psi(1/u)=u\kappa(u)^2$, and comparison with
\eqref{eq:complex-exterior-asymptotics} gives
\begin{equation}
h(u)=\frac{4u}{\Psi(1/u)},
\qquad
a_P(u)=\frac{A_P(1/u)}u,
\qquad
a_Q(u)=\frac{A_Q(1/u)}u.
\label{eq:complex-reversed-functions}
\end{equation}
Equivalently, direct expansion at infinity gives
\begin{equation}
\Psi(1/u)=4u+O(u^2),
\qquad
A_P(1/u)=u+O(u^2),
\qquad
A_Q(1/u)=u+O(u^2),
\label{eq:complex-reversed-removable-expansions}
\end{equation}
confirming that the apparent singularities in
\eqref{eq:complex-reversed-functions} are removable.  Shrinking
$\varepsilon_0$ once more, $a_P,a_Q,h$ are bounded away from zero on the
closed disc.  The uniform relative error in
\eqref{eq:complex-reversed-asymptotics} is then $<1/2$ for all sufficiently
large $n$, so neither reversed polynomial vanishes there.  Taking
$M=\varepsilon_0^{-1}$ proves \eqref{eq:complex-uniform-zero-bound}.
\end{proof}

\subsection{The common zero-attractor theorem}

The conformal coordinate naturally associated with the exterior region is
\begin{equation}
\varphi(y):=\frac{R}{\Psi(y)},
\qquad y\in E.
\label{eq:complex-exterior-Riemann-map}
\end{equation}
It maps $E$ conformally onto $\{w:|w|>1\}$, with
$\varphi(\infty)=\infty$, and satisfies
\begin{equation}
|\varphi(y)|>1\quad(y\in E),
\qquad
|\varphi(y)|=1\quad(y\in\Gamma).
\label{eq:complex-phi-modulus}
\end{equation}

For a polynomial $p$, let $Z(p)$ denote its zero set without multiplicity,
and let $d_H$ denote Hausdorff distance between nonempty compact subsets of
$\mathbb C$.  The preceding asymptotics now give the main result.

\begin{theorem}[Common complex zero attractor]
\label[theorem]{thm:complex-common-zero-attractor}
The complex zero sets of $P_n$ and $Q_n$ converge to the same Jordan curve
$\Gamma$ defined in \eqref{eq:complex-K-Gamma-def}.  More precisely,
\begin{equation}
\boxed{
\lim_{n\to\infty}d_H\bigl(Z(P_n),\Gamma\bigr)=0,
\qquad
\lim_{n\to\infty}d_H\bigl(Z(Q_n),\Gamma\bigr)=0.
}
\label{eq:complex-Hausdorff-main}
\end{equation}
Equivalently,
\begin{equation}
\liminf_{n\to\infty}Z(P_n)
=
\limsup_{n\to\infty}Z(P_n)
=
\Gamma,
\qquad
\liminf_{n\to\infty}Z(Q_n)
=
\limsup_{n\to\infty}Z(Q_n)
=
\Gamma.
\label{eq:complex-Kuratowski-main}
\end{equation}
\end{theorem}

\begin{proof}
For $n\geq2$, both polynomials have degree $n-1$, so their zero sets are
nonempty.  We prove the two Hausdorff inclusions separately.

\emph{Step 1: zeros cannot remain away from $\Gamma$.}
Fix $\varepsilon>0$ and set
\[
A_\varepsilon:=\{y:\operatorname{dist}(y,\Gamma)\geq\varepsilon\}.
\]
The set $A_\varepsilon\cap K$ is a compact subset of $K^\circ$, so
\eqref{eq:complex-interior-asymptotics} excludes zeros there for all large
$n$.  Choose $M$ as in \cref{lem:complex-infinity-control}.  Then
\[
A_\varepsilon\cap E\cap\{|y|\leq M\}
\]
is a compact subset of $E$, and a finite simply connected cover together
with \eqref{eq:complex-exterior-asymptotics} excludes zeros on this bounded
exterior part.  Finally, \eqref{eq:complex-uniform-zero-bound} excludes
zeros for $|y|\geq M$.  Hence
\begin{equation}
\max_{\zeta\in Z(P_n)}\operatorname{dist}(\zeta,\Gamma)\longrightarrow0,
\qquad
\max_{\zeta\in Z(Q_n)}\operatorname{dist}(\zeta,\Gamma)\longrightarrow0.
\label{eq:complex-outer-Hausdorff}
\end{equation}

\emph{Step 2: every point of $\Gamma$ is approached uniformly by zeros.}
Since the polynomials in \eqref{eq:complex-monic-normalization} are monic of
degree $n-1$, define, with multiplicity,
\begin{equation}
\nu_n^P:=\frac1{n-1}\sum_{P_n(\zeta)=0}\delta_\zeta,
\qquad
\nu_n^Q:=\frac1{n-1}\sum_{Q_n(\zeta)=0}\delta_\zeta.
\label{eq:complex-zero-counting-measures}
\end{equation}
The uniform bound \eqref{eq:complex-uniform-zero-bound} makes both sequences
of probability measures tight.  Take any weakly convergent subsequence,
denoted generically by $\nu_n\Rightarrow\nu$.  By
\eqref{eq:complex-outer-Hausdorff},
\[
\operatorname{supp}\nu\subseteq\Gamma.
\]
For $y\notin\Gamma$, the eventual zero sets stay a positive distance from
$y$, so the logarithmic kernel is uniformly continuous there.  For the
$P$-sequence we have the exact identity
\begin{equation}
\int\log|y-\zeta|\,d\nu_n^P(\zeta)
=\frac1{n-1}\log|\widehat P_n(y)|,
\label{eq:complex-potential-identity-P}
\end{equation}
and analogously for $Q$.  Weak convergence and the two-phase asymptotics
therefore give, for either family,
\begin{equation}
\int\log|y-\zeta|\,d\nu(\zeta)=U(y),
\qquad
U(y):=
\begin{cases}
\log\mathfrak c,&y\in K^\circ,\\[1ex]
\displaystyle\log\left|\frac4{\Psi(y)}\right|,&y\in E.
\end{cases}
\label{eq:complex-limit-potential}
\end{equation}

This determines the subsequential limit uniquely.  Indeed, $\Gamma$ is
real analytic except for the single algebraic corner at $1$, hence has
planar Lebesgue measure zero.  Logarithmic potentials are locally
integrable and satisfy
\[
\Delta\log|y-\zeta|=2\pi\delta_\zeta
\]
distributionally; see \cite[Chapters~2--3]{Ransford1995}.  If two
subsequential limits $\nu$ and $\widetilde\nu$ existed, their potentials
would agree on $\mathbb C\setminus\Gamma$ by
\eqref{eq:complex-limit-potential}, hence almost everywhere in
$\mathbb C$.  They therefore agree as distributions, and applying
$\Delta/(2\pi)$ gives $\nu=\widetilde\nu$.  Thus both full sequences
converge weakly to the same probability measure
\begin{equation}
\mu_\Gamma:=\frac1{2\pi}\Delta U.
\label{eq:complex-mu-Gamma-Laplacian}
\end{equation}

On every regular boundary arc, $U$ is constant on the $K$-side and equals
$\log\mathfrak c+\log|\varphi|$ on the $E$-side.  Hence the jump of the
normal derivative gives
\[
d\mu_\Gamma(y)
=\frac1{2\pi}\,\partial_{\mathbf n}\log|\varphi(y)|\,|dy|
=\frac{|\varphi'(y)|}{2\pi}|dy|>0.
\]
The regular arcs accumulate at the sole corner $y=1$, and therefore
\begin{equation}
\operatorname{supp}\mu_\Gamma=\Gamma.
\label{eq:complex-mu-support}
\end{equation}

It remains to turn this support statement into the uniform reverse
Hausdorff inclusion.  Fix $\varepsilon>0$.  Compactness of $\Gamma$ gives
points $y_1,\ldots,y_m\in\Gamma$ such that
\[
\Gamma\subseteq\bigcup_{j=1}^mB(y_j,\varepsilon/2).
\]
By \eqref{eq:complex-mu-support}, each of these open balls has positive
$\mu_\Gamma$-mass.  The Portmanteau theorem
\cite{Billingsley-convergence} gives
\[
\liminf_{n\to\infty}\nu_n^P\bigl(B(y_j,\varepsilon/2)\bigr)>0,
\qquad
\liminf_{n\to\infty}\nu_n^Q\bigl(B(y_j,\varepsilon/2)\bigr)>0
\]
for every $j$.  Since the cover is finite, for all sufficiently large $n$
each ball contains at least one zero of each polynomial.  Hence
\begin{equation}
\max_{y\in\Gamma}\operatorname{dist}\bigl(y,Z(P_n)\bigr)\longrightarrow0,
\qquad
\max_{y\in\Gamma}\operatorname{dist}\bigl(y,Z(Q_n)\bigr)\longrightarrow0.
\label{eq:complex-inner-Hausdorff}
\end{equation}
Combining \eqref{eq:complex-outer-Hausdorff} and
\eqref{eq:complex-inner-Hausdorff} proves
\eqref{eq:complex-Hausdorff-main}; the Kuratowski formulation
\eqref{eq:complex-Kuratowski-main} follows immediately.
\end{proof}

\begin{corollary}[Common limiting zero distribution]
\label[corollary]{cor:complex-common-zero-distribution}
With the notation of \eqref{eq:complex-zero-counting-measures},
\begin{equation}
\boxed{
\nu_n^P\xrightarrow{\mathrm w}\mu_\Gamma,
\qquad
\nu_n^Q\xrightarrow{\mathrm w}\mu_\Gamma,
}
\label{eq:complex-weak-zero-distribution}
\end{equation}
where $\mu_\Gamma=(2\pi)^{-1}\Delta U$ is the probability measure in
\eqref{eq:complex-mu-Gamma-Laplacian} and
$\operatorname{supp}\mu_\Gamma=\Gamma$.
\end{corollary}

\begin{proof}
This is the weak-convergence statement established in Step~2 of the proof
of \cref{thm:complex-common-zero-attractor}.
\end{proof}

\clearpage
\subsection{Numerical experiments}
\label{subsec:complex-numerical-experiments}

We briefly illustrate \cref{thm:complex-common-zero-attractor} numerically.
For $n=20,50,100,150$, the coefficients of $P_n$ were generated using the exact
integer recurrence \eqref{eq:unified-coefficient-recurrence}, and the
$n-1$ complex zeros were then computed with high-precision arithmetic.  The
limiting curve $\Gamma$ was sampled from the critical-point parametrization
\eqref{eq:complex-f-def}: on the relevant boundary component one has
$|f(x)|=\rho$ and $y=\csc^2x$, in agreement with
\eqref{eq:complex-Gamma-equimodular}.

% The H specifier prevents the figure from floating into the next subsection.
% The reduced panel widths allow the four plots to be arranged compactly.
\begin{figure}[H]
\centering
\begin{minipage}[t]{0.44\textwidth}
\centering
\includegraphics[width=\linewidth]{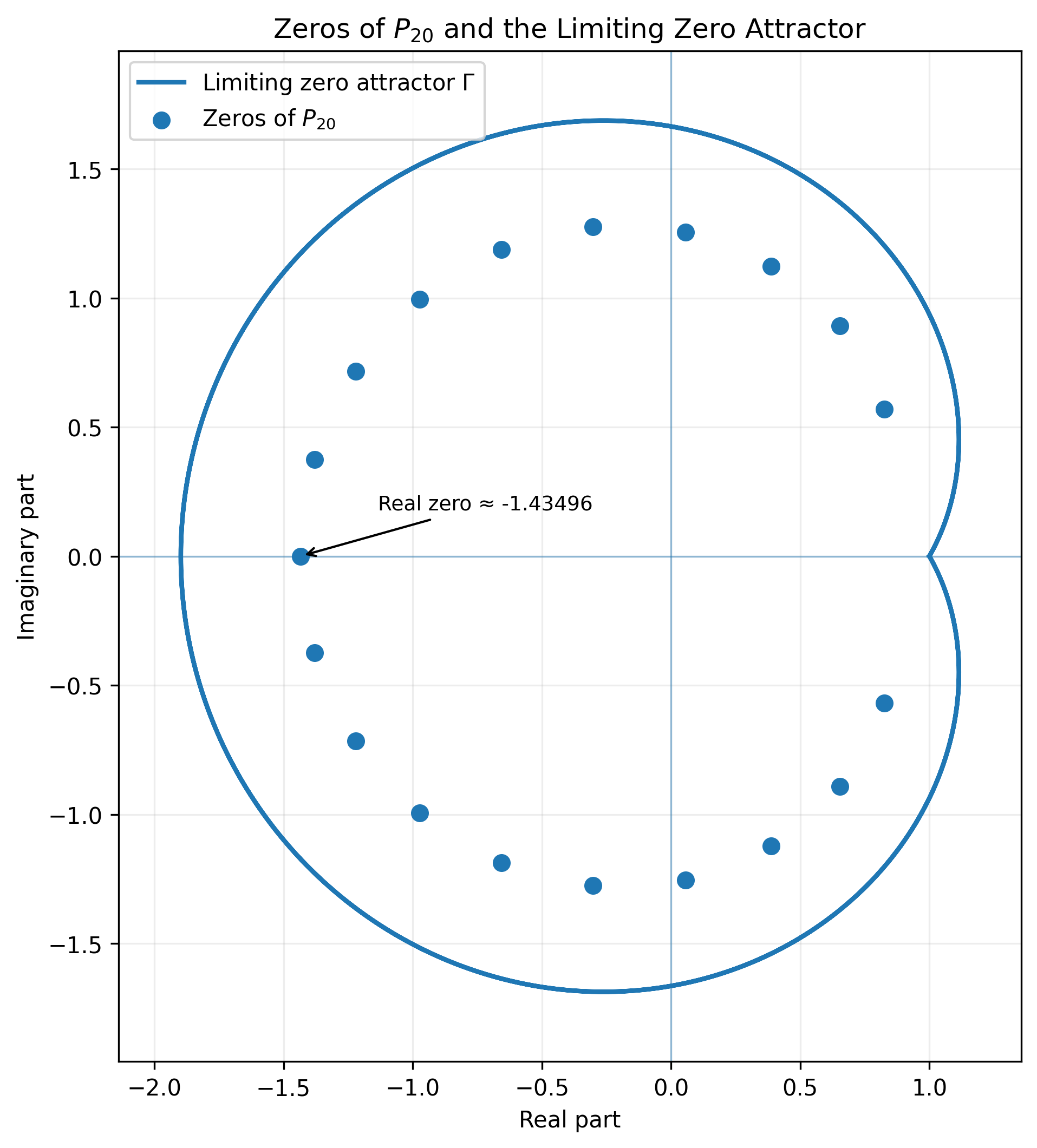}
\smallskip
{\small (a) $n=20$.}
\end{minipage}\hfill
\begin{minipage}[t]{0.44\textwidth}
\centering
\includegraphics[width=\linewidth]{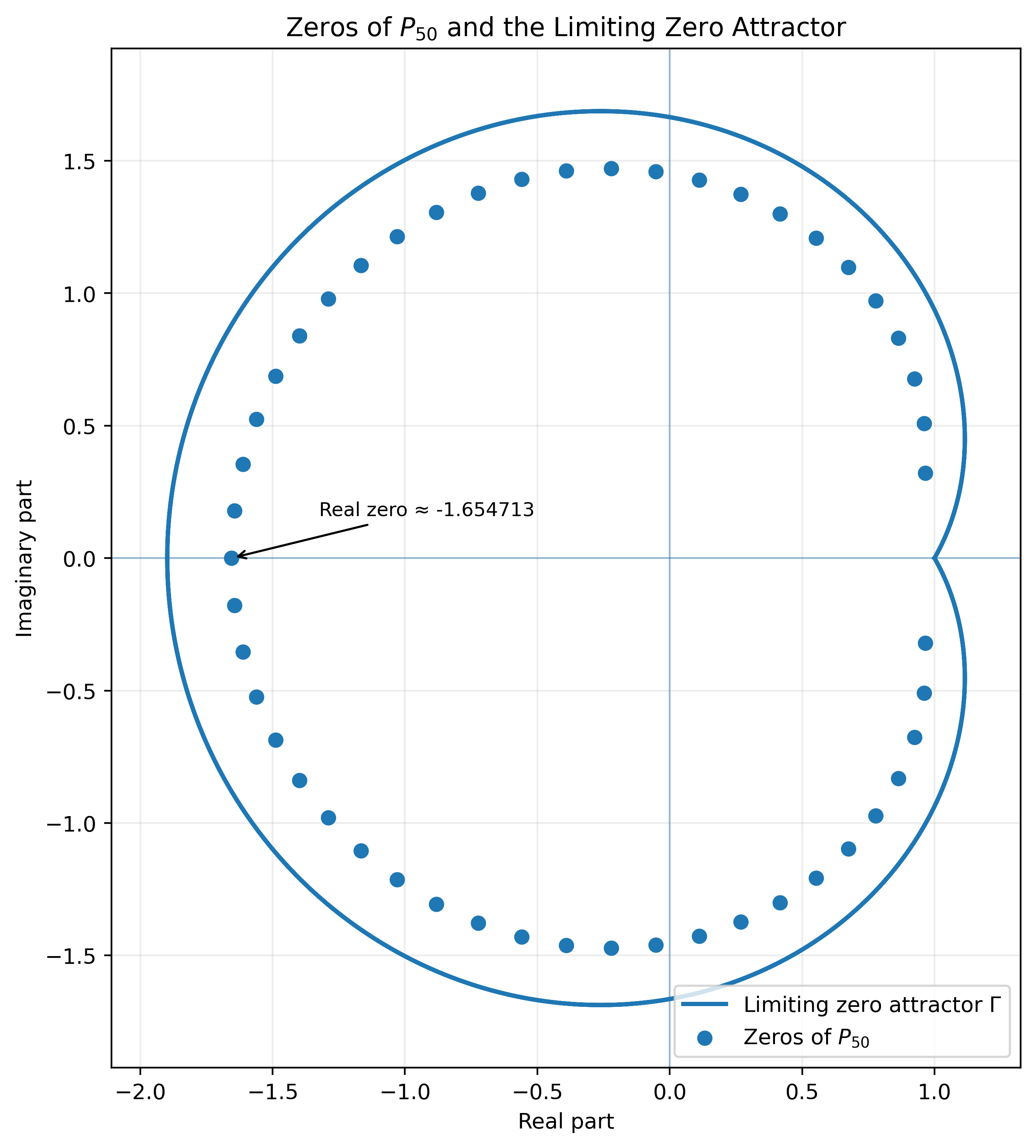}
\smallskip
{\small (b) $n=50$.}
\end{minipage}

\medskip
\begin{minipage}[t]{0.44\textwidth}
\centering
\includegraphics[width=\linewidth]{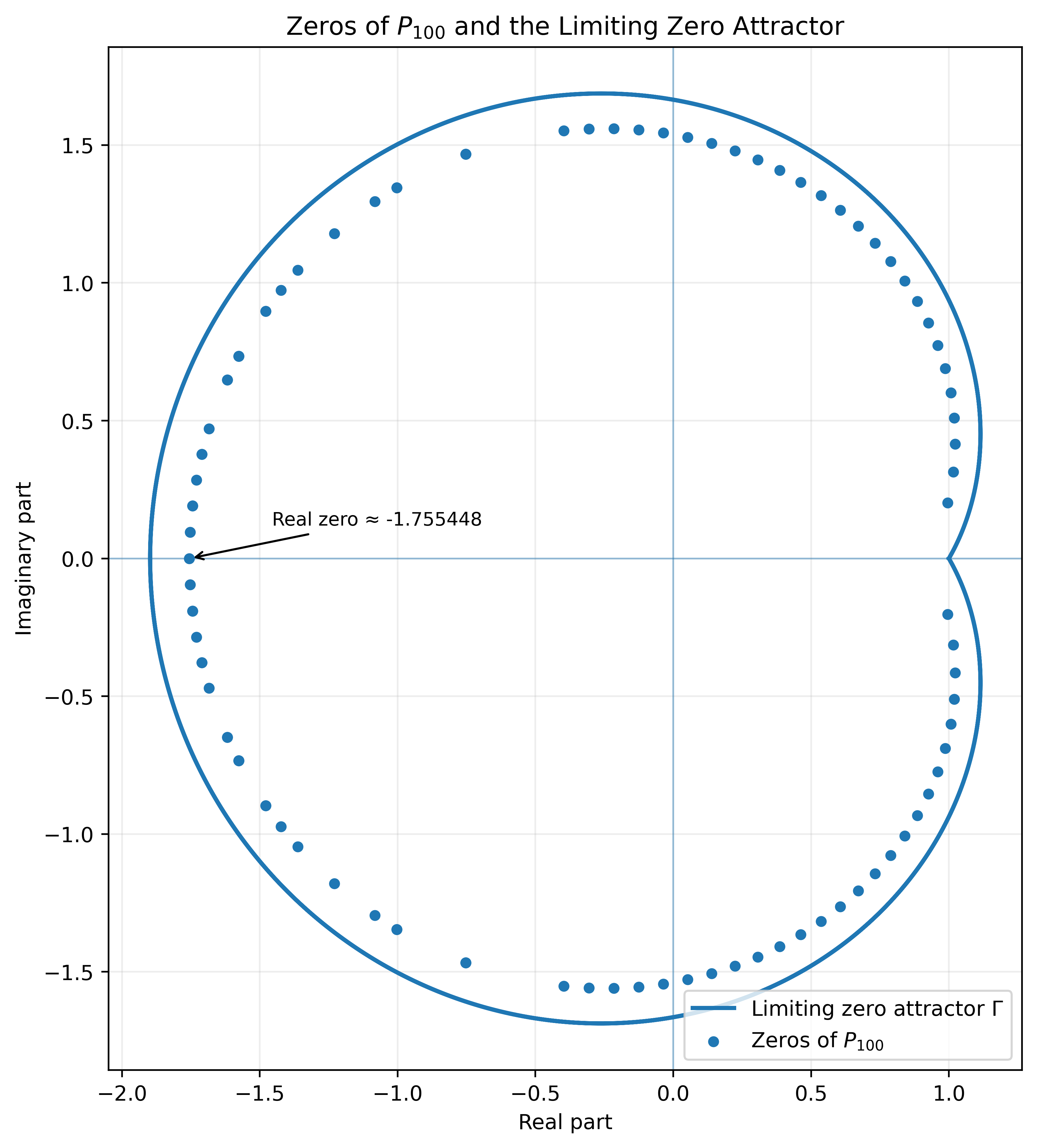}
\smallskip
{\small (c) $n=100$.}
\end{minipage}\hfill
\begin{minipage}[t]{0.44\textwidth}
\centering
\includegraphics[width=\linewidth]{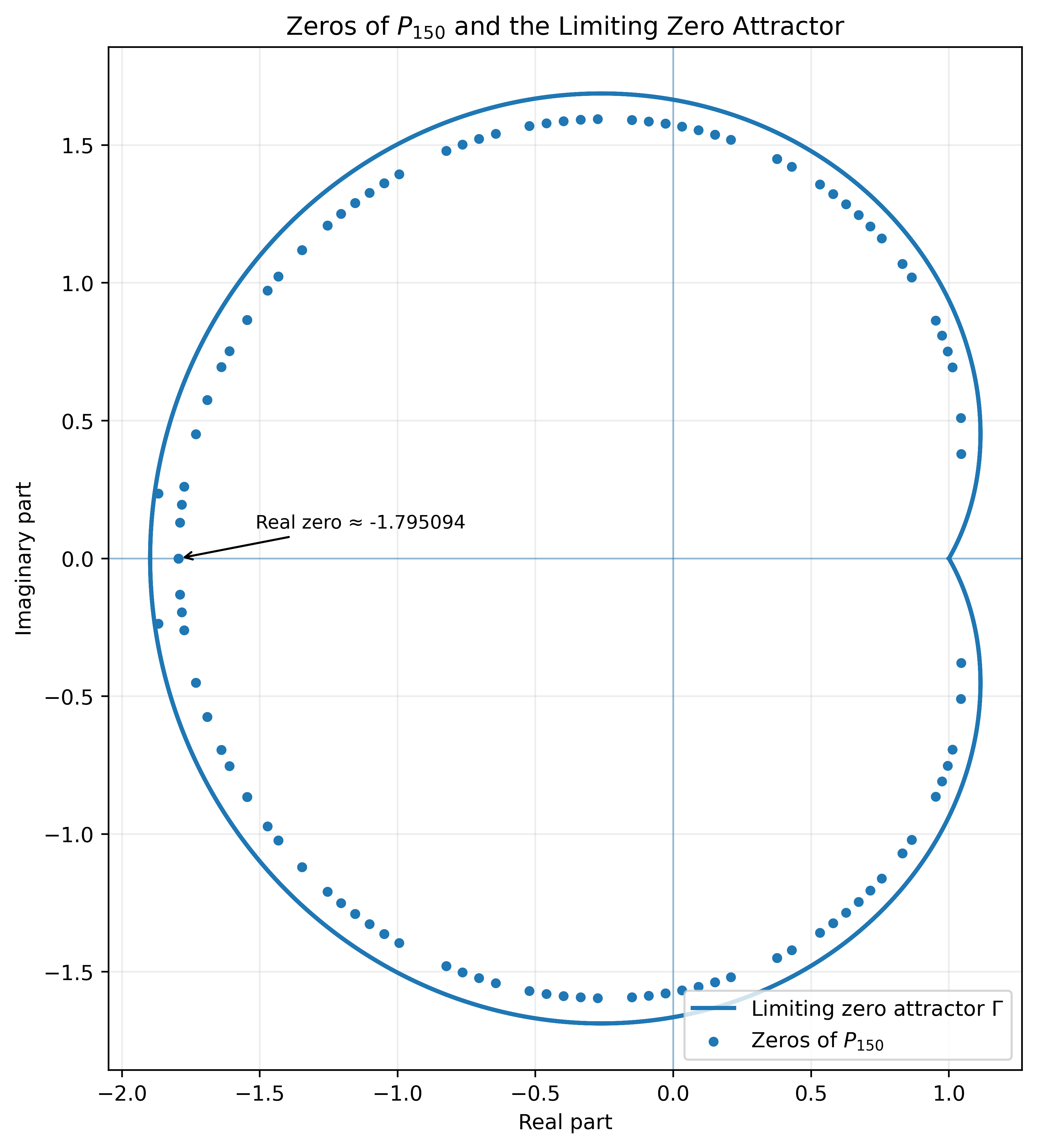}
\smallskip
{\small (d) $n=150$.}
\end{minipage}

\caption{Numerical zeros of $P_n$ for $n=20,50,100,150$ (dots), together with
the limiting zero-attractor curve $\Gamma$ (solid curve).  As $n$ increases,
the zeros visibly approach $\Gamma$, illustrating the Hausdorff convergence
in \cref{thm:complex-common-zero-attractor}.}
\label{fig:Pn-zero-attractor-numerics}
\end{figure}

The concentration toward $\Gamma$ is already apparent for $n=20$ and becomes
substantially sharper for $n=50$, $n=100$, and $n=150$; see
\cref{fig:Pn-zero-attractor-numerics}.  The unique real zeros in these four
even-index cases are approximately
\[
-1.4349642845,\qquad -1.6547134750,\qquad -1.7554484154,\qquad -1.7950941668,
\]
respectively.  Their motion toward the limiting value
$-c_*=-1.8976700911\ldots$ is also consistent with
\cref{prop:PQ-real-zero-limit}.  These computations are intended only as
a numerical illustration; the convergence of the full complex zero sets is
the content of \cref{thm:complex-common-zero-attractor}.

The preceding result gives more than convergence of sets: it identifies the
limiting zero distribution.  We now make this measure explicit.

\subsection{Capacity, limiting zero density, and further consequences}

The exterior map \eqref{eq:complex-exterior-Riemann-map} allows the
limit potential to be written
\begin{equation}
U(y)
=
\log\mathfrak c
+
\begin{cases}
0,&y\in K,\\
\log|\varphi(y)|,&y\in E.
\end{cases}
\label{eq:complex-Green-potential}
\end{equation}

\begin{proposition}[Capacity and limiting zero density]
\label[proposition]{prop:complex-density-capacity}
The compact set $K$ has logarithmic capacity
\begin{equation}
\boxed{
\operatorname{cap}(K)=\frac{16}{\pi^2}.
}
\label{eq:complex-capacity}
\end{equation}
The common weak limit in \eqref{eq:complex-mu-Gamma-Laplacian} is the
equilibrium measure of $K$, equivalently the harmonic measure of $E$ seen
from infinity.  On the regular part of $\Gamma$ it has arclength density
\begin{equation}
\boxed{
\frac{d\mu_\Gamma}{|dy|}
=
\frac{|\varphi'(y)|}{2\pi}
=
\frac{|\Psi'(y)|}{2\pi R}
=
\frac{2}{\pi^2}
\frac{|y-1|^{1/2}}{|y|^2}.
}
\label{eq:complex-zero-density}
\end{equation}
The right-hand side tends to $0$ as $y\to1$ along $\Gamma$; hence
$\operatorname{supp}\mu_\Gamma=\Gamma$.
\end{proposition}

\begin{proof}
From \eqref{eq:complex-Psi-Phi}, with $u=1/y$,
\begin{equation}
\begin{aligned}
\Phi(y)
&=2u^{1/2}
-\frac13u^{3/2}
-\frac1{20}u^{5/2}
-\frac1{56}u^{7/2}
+O(u^{9/2}),\\
\Psi(y)
&=4u
-\frac43u^2
-\frac4{45}u^3
-\frac4{105}u^4
+O(u^5).
\end{aligned}
\label{eq:complex-Psi-infinity}
\end{equation}
Therefore
\begin{equation}
\frac4{\Psi(y)}
=y+\frac13+\frac{2}{15y}
+\frac{58}{945y^2}
+O(y^{-3}),
\label{eq:complex-four-over-Psi}
\end{equation}
and hence
\begin{equation}
\varphi(y)
=\frac{\pi^2}{16}y+O(1)
\qquad (y\to\infty).
\label{eq:complex-phi-infinity}
\end{equation}
For the normalized exterior Riemann map, the coefficient of $y$ is the
reciprocal logarithmic capacity; see, for example,
\cite[Chapter~5]{Ransford1995}.  This proves \eqref{eq:complex-capacity}.
Equation \eqref{eq:complex-Green-potential} then identifies
$\mu_\Gamma=(2\pi)^{-1}\Delta U$ with the equilibrium measure.

On a smooth boundary arc, harmonic measure at infinity is the pullback of
normalized angular measure under $\varphi$, so
\[
d\mu_\Gamma(y)=\frac{|\varphi'(y)|}{2\pi}|dy|.
\]
Furthermore,
\begin{equation}
\Phi'(y)
=-\frac{\sqrt{y-1}}{y^2},
\qquad
\Psi'(y)=2\Phi(y)\Phi'(y).
\label{eq:complex-Phi-Psi-derivatives}
\end{equation}
Since $|\Phi(y)|=\rho=\pi/2$ on $\Gamma$,
\[
|\Psi'(y)|
=\pi\frac{|y-1|^{1/2}}{|y|^2}.
\]
Substitution into $|\varphi'|=|\Psi'|/R$ gives
\eqref{eq:complex-zero-density}.  This density is strictly positive at
every regular point and tends to $0$ only at the corner $y=1$; every
neighborhood of that point contains regular arcs of positive measure, so
the support is the whole curve.
\end{proof}

The limiting measure can also be described parametrically.  If
\begin{equation}
\Psi(y(\theta))
=R e^{i\theta},
\qquad 0\leq\theta<2\pi,
\label{eq:complex-Gamma-angular-param}
\end{equation}
then
\begin{equation}
d\mu_\Gamma(y(\theta))=\frac{d\theta}{2\pi}.
\label{eq:complex-uniform-angular-measure}
\end{equation}
Thus the zeros are asymptotically uniformly distributed after the conformal
coordinate $\varphi$ is applied.

\begin{corollary}[First limiting moments]
\label[corollary]{cor:complex-limiting-moments}
For either family, if the zeros are counted with multiplicity, then
\begin{equation}
\begin{aligned}
\lim_{n\to\infty}
\frac1{n-1}\sum_{P_n(\zeta)=0}\zeta
&=
\lim_{n\to\infty}
\frac1{n-1}\sum_{Q_n(\zeta)=0}\zeta
=-\frac13,\\
\lim_{n\to\infty}
\frac1{n-1}\sum_{P_n(\zeta)=0}\zeta^2
&=
\lim_{n\to\infty}
\frac1{n-1}\sum_{Q_n(\zeta)=0}\zeta^2
=-\frac7{45},\\
\lim_{n\to\infty}
\frac1{n-1}\sum_{P_n(\zeta)=0}\zeta^3
&=
\lim_{n\to\infty}
\frac1{n-1}\sum_{Q_n(\zeta)=0}\zeta^3
=-\frac{83}{945}.
\end{aligned}
\label{eq:complex-first-moments}
\end{equation}
\end{corollary}

\begin{proof}
By \cref{thm:complex-common-zero-attractor,prop:complex-density-capacity},
the normalized zero-counting measures converge weakly to $\mu_\Gamma$, and
\cref{lem:complex-infinity-control} gives a common compact set containing
all zeros for large $n$.  Hence every polynomial moment converges to the
corresponding moment of $\mu_\Gamma$.

It remains to identify those moments.  From
\eqref{eq:complex-four-over-Psi},
\begin{equation}
\log\left(\frac4{\Psi(y)}\right)
=\log y+\frac1{3y}+\frac7{90y^2}
+\frac{83}{2835y^3}+O(y^{-4}).
\label{eq:complex-log-potential-expansion}
\end{equation}
To justify comparison with complex moments, choose a simply connected
neighborhood of the positive real end of $E$ and compatible logarithm
branches there.  The function
\[
H(y):=\int_\Gamma\log(y-\zeta)\,d\mu_\Gamma(\zeta)
\]
is analytic in that neighborhood, and
$\Re H(y)=U(y)=\Re\log(4/\Psi(y))$.  Hence
$H(y)-\log(4/\Psi(y))$ is a purely imaginary constant.  Both functions
satisfy $H(y)-\log y\to0$ and
$\log(4/\Psi(y))-\log y\to0$ as $y\to+\infty$, so the constant is zero.
Therefore, as an analytic germ at infinity,
\begin{equation}
\log\left(\frac4{\Psi(y)}\right)
=\int_\Gamma\log(y-\zeta)\,d\mu_\Gamma(\zeta)
=\log y-
\sum_{k\geq1}\frac1{k y^k}
\int_\Gamma\zeta^k\,d\mu_\Gamma(\zeta).
\label{eq:complex-complex-potential-moments}
\end{equation}
Comparing \eqref{eq:complex-log-potential-expansion} and
\eqref{eq:complex-complex-potential-moments} yields
\[
\int_\Gamma\zeta\,d\mu_\Gamma=-\frac13,
\qquad
\int_\Gamma\zeta^2\,d\mu_\Gamma=-\frac7{45},
\qquad
\int_\Gamma\zeta^3\,d\mu_\Gamma=-\frac{83}{945},
\]
which proves \eqref{eq:complex-first-moments}.
\end{proof}

Finally, although the two families have the same exponential rate and the
same limiting zero distribution, their leading amplitudes are different.
The difference is especially transparent in the following ratio limit.

\begin{corollary}[Asymptotic ratio of the two families]
\label[corollary]{cor:complex-PQ-ratio}
With the monic normalizations in \eqref{eq:complex-monic-normalization},
locally uniformly away from $\Gamma$,
\begin{equation}
\frac{\widehat Q_n(y)}{\widehat P_n(y)}
\longrightarrow
\begin{cases}
\displaystyle \frac\pi4,
&y\in K^\circ,\\[2ex]
\displaystyle \frac12\sqrt{y\Psi(y)},
&y\in E,
\end{cases}
\label{eq:complex-PQ-ratio-limit}
\end{equation}
where the square root in the exterior region is chosen by analytic
continuation from $y>1$.
\end{corollary}

\begin{proof}
Divide the two lines of \eqref{eq:complex-interior-asymptotics} and of
\eqref{eq:complex-exterior-asymptotics}, respectively.  In the exterior
region,
\[
\frac{A_Q(y)}{A_P(y)}
=\frac12\sqrt{y\Psi(y)}.
\]
\end{proof}

The point $y=1$ is therefore the unique transition point at which the
moving quadratic singularity meets the fixed endpoint singularity.  The
$3/2$-power law in \eqref{eq:complex-Phi-near-one} explains both the
algebraic corner of $\Gamma$ and the vanishing of the limiting zero density
there.  A finer double-scaling analysis near this coalescence is a separate
problem and is not needed for the zero-attractor theorem proved above.

\section*{Acknowledgments}
The author acknowledges the use of an AI language model for language and
presentation assistance and for exploratory algebraic checks.  

\printbibliography

\end{document}